\documentclass[reqno, letterpaper, final]{amsart}
\PassOptionsToPackage{dvipsnames}{xcolor}
\usepackage{rbmwe}
\usepackage[parfill]{parskip}
\usepackage{enumitem}
\usepackage{dsfont}
\usepackage{mathrsfs}
\usetikzlibrary{decorations.pathmorphing}
\usepackage{tikzit}
\usepackage{tikzscale}
\tikzset{
  edge/.style={draw=black, line width=1.2pt},
  vBlack/.style={circle, fill=black, draw=none, minimum size=6pt, inner sep=0pt},
  vRed/.style={circle, fill=red, draw=none, minimum size=7pt, inner sep=0pt},
  vGreen/.style={circle, fill=green!60!black, draw=none, minimum size=7pt, inner sep=0pt},
  vBlue/.style={circle, fill=blue, draw=none, minimum size=7pt, inner sep=0pt},
  vTeal/.style={circle, fill=cyan!60!black, draw=none, minimum size=7pt, inner sep=0pt},
  vPink/.style={circle, fill=magenta!35, draw=none, minimum size=7pt, inner sep=0pt},
  vPurpleX/.style={draw=none, text=magenta, font=\Large},
  greenArrow/.style={->, draw=green!60!black, line width=1.2pt}
}
\input{pictures/tikzpreamble}

\providecommand{\LabelScaleExpr}{min(1.5, 0.3 + 5.5/(\Den+2))}
\newcommand{\SetLabelScaleExpr}[1]{\gdef\LabelScaleExpr{#1}} 
\providecommand{\InfLabelScaleExpr}{1.7}

\makeatletter
\newcommand{\fareyline}{%
  \@ifstar{\fareylineplain{2}}{\fareylineplain{0}}%
 }
\newcommand{\FareyParseFrac}[3]{
  \expandafter\FareyParseFrac@i#1//\relax{#2}{#3}%
}
\def\FareyParseFrac@i#1/#2/#3\relax#4#5{%
  \pgfmathtruncatemacro{#4}{#1}%
  \if\relax\detokenize{#2}\relax
    \pgfmathtruncatemacro{#5}{1}%
  \else
    \pgfmathtruncatemacro{#5}{#2}%
  \fi
}
\makeatother

  \newcommand{\SetFracColor}[3]{%
    \expandafter\def\csname frac@color@#1@#2\endcsname{#3}%
  }
  
  \def\ExtraFracList{} 
  \def\ColouredIntervalList{} 
  
  \newcommand{\MarkFrac}[2]{\expandafter\def\csname frac@seen@#1@#2\endcsname{1}}
  \newcommand{\DrawFrac}[2]{
    \pgfmathsetmacro{\x}{#1/#2}
    \pgfmathtruncatemacro{\out}{ (\x<\LeftVal) || (\x>\RightVal) }
    \ifnum\out=0
      \ifcsname frac@seen@#1@#2\endcsname
      \else
        \edef\Num{#1}\edef\Den{#2}          
        \pgfmathsetmacro{\s}{\LabelScaleExpr}

        \MarkFrac{#1}{#2}

        \pgfmathsetmacro{\h}{\TickBase/#2}

        \pgfmathtruncatemacro{\parity}{mod(#2, 2)}
        \ifnum\parity=0 \def\anch{south}\def\sgn{1}\else \def\anch{north}\def\sgn{-1}\fi
        \pgfmathsetmacro{\y}{\sgn*(\h+\LabelGap)}

        \def\col{black}
        \ifcsname frac@color@#1@#2\endcsname
          \edef\col{\csname frac@color@#1@#2\endcsname}
        \fi
        \draw[thick, draw=\col] (\x,-\h) -- (\x,\h);
       \node[frac label, anchor=\anch, text=\col] at (\x,\y) {\scalebox{\s}{$\frac{#1}{#2}$}};
      \fi
    \fi
  }

\newcommand{\fareylineplain}[5]{
\ifnum#1=-1\relax \showminusinftytrue \else \showminusinftyfalse \showinftyfalse \fi
\ifnum#1=1\relax \showplusinftytrue \else \showplusinftyfalse \showinftyfalse \fi
\ifnum#1=2\relax \showinftytrue \showminusinftytrue \showplusinftytrue\fi
  \def\myaxiswidth{#2}%
  \FareyParseFrac{#3}{\LeftNum}{\LeftDen}%
  \FareyParseFrac{#4}{\RightNum}{\RightDen}%
  \pgfmathsetmacro{\LeftVal}{\LeftNum/\LeftDen}%
  \pgfmathsetmacro{\RightVal}{\RightNum/\RightDen}%
  \pgfmathtruncatemacro{\DenMax}{#5}%
  \def\InfPadFactor{0.1}   

\noindent
\begin{tikzpicture}[axis width=\myaxiswidth, y=1cm, line cap=round]

  \def\TickBase{0.35}
  \def\LabelGap{0.12}
  \def\FracFont{\normalsize}
  \tikzset{frac label/.style={font=\FracFont, inner sep=1pt}}

  \draw[thick] (\LeftVal,0) -- (\RightVal,0);

 \pgfmathsetmacro{\yInfLab}{-(\TickBase+\LabelGap)}
 \pgfmathsetmacro{\s}{\InfLabelScaleExpr}
 \ifshowplusinfty
    \draw[blue, thick,->, dashed] (\RightVal,0) -- (\xRinf,0);
    \draw[thick] (\xRinf,-\TickBase) -- (\xRinf,\TickBase);
    \def\col{black}
    \ifcsname frac@color@1@0\endcsname
          \edef\col{\csname frac@color@1@0\endcsname}
     \fi
    \node[frac label, anchor=north, text=\col] at (\xRinf,\yInfLab) {\scalebox{\s}{$\frac{1}{0}$}};
  \fi
  \ifshowminusinfty
    \draw[blue, thick,<-, dashed] (\xLinf,0) -- (\LeftVal,0);
    \draw[thick] (\xLinf,-\TickBase) -- (\xLinf,\TickBase);
     \def\col{black}
    \ifcsname frac@color@1@0\endcsname
          \edef\col{\csname frac@color@1@0\endcsname}
     \fi
    \node[frac label, anchor=north] at (\xLinf,\yInfLab) {\scalebox{\s}{$\frac{-1}{0}$}};
  \fi

\foreach \d in {1,...,\DenMax} {

  \pgfmathsetmacro{\nminf}{ceil(\LeftVal*\d)}
  \pgfmathsetmacro{\nmaxf}{floor(\RightVal*\d)}
  \pgfmathtruncatemacro{\nmin}{\nminf}
  \pgfmathtruncatemacro{\nmax}{\nmaxf}

  \ifnum\nmin>\nmax\relax
  \else
    \foreach \n in {\nmin,...,\nmax} {
      \pgfmathgcd{\n}{\d}\edef\g{\pgfmathresult}
      \ifnum\g=1\relax
        \DrawFrac{\n}{\d}
      \fi
    }
  \fi
}

  \foreach \p/\q in \ExtraFracList {
    \pgfmathtruncatemacro{\P}{\p}
    \pgfmathtruncatemacro{\Q}{\q}
    \pgfmathgcd{\P}{\Q}\edef\g{\pgfmathresult}
    \pgfmathtruncatemacro{\Pr}{\P/\g}
    \pgfmathtruncatemacro{\Qr}{\Q/\g}
    \DrawFrac{\Pr}{\Qr}
  }
  \foreach \a/\b/\c/\d/\col in \ColouredIntervalList {
   \draw[draw=\col] (\a/\b,0) -- (\c/\d,0);
}

\end{tikzpicture}
}

\usetikzlibrary{arrows.meta,calc,decorations.pathmorphing}

\tikzset{
  axis/.style={black, line width=0.9pt, -{Stealth[length=3mm]}},
  tick/.style={black, line width=0.8pt},
  dot/.style={circle, fill=black, inner sep=1.2pt},
  bcurve/.style={blue!75!black, line width=1.2pt},
  gcurve/.style={green!60!black, line width=1.0pt},
  graydiag/.style={gray!70, line width=0.9pt},
  maparrow/.style={-{Stealth[length=3mm]}, line width=0.9pt},
}

\newcommand{\mypt}[2]{\fill (#1,#2) circle[radius=1.2pt];}

\newcommand{\ExampleAttractor}{%
\begin{tikzpicture}[x=6cm,y=6cm, font=\small]
  \draw[axis] (0,0) -- (2.15,0) node[below right] {};
  \draw[axis] (0,0) -- (0,1.15) node[above left] {};

  \draw[gcurve] (0,1) -- (1,1) -- (1, 0.5) -- (0.5, 0.5) -- (0.5, 0.75) -- (0.75,0.75) -- (0.75,0.625) -- (0.625,0.625) -- (0.625, 0.6875) -- (0.6875,0.6875);

  \draw[graydiag] (-0.15,-0.15) -- (1.25,1.25);

  \draw[bcurve] (-0.1,1.05) -- (2.1,-0.05);

  \foreach \x/\y in {0/1, 0.5/0.75, 0.625/0.6875, 0.6666667/0.6666667, 0.75/0.625, 1/0.5, 2/0, 1/1, 0.5/0.5, 0.75/0.75, 0.625/0.625, 0.6875/0.6875} {
    \mypt{\x}{\y}
  }

  \foreach \y/\lab in {
    1/{1},
    0.625/{\frac58},
    0.5/{\frac12}
  }{
    \draw[tick] (-0.015,\y) -- (0.015,\y);
    \node[left] at (-0.02,\y) {$\lab$};
  }
  \foreach \y/\lab in {
    0.75/{\frac34},
    0.6666667/{\frac23},
  }{
    \draw[tick] (-0.015,\y) -- (0.015,\y);
    \node[right] at (0.02,\y) {$\lab$};
  }
  \foreach \y/\lab in {
    0.6875/{\frac{11}{16}},
  }{
    \draw[tick] (-0.015,\y) -- (0.015,\y);
    \node[left] at (-0.02,\y+0.01) {\scriptsize$\lab$};
  }

  \foreach \x/\lab in {
    0.25/{\frac14},
    0.3333333/{\frac13},
    0.5/{\frac12},
    0.625/{\frac58},
    0.6666667/{\frac23},
    0.75/{\frac34},
    1/{1},
    2/{2}
  }{
    \draw[tick] (\x, -0.015) -- (\x, 0.015);
    \node[below] at (\x,-0.02) {$\lab$};
  }
  \foreach \x/\lab in {
    0.6875/{\frac{11}{16}},
  }{
    \draw[tick] (\x, -0.015) -- (\x, 0.015);
    \node[below] at (\x+0.015,-0.02) {\scriptsize$\lab$};
  }
  
    \draw[tick] (-0.015,0) -- (0.015,0);
    \draw[tick] (0,-0.015) -- (0,0.015);
    \node[below] at (0,-0.02) {$0$};
\end{tikzpicture}%
}

\usetikzlibrary{decorations.pathmorphing}
\tikzstyle{black vertex}=[fill=black, draw=none, shape=circle, tikzit category=nodes]
\tikzstyle{blue vertex}=[fill=blue, draw=none, shape=circle, tikzit category=nodes]
\tikzstyle{red vertex}=[fill=red, draw=black, shape=circle, tikzit category=nodes]
\tikzstyle{G{n+1}}=[fill=red, draw=black, shape=circle, tikzit category=vertex set colour scheme]
\tikzstyle{smooth n}=[fill=blue, draw=black, shape=circle, tikzit category=vertex set colour scheme]
\tikzstyle{G{n}}=[fill={rgb,255: red,0; green,150; blue,0}, draw=black, shape=circle, tikzit category=vertex set colour scheme]
\tikzstyle{empty}=[fill=none, draw=black, shape=circle, tikzit category=vertex set colour scheme, tikzit fill=white]
\tikzstyle{G-bdry}=[fill=blue, draw=gamdomcol, shape=circle, thick, tikzit category=vertex set colour scheme]
\tikzstyle{white}=[fill=white, draw=black, shape=circle]
\tikzstyle{G{n+1}-bdry}=[fill={rgb,255: red,128; green,0; blue,128}, draw=gamdomcol, shape=circle, tikzit category=vertex set colour scheme, thick]
\tikzstyle{1/2 smooth n}=[fill=blue, draw=black, shape=circle, tikzit category=vertex set colour scheme, scale={1/2}]
\tikzstyle{1/3 smooth n}=[fill=blue, draw=black, shape=circle, tikzit category=vertex set colour scheme, scale={1/3}]
\tikzstyle{1/2G{n}q}=[fill={rgb,255: red,0; green,150; blue,0}, draw=black, shape=circle, scale={1/2}]
\tikzstyle{1/4 smooth n}=[fill=blue, draw=black, shape=circle, tikzit category=vertex set colour scheme, scale={1/4}]
\tikzstyle{1/2}=[fill=none, draw=black, shape=circle, tikzit category=vertex set colour scheme, scale=0.5, tikzit fill=white]
\tikzstyle{yellow vertex}=[fill=yellow, draw=black, shape=circle]
\tikzstyle{mag vertex}=[fill={rgb,255: red,178; green,0; blue,178}, draw=black, shape=circle]
\tikzstyle{new style 0}=[fill={rgb,255: red,119; green,119; blue,119}, draw=black, shape=circle, thick]
\tikzstyle{labelnode}=[rectangle, draw=none, fill=none, inner sep=2pt, outer sep=0pt]
\tikzstyle{labelnodesm}=[rectangle, draw=none, fill=none, inner sep=2pt, outer sep=0pt, font={\small}]
\tikzstyle{labelnodeft}=[rectangle, draw=none, fill=none, inner sep=2pt, outer sep=0pt, font={\footnotesize}]
\tikzstyle{bknode}=[fill=black, draw=none, shape=circle, tikzit category=nodes, inner sep=0pt, minimum size=7pt, {every label/.append style}={inner sep=1pt}]
\tikzstyle{bknodesm}=[fill=black, draw=none, shape=circle, tikzit category=nodes, inner sep=0pt, minimum size=5pt, {every label/.append style}={font=\small, inner sep=1pt}]
\tikzstyle{blnodeft}=[fill=blue, draw=none, shape=circle, tikzit category=nodes, inner sep=0pt, minimum size=3pt, {every label/.append style}={font=\footnotesize, inner sep=1pt}]
\tikzstyle{gnnodeft}=[fill=green, draw=none, shape=circle, tikzit category=nodes, inner sep=0pt, minimum size=5pt, {every label/.append style}={font=\footnotesize, inner sep=1pt}]

\tikzstyle{v thick dots}=[-, mydots]
\tikzstyle{mydashed}=[-, dashed]
\tikzstyle{arrow}=[->]
\tikzstyle{grey edge}=[-, draw={rgb,255: red,180; green,180; blue,180}]
\tikzstyle{G-domain}=[-, draw=gamdomcol]
\tikzstyle{thick arrow}=[->, draw=black, very thick]
\tikzstyle{mapst}=[{|->}]
\tikzstyle{mythick}=[-, thick]
\tikzstyle{thickdashed}=[-, dashed, thick]
\tikzstyle{dashed+G-domain}=[-, draw=gamdomcol, dashed]
\tikzstyle{bluedashed}=[-, dashed, thick, draw=blue]
\tikzstyle{yellowedge}=[-, draw={rgb,255: red,255; green,226; blue,0}, thick]
\tikzstyle{rededge}=[-, draw=red, thick]
\tikzstyle{greenedge}=[-, draw={rgb,255: red,0; green,150; blue,0}, thick]
\tikzstyle{magentaedge}=[-, draw={rgb,255: red,178; green,0; blue,178}, thick]
\tikzstyle{bluesolid}=[-, draw=blue, thick]
\tikzstyle{divmapsto}=[draw={rgb,255: red,0; green,150; blue,0}, {|->}]
\tikzstyle{orangeyedge}=[-, draw={rgb,255: red,255; green,128; blue,0}]
\tikzstyle{snakey}=[->, decorate, decoration={snake, segment length=5mm, amplitude=2pt}]
\tikzstyle{label arrow}=[->, draw=red]

\usetikzlibrary{decorations.pathmorphing}

\setlist[enumerate]{before=\setlength{\baselineskip}{20pt}, itemsep=0pt}
\setlist[itemize]{before=\setlength{\baselineskip}{20pt}, itemsep=0pt}

\renewcommand{\arraystretch}{1.5}

\newcommand{\F}{\mathcal F}
\newcommand{\J}{\mathcal J}
\newcommand{\V}{\mathcal V_x}
\newcommand{\T}{\mathcal T}
\newcommand{\TV}{\mathcal W}
\newcommand{\lo}{o}
\newcommand{\hk}{\mathbb{\hat K}}
\renewcommand{\k}{k}
\newcommand{\K}{\mathbb K}
\newcommand{\Kk}[1][\k]{\mathbb K_{#1}}
\newcommand{\hkk}[1][\k]{\hk_{#1}}

\newcommand{\bH}{\mathbb{H}}
\newcommand{\CD}{\overline D}

\newcommand{\an}{\text{an}}
\newcommand{\ma}{\mathfrak a}
\newcommand{\mb}{\mathfrak b}

\newcommand{\mg}{\mathfrak g}
\newcommand{\mm}{\mathfrak m}
\newcommand{\q}{q}
\newcommand{\bvec}[1]{{\bf#1}}
\newcommand{\emp}{\emptyset}
\newcommand{\prim}{\mathcal{P}}
\newcommand{\nin}{\notin}

\newcommand{\dilating}{dilating{}}
\newcommand{\adilating}{a \dilating{}}
\renewcommand{\and}{\text{ and }}
\newcommand{\dashto}{\dashrightarrow}

\newcommand{\Dir}[1]{T_{#1}\P^1_\an}

\newcommand{\fml}{\mathfrak}

\newcommand{\smodel}{\Sigma}

\newcommand\capcirc{\mathbin{\ooalign{$\cap$\cr
  \hidewidth\raise.2ex\hbox{\scalebox{0.4}{\mbox{$\circ$}}}\hidewidth\cr}}}

\DeclareMathOperator{\ord}{ord}

\DeclareMathOperator{\dv}{div}
\DeclareMathOperator{\Div}{Div}

\DeclareMathOperator{\emb}{emb}

\DeclareMathOperator{\PGL}{PGL}

\DeclareMathOperator{\rdeg}{rdeg}
\DeclareMathOperator{\wdeg}{wdeg}

\DeclareMathOperator{\Hull}{Hull}

\DeclareMathOperator{\GCD}{GCD}
\DeclareMathOperator{\diam}{diam}

\DeclareMathOperator{\Frac}{Frac}

\DeclareMathOperator{\redct}{red}

\DeclareMathOperator{\FSpec}{FSpec}

\DeclareMathOperator{\Gal}{Gal}

\DeclareMathOperator{\themultiplier}{m}
\newcommand{\mt}[3][\phi]{\themultiplier_{#1}(#2, #3)}
\newcommand{\mpt}[2][\phi]{\themultiplier_{#1}(#2)}
\newcommand{\mmp}[1]{\themultiplier_{#1}}
\newcommand{\fmult}{dilation}

\newcommand{\mono}{monomial}
\newcommand{\Mono}{Monomial}
\newcommand{\equivar}{dilating{}}
\newcommand{\anequivar}{a \equivar{}}
\newcommand{\annuloid}{annuloid}
\newcommand{\Annuloid}{Annuloid}
\newcommand{\anannuloid}{an \annuloid{}}
\newcommand{\spine}{spine}

\newcommand{\vsmooth}{Farey}
\newcommand{\dsmooth}{Farey}

\newcommand{\shiftbracket}[3]{\raisebox{#1}{$\left(\raisebox{-#1}[#2][0mm]{$#3$}\right)$}}

\makeatletter
\newcommand{\abs}[2][\@nil]{%
  \def\tmp{#1}%
   \ifx\tmp\@nnil
       \left\lvert#2\right\rvert
    \else
         \left\lvert#2\right\rvert_{#1}
    \fi}
\makeatother

\newtheorem*{thm*}{Theorem}

\def\thesisbool{0}
\newcommand{\thesisarticle}[2]{\if\thesisbool1{#1}\else{#2}\fi}

\newcommand{\refproplowerstarfunctorial}{\thesisarticle{\autoref{prop:lowerstarfunctorial}}{\cite[Proposition 3.2]{berkskew}}}
\newcommand{\refchapberk}{\thesisarticle{\autoref{??}}{\cite{berkskew}}}
\newcommand{\refsecskew}{\thesisarticle{\autoref{sec:skew:skew}}{Section 3.2}}
\newcommand{\reffatoujuliabasics}{\thesisarticle{\autoref{prop:dyn:fatoujuliabasics}}{\cite[Proposition 4.19]{berkskew}}}
\newcommand{\refthmskewopencts}{\thesisarticle{\autoref{thm:skew:opencts}}{\cite[Theorem 3.7]{berkskew}}}
\newcommand{\refpropskewcomp}{\thesisarticle{\autoref{prop:skew:comp}}{\cite[Proposition 3.3]{berkskew}}}
\def\thmintervalstretch{\Mono{} \Annuloid{} Theorem}

\newcommand{\restatemonoannuloidthm}{\thmintervalstretch{}, \thesisarticle{\ref{thm:skew:ultimatemonoannuloid}}{\cite[Theorem 3.69]{berkskew}}}
\newcommand{\refthmskewreduction}{\thesisarticle{\autoref{thm:skew:reduction}}{\cite[Theorem 3.60]{berkskew}}}
\newcommand{\reflocalmults}{\thesisarticle{\autoref{sec:localmults}}{\cite[Section 3.8]{berkskew}}}

\begin{document}
  \title{The Calculus of Blowups on a Ruled Surface}
 \author{\mylongname}
  \address{Department of Mathematics\\ Brown University\\ Providence\\ RI 02912\\ USA}
 \email{richard\_birkett@brown.edu}
 
\begin{abstract}
The purposes of this article are threefold. First, to determine numerically when an arbitrary blowup of a smooth surface is smooth. We show the surface is smooth if and only if certain rational parameters involving log discrepancy and multiplicity of the exceptional divisors form a generalised Farey sequence within the dual graph of divisors. 
Second, in doing the above we provide an exposition of the Berkovich projective line $\P^1_\an(\K)$ over the Puiseux series as a universal dual graph for divisors on a ruled surface. 
Third, to explain how non-Archimedean skew products interact with this multiplicity structure of the tree.
\end{abstract}
 
 \maketitle
 

\makeatletter
\let\the@thm\thethm
\let\the@prop\theprop
\let\the@cor\thecor
\let\the@lem\thelem
\makeatother

\renewcommand*{\thethm}{\Alph{thm}}
\renewcommand*{\theprop}{\Alph{prop}}
\renewcommand*{\thecor}{\Alph{cor}}
\renewcommand*{\thelem}{\Alph{lem}}

\section{Introduction}

The resolution of singularities on a two-dimensional variety is quite well understood. The general process involves repeatedly finding singularities and blowing them up, until no more remain. However, there is typically no explicit procedure or descriptive pattern for these blowups, even when plentiful information is given about the desired resolution. In applications with a rational map $f : X \dashto Y$ of varieties, one uses that always possible to resolve its indeterminacy (where $f$ is not continuously defined) or its exceptional subvarieties through blowups. For instance, $f$ on a surface may contract a curve $C \subset X$ to a point $f(C) = q$ and we might seek a modification $\pi : Z \to Y$ which blows up $q$ until the lift $(\pi^{-1}\circ f)(C)$ is a curve. The local behaviour of $f$ near $C$ appears to hint at how many blowups are required, and ought to describe $\pi$ explicitly. A similar situation arises in reverse regarding the images of indeterminate points. See \autoref{ex:mono34} below.

For toric varieties we have a more clear perspective. For instance, the universal dual graph of boundary divisors for toric surfaces forms a circle, parametrised naturally and geometrically by coprime integer pairs (rays of rational slope in $\R^2$). Here, blowing up the intersection of two components $E_{(a, b)} \cap E_{(c, d)}$ yields precisely the divisor labelled $(a+ c, b+d)$. Furthermore, the surface is smooth at this point if and only if $ad-bc = \pm 1$ \cite[\S 1.4]{Oda}\cite[\S 2.2]{Ful}. Overall, this machinery has facilitated various applications in dynamics on toric varieties \cite{HP, Fav, JW, BDJ, BDJK}. 
However, toric varieties have limited scope, only allowing for further blowups at the intersections of other component curves in the toric boundary; what if one begins blowing up the uncountably many `free' points lying elsewhere in this boundary cycle? 

As curves can be viewed as valuations on the function field of a surface, valuative spaces provide a more general approach. Beginning with a fixed surface $X$, consider all divisors on a surface obtained through blow-ups and blow-downs over one-dimensional subset $X_0 \subset X$. We can form a \emph{universal dual graph} $\mathcal V$ with a geometric map $\redct_X : \mathcal V \to X_0$. For any birational map $\rho : X \dashto Y$ the dual graph $\Delta(Y)$ with vertices $\Gamma(Y)$ for the components in $Y_0 = \rho(X_0)$ embeds into $\mathcal V$ such that $\redct_Y = \rho \circ \redct_X$. Such a universal dual graph is an infinitely and densely branched $\R$-graph with a finite type skeleton, in fact, there is a natural isomorphism to a Berkovich curve over the Puiseux series $\K = \overline{\k((x))}$.

In this article, we study the situation where this universal dual graph is isomorphic to the Berkovich projective line, $\P^1_\an(\K)$. In the aforementioned case of toric surfaces, the total universal dual graph for the boundary divisors is a Berkovich curve of genus $1$, which has an $\R$-tree emanating from each rational point on its circular skeleton; this structure allows for free blowups.

 \begin{figure}[htbp]
\begin{center}
\includegraphics[width=\textwidth]{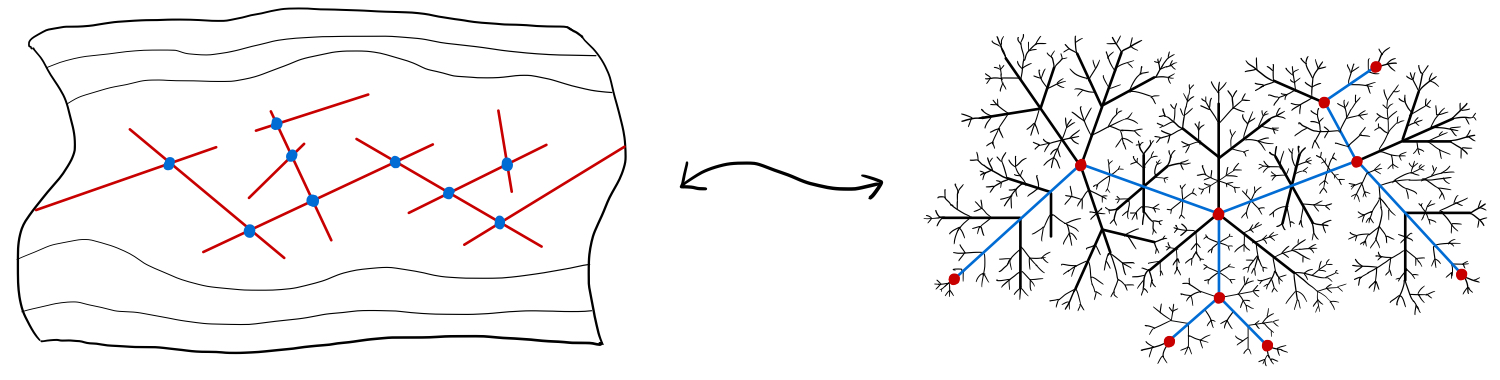}
\caption{A fibre with multiple components related with its dual graph, sitting inside a universal dual graph. Divisors are interpreted by $\Gamma(X)$, the red vertices, and intersections by blue edges.}
\label{default}
\end{center}
\end{figure}

The converse demonstrated by Baker, Payne and Rabinoff \cite{BPR} states that each finite subset of vertices $\Gamma \subset \mathcal V$ in the universal dual graph gives rise to a birational model $Y$ of $X$ such that $\Gamma(Y) = \Gamma$. Note that these model surfaces are possibly (very likely) not smooth. Thus, a natural question to ask is: which of the finite sets yield a smooth model?

Favre and Jonsson required such advanced techniques for their earlier dynamical work on two variable polynomials \cite{FJ07, FJ11}, founded on their multi-faceted exposition on valuative trees \cite{FJ04}. This has been fruitful to study degenerations of one variable maps and also higher dimensional rational maps. In \cite[\S 4, \S6]{FJ11}, they discuss the valuative (Berkovich) approach to dual graphs in detail, defining equivalent \emph{Puiseux} and \emph{Farey} parameters ($\frac ab = A$ therein), and \emph{multiplicities} ($m$ and $b$). See also \cite{berkandapps}.
We observe that the regularity condition for toric surfaces is the same as for the fractions $\frac ab, \frac cd$ to be Farey adjacent and that the blowup yields the \emph{Farey sum} or \emph{mediant} $\frac{a+c}{b+d}$. Although the approach of Favre and Jonsson is essentially restricted to smooth surfaces, they show that the same Farey addition determines parameters of smooth point blowups (see \cite[6.3.1]{FJ04}). This hints strongly that the Farey parameter plays the same r\^ole for smoothness as the integer pairs do in the toric case.

The main contribution of this paper is to better understand these invariants $\mm, \ma, \mb$ and prove a valuative criterion for smoothness in our setting, which turns out to involve a generalised notion of Farey sequences on a tree. These terms are explained further below and in \autoref{defn:smooth:smooth}.

\begin{thm}\label{intro:thm:smooth}
Let $\check X \to B$ be a ruled surface and $X$ be a model of $\check X$ over $t \in B$ and let $\Gamma = \Gamma(X) \subset \P^1_\an$ be its vertex set (a finite set of Type II points), and $U$ a $\Gamma$-domain. Then $U$ is \dsmooth{} if and only if $\redct_X(U) = p$ is a smooth point of $X$.

In particular, any given vertex set (a finite set of Type II points) $\Gamma \subset \P^1_\an$ is \vsmooth{} if and only if there exists a model $X$ of $\check X$ over $\bO$ such that $\Gamma(X) = \Gamma$ and $X$ is smooth at every $p \in X_0$.
\end{thm}

Let $X \to B$ be a birationally ruled surface. Via the reduction map, we associate to each divisor on the special fibre $X_0$ (over some $t\in B$ with $\k[[x]]$ its completed local ring) with a Type II point $\zeta = \zeta(\gamma, \abs x^\frac\ma\mb)$; the fraction $\frac \ma\mb$ forms its \emph{Farey parameter} and $\mm(\zeta) = \mm(\gamma)$ the multiplicity (for a good choice of $\gamma$). A \emph{$\Gamma$-domain} $U \subset \P^1_\an$ is an open set bounded by some vertices in $\Gamma$, and it is said to be \emph{\dsmooth{}} iff it adheres to a generalised Farey condition (\ref{defn:smooth:smooth}) on their parameters described more precisely by the following theorem. Overall, $\Gamma$ is \emph{\vsmooth{}} iff each of its $\Gamma$-domains are \dsmooth{}. 

\begin{thm}\label{thm:intro:annulustfae}
 Let $\Gamma \subset \P^1_\an$ be a finite vertex set of Type II points, and let $U$ be a $\Gamma(X)$-domain. If $U$ is a disk bounded by $\zeta$, then $U$ is \dsmooth{} if and only if $\mb(\zeta) = \min_{\xi \in U}\mm(\xi)$, in which case $\mb(\zeta) \mid \mm(\xi) \mid \mb(\xi)$ for every $\xi \in U$. 
If $U$ is an annulus bounded by $\alpha$ and $\beta$, then TFAE
\begin{enumerate}
 \item $U$ is \dsmooth{}: $\displaystyle\min_{\xi \in U}\mb(\xi) > \max(\mb(\alpha), \mb(\beta))$.
 \item $\mb(\zeta) > \max(\mb(\alpha), \mb(\beta))$ for every $\zeta \in (\alpha, \beta)$.
 \item $\mg(\zeta) \in \N^+\langle \mb(\alpha), \mb(\beta)\rangle$ for all $\zeta \in U$.
 \item $\alpha \prec \beta$ and $\ma(\alpha)\mb(\beta)-\mb(\alpha)\ma(\beta) = \mm(\alpha) = \HCF(\mb(\alpha), \mb(\beta))$ or vice versa.
\end{enumerate}
 \end{thm}

Just as Farey adjacency and sequences can be defined in multiple equivalent ways (e.g.\ \ref{prop:fareyaux:haros}), the notions of \dsmooth{} $\Gamma$-domain have these equivalent forms. Our new adjacency rule is complicated by multiplicity. \autoref{ex:farey:multtwo} gives a smooth surface with transversely intersecting divisors, whose Farey parameters, $\frac 12$ and $\frac 34$, are not Farey adjacent in the classical sense, but $3 \cdot 2 - 1\cdot 4 = 2 = \HCF(2, 4) = \mm(\alpha)$. Furthermore, neither of the two equalities in (iv) can be dropped; see \autoref{rmk:smooth:adjacency}. 

In the process of establishing this characterisation, we will present an updated account of the Berkovich projective line over the Puiseux series as a universal dual graph for divisors. We will delve into the action of a \emph{non-Archimedean skew product} on $\P^1_\an$, studied in \cite{berkskew}, specific to the field of Puiseux series. This corresponds to a rational skew product on a ruled surface and provides a means to comprehend the mapping of curves under such a rational map, regardless of birational equivalence. It turns out that the mapping on $\P^1_\an$ is piecewise linear with rational slopes (of a common denominator) with respect to the metric induced by the Farey parameters. An important technical result in its most basic form is the following.

\begin{thm}\label{thm:intro:multiplicitydivides}
 Let $\phi_*$ be a simple rational skew product, and $\zeta \in \P^1_\an$. Then $\mm(\phi_*(\zeta)) \mid \mm(\zeta)$. Hence $\phi_*(\T_m) \subseteq \T_m$ for any $m \in \N_+$.
\end{thm}

When studying the dynamics of rational maps of the surface, resolving or controlling indeterminacy and contracted curves is of interest in applications such as studying invariant currents and measures \cite{BD, Gue, DDG1, DDG2, DDG3}. It is often important to track the orbits of curves as well as of points. Furthermore, it is highly desirable to understand these orbits in a way that transcends the structure of a fixed surface, rather to take a perspective which is invariant of birational transformations.

 In particular, studying the orbits of curves is of interest to show that a rational map $f : X \dashto X$ on a surface is potentially algebraically stable, meaning we can find a choice of birational $\pi : Y \dashto X$ where the lift $g : Y \dashto Y$ of $f : X \dashto X$ is algebraically stable; of course, it is also helpful in proving when this is impossible. In a subsequent work of the author, a characterisation for potential algebraic stability is given for rational skew products: a rational map on a ruled surface which respects its fibration; see \cite{skewstab}. The present article is necessary to prove the main theorem \cite[Theorem A]{skewstab} about finding a \emph{smooth} stabilisation for rational skew products.

Note that although we appear to specialise to ruled surfaces, in reality this is a slight generalisation; the smoothness criterion applies to blowups infinitely near any closed point of any surface, by restricting to the open Berkovich disk. Further, it should also allow for gluing together charts of a larger Berkovich space, involving a cycle of rational curves on a surface. In contrast to the presentation of Favre and Jonsson \cite{FJ04}, we assume \cite{BPR} so that less constructive heavy lifting is required. This means that most of our work on multiplicity lies in the Berkovich projective line, where it is more readily comprehensible. For instance, this avoids reasoning on the dual graph, which typically requires repetitive inductions.

Now we give an overview of the structure of the \thesisarticle{chapter}{article} and its contents. We begin with an exposition of these techniques in \autoref{sec:moremaps} through several examples. This is a practical introduction and builds upon the technical details and certain insights deferred to later. Diagrams abound. This includes brief discussions of skew products, dual graphs and Farey parametrisation with a view to applications. \autoref{ex:secondexample} can be considered an example of the main theorem where we construct the minimal \vsmooth{} vertex set which contains $\zeta(x^{5/7} + x + x^{4/3}, \abs x^{3/2})$; this requires $17$ blowups on $\P^1\times \P^1$. For the reader interested in the essential take-aways, understanding \autoref{sec:moremaps} may be a sufficient goal, taking items from other sections only as needed.

In \autoref{sec:background} we review essential background from a diverse range of topics. It begins with the basics of blowups on surfaces, and some instructive examples. This is followed by a note on skew products, Puiseux series and their Galois absolute group over the Laurent series, the Berkovich projective line, and a full introduction to Farey sequences (which we define slightly differently).

Later, in \autoref{sec:space} we dive into the finer geometric and arithmetic structure of the Berkovich projective line, $\P^1_\an(\K)$ over the field of Puiseux series $\K = \Kk$ (or Levi-Civita series $\hkk$). Here we study the concepts of \emph{multiplicity} $\mm$ and \emph{generic multiplicity} $\mg$, introduced in \cite{FJ04}, which prove fundamental to our geometric understanding of dual graphs. These multiplicities can be read off a Puiseux series and in turn correspond to the Galois group. The space is an $\R$-tree with an interior \emph{hyperbolic} metric and branching compatible with these multiplicities.

Then, in \autoref{sec:maps} we focus on the translation from a skew product on a ruled surface to a non-Archimedean skew product on the Berkovich projective line. According to the most general definition given in \cite{berkskew}, this requires a little care to extend an algebra map from Laurent series to Puiseux series. The non-uniqueness of this extension is explored in terms of the Galois group and the notion of multiplicity. The section finishes by observing how the invariants of divisors change under a rational mapping; we prove a version of \autoref{thm:intro:multiplicitydivides} (c.f.\ \autoref{thm:multiplicitydivides}) in its highest reasonable generality: \autoref{thm:multiplicitydividesadv}.

\autoref{sec:corresp} continues the ideas in \autoref{sec:space} with an exposition of $\P^1_\an(\K)$ as a universal dual graph for fibral divisors and sections of a ruled surface. We explain how a finite set of vertices $\Gamma$ of the tree correspond to divisors on the surface, and the open subsets bounded by them, \emph{$\Gamma$-domains}, are closed points on the surface. The section concludes in \autoref{sec:metrics} by defining the \emph{log discrepancy} $\ma$ and demonstrating equivalence of multiplicity of (Puiseux) Berkovich points with the geometric invariants of the corresponding divisors and sections on the surface ($\mb = \mg$). The Farey parameter (fraction or pair) $\ma/\mb$ turns out induce a metric equal to the hyperbolic metric for $\P^1_\an(\K)$.

Perhaps the best property of the Farey parameter pair is that it also controls blowups on the surface through Farey addition; see \autoref{thm:farey:addition}. \autoref{sec:smooth} expands on this observation to prove the (valuative) criteria for smoothness of birationally ruled surfaces \autoref{intro:thm:smooth} (c.f.\ \autoref{thm:smooth}). This involves the notion of Farey sequences generalised beyond $\Q$ to a tree with rational points. The new Farey adjacency is explained by \autoref{thm:intro:annulustfae} c.f.\ \autoref{thm:smooth:multdesc}, \autoref{thm:farey:annulustfae} etc.

\subsection*{Comparison with Favre \& Jonsson}

To be clear, much of this work extends that of C.\ Favre and M.\ Jonsson. In \cite{FJ04} they studied the invariants $\ma, \mm, \mb=\mg$, including proofs of equivalences between trees, dual graph representations and metrics. Their book focuses on the valuative tree $\mathcal V$ which represents blowups over a single point normalised by $\min(\nu(x), \nu(y)) = 1$, whereas our focus on a `relative valuative tree' normalised by $\nu(x) = 1$. The relative valuative tree $\mathcal V_x$ which Favre \& Jonsson study primarily in \cite[\S4]{FJ04} is subject to the extra condition $\nu(y) > 0$ and turns out to be equivalent to the subset $\overline {D_\an(0, 1)} \subset \P^1_\an(\k((x)))$ of the tree studied in the present article. However, please note that much of their presentation of a universal dual graph \cite[\S6]{FJ04} and blowups is based in the former setting $\mathcal V$ which has a slightly different parametrisation; this difference is ultimately explained in their subsection \cite[\S6.8]{FJ04}. It is possible to view a subset of one in the other with a little work; for instance $\mathcal V$ is the closed disk $\CD_\an(0, 1) \subset \P^1_\an$ with altered parametrisation and \cite[Theorem 6.51]{FJ04} gives another direction. Therefore, the reader may consider \autoref{sec:space} largely as expository work. \autoref{sec:corresp} recalls and extends results found in \cite[\S6]{FJ04}, combined with description of models and dual graphs akin to others \cite{BL85, BL93, BPR, DeF2, BFJ}. Our approach takes a geodesic to an understanding of the universal dual graph grounded in the (entire) Berkovich projective line, as opposed to Favre and Jonsson's standpoint of equivalence between different definitions of trees. Note that  \autoref{sec:corresp} ends with our own important technical result \autoref{thm:farey:puiseuxsection}, in preparation for the valuative criterion for smoothness. \autoref{sec:smooth} builds upon the above to deliver \autoref{thm:smooth}. Favre \& Jonsson practically take \autoref{thm:farey:addition} as definition of $\ma, \mb$ and work through equivalence metrics later. We prefer to define Farey parameters more intrinsically and prove they have the desired properties. Further comparisons to \cite{FJ04} are signposted therein.

\makeatletter
\let\thethm\the@thm
\let\theprop\the@prop
\let\thecor\the@cor
\let\thelem\the@lem
\makeatother

  \tableofcontents
  \clearpage
  

\section{The Calculus of Blowups}\label{sec:skewprodcor}\label{sec:moremaps}

In this section, we show the ideas of this \thesisarticle{chapter}{article} in application by example, especially with an eye to dynamics. The result is a `calculus of blowups' for ruled surfaces, where we can understand how a skew product $\phi : X \dashto X$ maps divisors through the mapping of vertices in a universal dual graph $\P^1_\an$ under $\phi_*$. Gradients of the resulting maps are piecewise linear and rational, and these vertices we study are the rational points on the tree.

\subsection{Skew Products on a Ruled Surface}

A \emph{skew product} $\phi : X \dashto X$ on a ruled surface $h : X \to B$ over $k = \bar k$ can be defined by a commuting diagram of maps over a self-map on the base $\phi_1 : B \to B$. The key idea in our applications is that $\phi$ induces a non-Archimedean skew product $\phi_*$ on the Berkovich projective line $\P^1_\an$ over the Puiseux series $\Kk$. The latter map will represent the action of $\phi$ on divisors and points in (and around) a source and target fibre of $X$. 

\begin{minipage}{0.25\textwidth}
 \[
\begin{tikzcd}
 X \arrow[dashed]{r}{\phi} \arrow[swap]{d}{h} & X \arrow{d}{h} \\
 B \arrow[swap]{r}{\phi_1} & B
\end{tikzcd}
\]
\end{minipage}
$\leadsto$\hfill
\begin{minipage}{0.25\textwidth}
   \[
\begin{tikzcd}
 \K(y) & \K(y) \arrow[swap]{l}{\phi^*} \\
\K \arrow[hook]{u} & \K \arrow[hook]{u} \arrow{l}{\phi_1^*}
\end{tikzcd}
\]
\end{minipage}
$\leadsto$\hfill
\begin{minipage}{0.3\textwidth}
 \begin{align*}
 \phi_* : \P^1_\an(\K) &\longrightarrow \P^1_\an(\K)\\
 \zeta &\longmapsto \phi_*(\zeta)\\
 \text{where } \norm[\phi_*(\zeta)]{f} &= \norm[\zeta]{\phi^*(f)}^\q
\end{align*}
\end{minipage}

Localising to a fibre $b \in B$ (and $\phi_1(b)$ respectively) and taking completion over the local ring $\bO_b$, this is (locally) equivalent to a rational skew endomorphism $\phi^* : \k((x))(y) \to \k((x))(y)$. Deriving this is easy if $B \cong \P^1$, as $\phi$ is naturally given as a two variable function $\phi(x, y)$; computing a completion is not required. See \autoref{ex:firstexample} for a very simple example in practice. In \autoref{sec:maps} we describe in detail how to extend this to a map over the Puiseux series $\K = \Kk$, and whence a non-Archimedean skew product $\phi_* : \P^1_\an(\K) \to \P^1_\an(\K)$. Later, in \autoref{sec:corresp} we examine the arithmetic structure of $\P^1_\an(\K)$, with its Farey parameters. 

For simplicity, we will focus attention to the situation where $\phi_1(b) = b$ is fixed. Recall that each Type II point $\zeta(\gamma, \abs x^r)$ represents a divisor, say $E_r$ in some model $X$ of $\P^1\times \P^1$ via the reduction map, $\redct_X$. In a general model, the reduction of $\zeta(\gamma, \abs x^r)$ can be a divisor or a closed point depending on if $E_r$ is realised in the model. To a model we associate a vertex set $\Gamma = \Gamma(X)$ of those representing divisors in $X_b$. The open components, called $\Gamma$-domains, bounded by these vertices represent closed points.

Before we proceed, let us briefly recall and fix terminology about rational maps on a surface (see \autoref{sec:background}). If the proper transform of a divisor $E \subset X$ is a closed point, we shall say it is \emph{contracted} by $\phi$. An \emph{indeterminate point} $p \in \mathcal I(\phi)$ is one where $\phi$ is not continuously defined, or equivalently $\phi(p)$ is a curve. Recall also our convention that images of curves are taken by proper transform, which effectively ignores indeterminate points.

As such, whether $\phi_*$ maps a vertex to a vertex or within a $\Gamma$-domain tells us if the associated divisor is mapped by $\phi$ to another divisor or contracted to a point. Likewise, the image of a $\Gamma$-domain will contain another vertex if and only if the associated point is indeterminate. The following proposition is a modification of \cite[Lemma 4.7]{DeF2}. The argument follows from the discussion in \autoref{sec:models}, using \autoref{cor:galois:pointcorrespondence} for instance.

\begin{prop}\label{prop:appl:indetcontract}
 Let $\phi : X \dashto X$ be a skew product on a ruled surface $X \to B$ and consider $b \in B$. Let $\Gamma \subset \P^1_\an$ be the vertex set associated to divisors in $X_b$, similarly let $\Gamma'$ be the vertex set associated to $X_{\phi_1(b)}$; finally let $\phi_* : \P^1_\an \to \P^1_\an$ be the associated rational skew product. Then $\redct_{X_{\phi(b)}}(\phi_*(\zeta))$ is the total transform of $\redct_{X_b}(\zeta)$ by $\phi$. 
 More specifically, suppose $\zeta \in \Gamma$ then $\redct_{X_b}(\zeta) = E$ is a divisor in $X$.
\begin{enumerate}
 \item If $\phi_*(\zeta) \in \Gamma'$, then the proper transform by $\phi$ is $\phi(E) = F = \redct_{X_b}(\phi_*(\zeta))$.
 \item Otherwise, if $\phi_*(\zeta) \nin \Gamma'$ then $E$ is contracted to the closed point $q = \redct_{X_{\phi(b)}}(\phi_*(\zeta)) \in X_{\phi_1(b)}$ by $\phi$. Here, $\phi_*(\zeta)$ lies in the $\Gamma'$-domain $V = \emb_{X_{\phi(b)}}(q)$.
\end{enumerate}
Suppose instead that $U$ is a $\Gamma$-domain then $\redct_{X_b}(U) = p$ is a closed point in $X$.
\begin{enumerate}
 \item If $\phi_*(U)$ is contained in a $\Gamma'$-domain $V$, then $\phi$ is continuous at $p$ and $\phi(p) = q$, where $q = \redct_{X_{\phi(b)}}(V)$.
 \item Otherwise, if $\phi_*(U) \cap \Gamma' = \{\xi_1, \dots \xi_n\} \ne \emp$, then $p$ is an indeterminate point of $\phi$, whose total transform corresponds to $F_1 \cup \cdots \cup F_n$ where $F_i = \redct_{X_{\phi(b)}}(\xi_i)$.
\end{enumerate}
\end{prop}

The next result is copied from \autoref{cor:skew:reduction}, itself a corollary of \autoref{thm:skew:reduction}. 
\begin{prop}
 Let $\phi = (\phi_1, \phi_2)$ be a skew product. Then the reduction $\overline{\phi}(y)$ is the same as the rational mapping $y \mapsto \lim_{x\to 0}\phi_2(x, y)$, or simply $\phi_2(0, y)$. In particular, the degree in a fibre is the degree of the reduction.%
 \[\deg(\phi_2(0, y)) = \deg(\overline{\phi})\] 
\end{prop}

\begin{rmk}
 It follows that $\phi : X \dashto X$ has good reduction in the special fibre ($x=0$ say) if and only if $\phi(0, y)$ has degree $\rdeg(\phi) = \deg_y(\phi_2)\ (= \deg(\overline\phi))$. Failure of this--\emph{bad reduction}--is equivalent to $\zeta(0, 1)$ having a preimage under $\phi_*$ other than itself. See \autoref{thm:skew:reduction}. If $\phi_*(\zeta) = \zeta(0, 1)$ then necessarily this disk or `direction' $\vec v(\zeta)$ corresponds to a closed point whose image is an indeterminate point of $\phi$ on $X$.  Moreover, we can count the multiplicity of each indeterminate point as the number of such preimages of $\zeta(0, 1)$ in its corresponding direction. This remark fully generalises to special fibres with multiple components. See \cite{Faber13.1} for the notion of surplus multiplicity this equates introduced by Faber. The apparent drop in degree $\rdeg(\phi) - \deg(\overline\phi)$ equals the total number of indeterminate points, counting multiplicity.
\end{rmk}

\subsection{Mapping on the Dual Graph}%

A very simple first example demonstrates several aspects of the theory in practice: dual graph selection, mapping, and resolution.

\begin{ex}\label{ex:firstexample}
 Consider the skew product $\phi(x, y) = (x^2, x^2/y)$ on $\P^1\times \P^1$. The induced non-Archimedean skew product $\phi_*$ has scale factor $\q = \frac 12$ and relative degree $1$. It follows that $\phi_*$ uniformly contracts $\P^1_\an$ by $\q$. We study this example for the remainder of the subsection.
 \end{ex}
 
 We can see this easily on the interval $(0, \infty) \subset \P^1_\an$. Observe that $x^r \mapsto \phi_{1*}(x^{2-r}) = x^{1-r/2}$ so $\phi_*(\zeta(0, \abs x^r) = \zeta(0, \abs x^{1-r/2})$. Taking the viewpoint of Farey parameters on this interval, $\phi_*$ is represented by $r \mapsto 1-\frac r2$. 
 
  Notice that the divisor $\{x=0\}$ is contracted by $\phi$ to $(0, 0)$. We can see this through the non-Archimedean lens as follows. This divisor is modelled by $\zeta(0, 1) \in \P^1_\an$; being the whole fibre, we can say that $\Gamma(\P^1\times\P^1) = \{\zeta(0, 1)\}$. We find that $\phi_*(\zeta(0, 1)) = \zeta(0, \abs x)$, which lies in the $\Gamma(\P^1\times\P^1)$-domain $D_\an(0, 1)$. This domain is associated to the closed point $(0, 0)$.
 
 We also notice that $(0, 0) \in \P^1\times\P^1$ is an indeterminate point due to the expression $x^2/y$. Again, this can be seen using \autoref{prop:appl:indetcontract}. Indeed, one can check that $D_\an(0, 1)$, the $\Gamma(\P^1\times\P^1)$-domain associated to $(0, 0)$, is mapped by $\phi_*$ to $\P^1_\an \sm D_\an(0, \abs x)$, which contains $\zeta(0, 1) \in \Gamma(\P^1\times\P^1)$.
 
 Note that whilst $\phi(\{x=0\}) = (0, 0)$ and $\phi(0, 0) = \{x=0\}$, it is not the case that $\phi_*^2(\zeta(0, 1)) = \zeta(0, 1)$, in fact the orbit of $\zeta(0, 1)$ turns out to be \[\zeta(0, 1) \mapsto \zeta(0, \abs x) \mapsto \zeta(0, \abs x^{\frac 12}) \mapsto \zeta(0, \abs x^{\frac 34}) \mapsto \zeta(0, \abs x^{\frac 58}) \mapsto \zeta(0, \abs x^{\frac {11}{16}}) \mapsto \cdots\]
 This corresponds to the Farey parameter orbit of $0$ under $1-\frac r2$.
 \[\frac 01 \mapsto \frac 11 \mapsto \frac 12 \mapsto \frac 34 \mapsto \frac 58 \mapsto \frac {11}{16} \mapsto \cdots\]
 
 \autoref{fig:attractor} depicts the graph of $\phi_*$ restricted to $[\zeta(0, 1), \zeta(1, \abs x^2)]$, presented as the Farey parameters $r \in [0, 2]$. 
 \begin{figure}[ht]
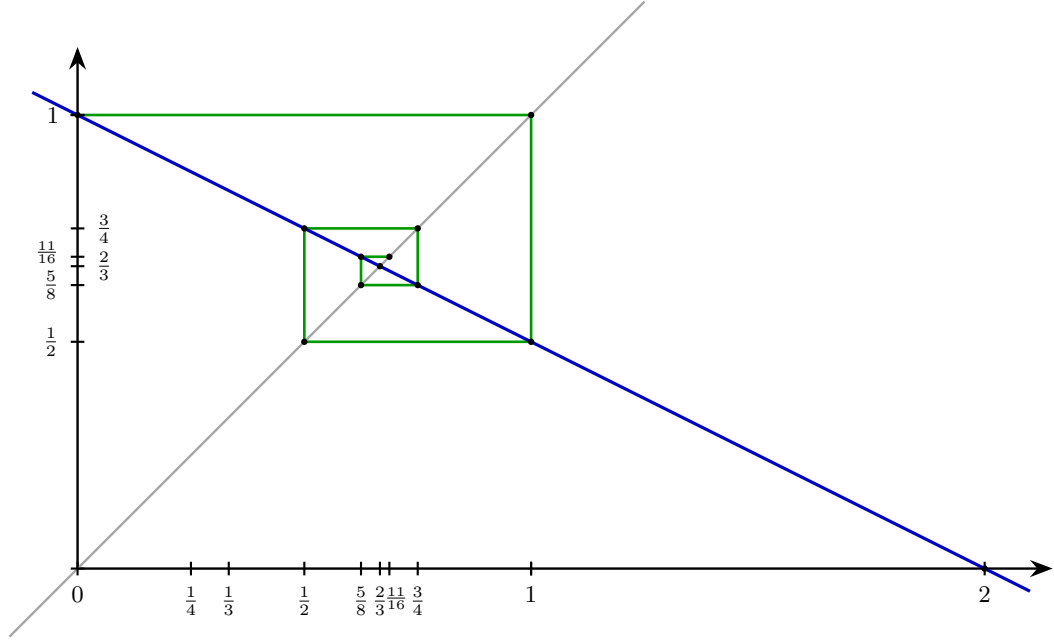

\begin{center}
\ExampleAttractor
\caption{The orbit of $0$ under $r \mapsto 1-\frac r2$.}
\label{fig:attractor}
\end{center}
\end{figure}
\begin{figure}[ht]
\begin{center}
\includegraphics[width=\textwidth]{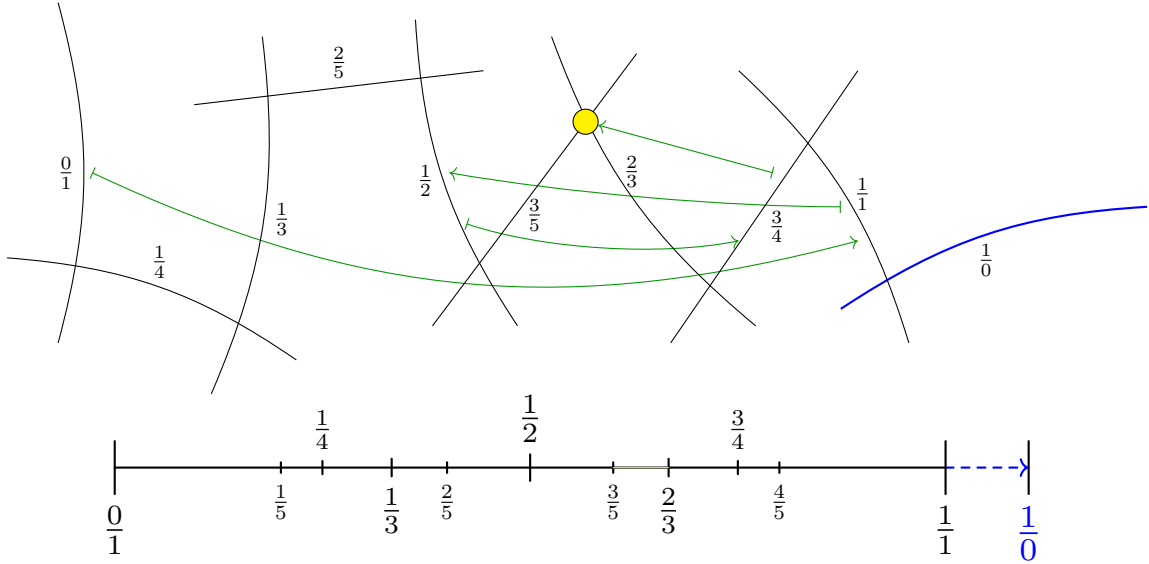}
\SetFracColor{1}{0}{blue}
\def\ColouredIntervalList{3/5/2/3/yellow}
\vspace{-5mm}
\fareylineplain{1}{0.8\textwidth}015
\caption{Beginning with satellite blowups between $\frac 01$ and $\frac 11$ up to height $5$.\\
Accompanied by its dual graph in $[\zeta(0, 1), \zeta(0, \abs x)]$.}
\label{fig:attractorblowupfirst}
\end{center}
\end{figure}
\begin{figure}[ht]
\begin{center}
\includegraphics[width=\textwidth]{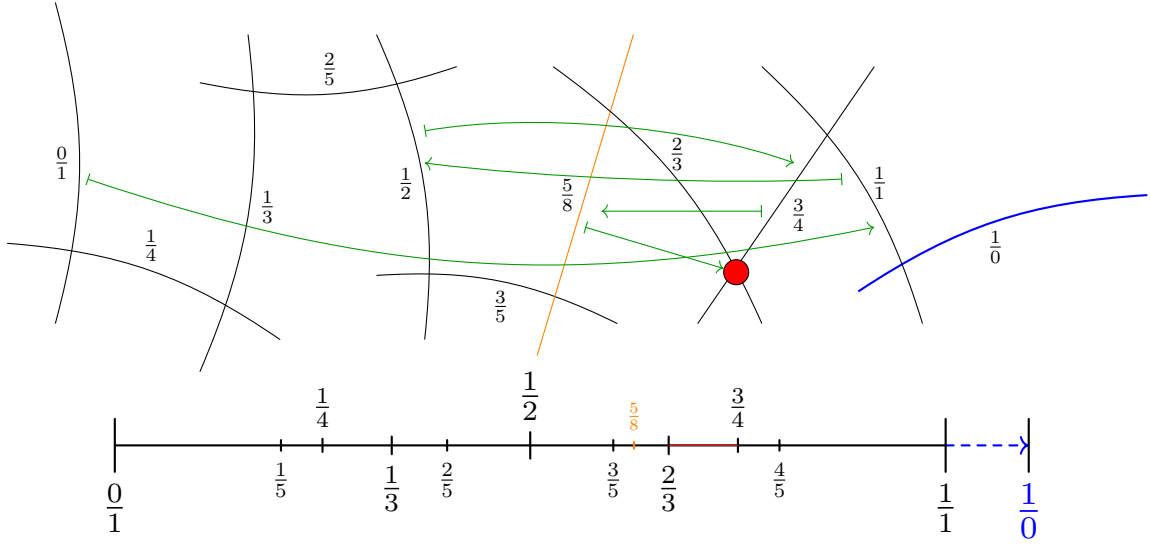}
\def\ColouredIntervalList{2/3/3/4/red}
 \SetFracColor{1}{0}{blue}
\def\ExtraFracList{5/8}
 \SetFracColor{5}{8}{orange}
 \vspace{-5mm}
\fareylineplain{1}{0.8\textwidth}015
\caption{Blowing up resolves the previously contracted $E_{\frac 34}$ to $E_{\frac 58}$.}
\label{fig:attractorblowupsecond}
\end{center}
\end{figure}
\begin{figure}[ht]
\begin{center}
\includegraphics[width=0.8\textwidth]{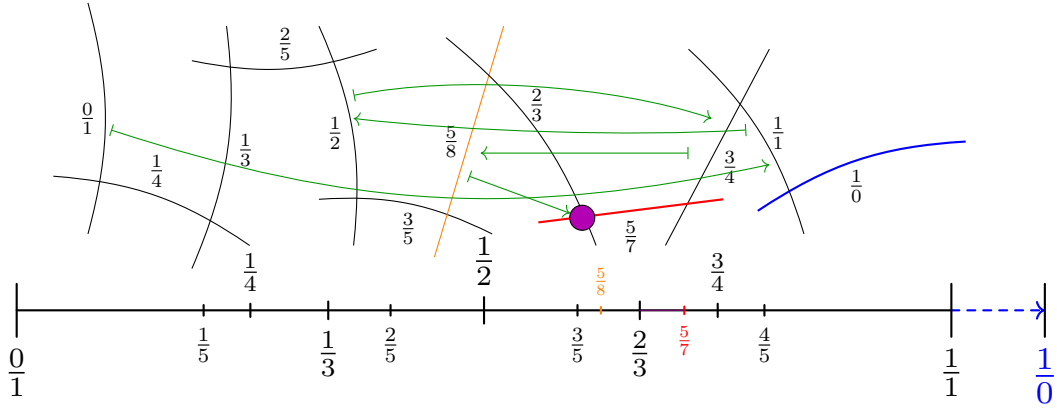}
\def\ColouredIntervalList{2/3/5/7/violet}
 \SetFracColor{1}{0}{blue}
\def\ExtraFracList{5/8, 5/7}
 \SetFracColor{5}{8}{orange}
 \SetFracColor{5}{7}{red}
 \vspace{-5mm}
\fareylineplain{1}{0.9\textwidth}015
\caption{$\phi$ contracts $E_{\frac 58}$ to a point, even after a blowup. This will require three more blowups.}
\label{fig:attractorblowupthird}
\end{center}
\end{figure}
\begin{figure}[ht]
\begin{center}
\includegraphics[width=0.8\textwidth]{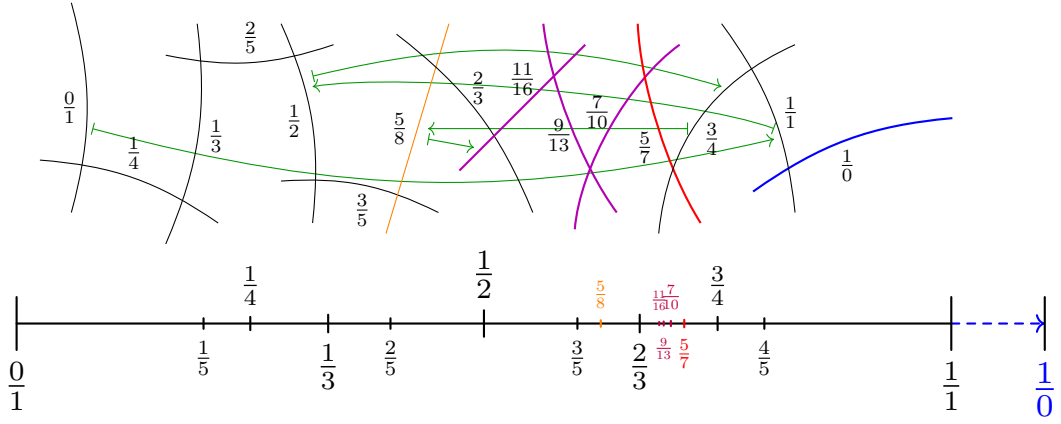}
\def\ExtraFracList{5/8, 5/7, 7/10, 9/13, 11/16}
 \SetFracColor{1}{0}{blue}
 \SetFracColor{5}{8}{orange}
 \SetFracColor{5}{7}{red}
 \SetFracColor{7}{10}{purple}
 \SetFracColor{9}{13}{purple}
 \SetFracColor{11}{16}{purple}
\fareylineplain{1}{0.9\textwidth}015
\caption{Finally the image $\frac 58 \mapsto \frac {11}{16}$ is resolved after three more blowups on the surface corresponding to Farey addition.}
\label{fig:attractorblowupmapping}
\end{center}
\end{figure}

We also see that $r = \frac 23$ is fixed and attracting with multiplier $\frac 12$ which explains the dynamical behaviour. Of course, the orbit of parameters translates into an orbit of divisors, if we have blown up $\P^1\times\P^1$ sufficiently; see Figures \ref{fig:attractorblowupfirst}, \ref{fig:attractorblowupsecond}, \ref{fig:attractorblowupthird}, and \ref{fig:attractorblowupmapping}. As an example calculation, in \autoref{fig:attractorblowupsecond} we see that $\frac 58 \mapsto \frac{11}{16} \in (\frac23, \frac34)$ so we can deduce that $E_{\frac 58}$ gets contracted to the point $E_{\frac 23} \cap E_{\frac 34}$. Blowing up this point to produce $E_{\frac57}$ is not enough to resolve the contracted curve, as $\frac 58 \mapsto \frac{11}{16} \in (\frac23, \frac57)$. See \autoref{fig:attractorblowupthird}. In other words, $\zeta(0, \abs x^{\frac58})$ still gets mapped by $\phi_*$ into a $\Gamma'$-domain (with respect to a new larger vertex set $\Gamma'$) bounded by $\zeta(0, \abs x^{\frac23})$ and $\zeta(0, \abs x^{\frac57})$.

We conclude this example with a more geometric (but still dynamical) \textbf{problem}:\newline
\emph{Can we blowup $\P^1\times \P^1$ to a surface $X$, so that the lift of $\phi$ to $X$ has no indeterminacy?}
\newline Regularisation, even somewhat locally, is rare. However, in this case the answer is \emph{yes}, locally at $(0, 0)$. We claim that after only two blowups corresponding to Farey parameters $\frac 11$ and $\frac 21$, the map becomes continuous. The situation is depicted in \autoref{fig:attractorblowupcts}. First, take $\Gamma(X) = \{\zeta(0, 1), \zeta(0, \abs x)\}$ in the Berkovich projective line. We see that the $\Gamma(X)$-annulus $A$ bounded by these two points is mapped to the annulus $\phi_*(A)$ which is bounded by $\zeta(0, \abs x^{\frac 12})$ and $\zeta(0, \abs x)$. Since $\phi_*(A)$ is contained in a $\Gamma(X)$ domain, namely $A$, it follows from \autoref{prop:appl:indetcontract} that $\phi$ is continuous at $p = \redct_X(A)$ and $\phi(p) = p$.

\begin{figure}[ht]
\begin{center}
\includegraphics[width=\linewidth]{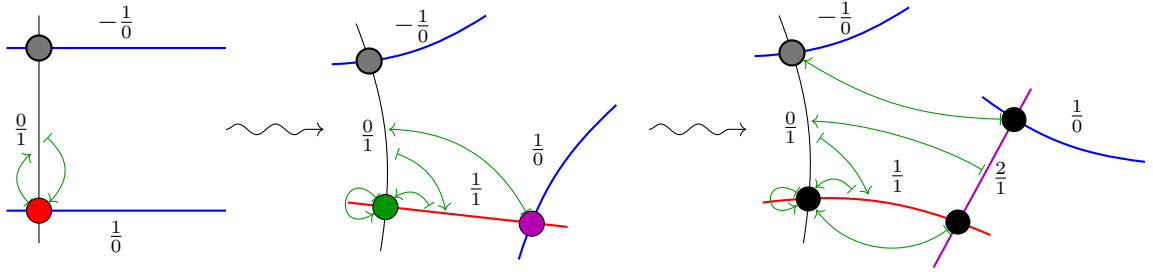}
\caption{After only two blowups, the lift of $\phi$ is continuous.\\
The point $p$, in green, is fixed point where $\phi$ is continuous in a lifted neighbourhood of the origin.}
\label{fig:attractorblowupcts}
\end{center}
\end{figure}

We are not quite done--the new curve $\sim \zeta(0, \abs x)$ created has other closed points and one, whose $\Gamma$-domain is the Berkovich disk $D_\an(0, \abs x)$, is now indeterminate because \sloppy$\phi_*(D_\an(0, \abs x)) =\sloppy \P^1_\an \sm \CD_\an(0, \abs x^{\frac 12})$ contains $\zeta(0, 1)$. However, one more blowup resolves the indeterminacy because it finds $\zeta(0, \abs x^2)$, the preimage of $\zeta(0, 1)$. One can check that the lift of $\phi$ is continuous at all the new closed points by observing that $\phi_*$ maps $\Gamma$-disks to $\Gamma$-disks. 
\begin{align*}
 \phi_*(D_\an(0, \abs x^2)) &= \P^1_\an \sm \CD_\an(0, 1)\\
 \phi_*(D_\an(\gamma, \abs x^2)) &= D_\an(x/\gamma(x^{\half})) \qquad \text{if } \abs\gamma = \abs x^2
\end{align*}

Finally, we remark that $\phi$ cannot be globally regularised, with a smooth surface. Speaking with respect to Farey parameters, for any closed interval $[a, b] \subset \R$, the image, under $t \mapsto 1-2t$ is an interval half the length, so necessarily the image of $\R \sm [a, b]$ intersects $[a, b]$. It follows that for any choice of $a, b \in \Q$ for $\zeta(0, \abs x^a)$ and $\zeta(0, \abs x^b)$ in a hypothetical $\Gamma$, one of these points will always be in the image of some $\Gamma$-domain. Therefore, $\phi$ would have an indeterminate point. That is, unless $a=b = \frac 23$, the (unique) fixed point, or perhaps there are no such $\zeta(0, r)$ in the $\Gamma$-domain. To the former point, our study of smoothness in \autoref{sec:fareyalt} shows that $\zeta(0, 1)$ and $\zeta(0, \abs x)$ would have to appear in any \vsmooth{} vertex set $\Gamma$ which includes $\zeta(0, \abs x^{\frac 23})$; by \autoref{thm:smooth}, no such smooth surface exists, although there certainly exists a singular model with $\Gamma = \set{\zeta(0, \abs x^{\frac 23})}$. The latter point takes more analysis, but it is not too hard to see that any direction $\bvec v$ at a point $\zeta(0, \abs x^t)$ which does not contain $0$ or $\infty$ must be a good direction, and maps likewise to a direction at $\zeta(0, \abs x^{1 - t/2})$. Preimages of these disks are also disks, and it follows that some such disk (a direction $\bvec v$) intersecting $\Gamma$ has a backward orbit which is eventually in a $\Gamma$-domain, thus proving that there is an indeterminate point.

\subsection{Blowup Patterns}

In this non-dynamical example, we see what it takes to resolve a contracted curve of a rational map $f : X \dashto Y$. Often, the number and pattern of blowups in such a desingularisation seems to be a mystery handed to us by an existence theorem. Farey parameters clarify the situation. Much of this was already known, extrapolated from toric blowups or otherwise. However, the ideas contained here extend to more complicated situations where the dual graph is not a line; see \autoref{sec:secondexample}.

\begin{ex}\label{ex:mono34}
 Let $X = Y = \P^1 \times \P^1$ and $C = \set 0 \times \P^1$. 
\begin{enumerate}
 \item Let $f(x, y) = (x^{34}, xy)$, then $C$ is contracted to $(0, 0)$; then $(0, 0)$ must be blown up $34$ times ($\pi_1 : Y_1 \to Y$) to resolve $f$ to $\hat f = \pi_1^{-1}\circ f : X \to Y_1$ so that $f(C)$ is a curve.
 \item If instead we consider $g(x, y) = (x^{34}, x^{11}y)$, then a $14$-fold blowup $\pi_2 : Y_2 \to Y$ is required to resolve the image of $C$.
 \item Better yet, if $h(x, y) = (x^{34}, x^{21}y)$, then only $8$ blowups are needed to resolve $h(C)$.
\end{enumerate}
 Let $E_0 = \set 0 \times \P^1 \subset Y$. In any case, we should first blowup $(0, 0)$ with exceptional curve $E_1$.
\begin{enumerate}
 \item Now, $\pi_1$ is obtained by repeatedly blowing up the exceptional curve at its intersection with (the proper transform of) $C$.
 \item Next, $\pi_2$ is obtained by $3$ such blowups, but then blowing up the intersection of the latest exceptional curve with the fourth one $11$ times.
 \item In the final scenario, $\pi_3$ is generated by an \emph{alternating} pattern after the first, with the $n$-th blowup centred on $E_{n-1} \cap E_{n-2}$ where $C = E_0$.
\end{enumerate}
  
 We can explain all of this using Farey parameters. We use the non-Archimedean setup which is by now familiar, centred over $x=0 \in \P^1$. First note that $C = \redct_X(\zeta(0, 1))$ and so $f(C) = \redct_Y(f_*(\zeta(0, 1)))$ (the same goes for model of $Y$).\\ 
 (i) In the first case, $f_*(\zeta(0, 1)) = \zeta(0, \abs x^{1/34}) \in (0, \zeta(0, 1)) \subset D_\an(0, 1)$. To resolve the image, we must blowup $Y$ until $\zeta(0, \abs x^{1/34})$ reduces to a divisor on the new surface. By \autoref{thm:smooth}, we need a \vsmooth{} vertex set which includes $\zeta(0, 1)$ and $\zeta(0, \abs x^{1/34})$. This is equivalent to having a \emph{Farey sequence} including $0$ and $\frac 1{34}$.
 
 The corresponding Farey parameter, $\frac1{34}$ is deep in the Stern–Brocot tree, because the only way to obtain it is $34$ Farey additions of $\frac 01$ to $\frac 11$.
 \[\frac 01 \oplus \frac 11 = \frac 12, \quad \frac01 \oplus \frac 12 = \frac 13, \quad \frac01 \oplus \frac 13 = \frac 14 \cdots \frac01 \oplus \frac 1{33} = \frac 1{34}\]
 Therefore, considering \autoref{thm:farey:addition}, the resolution of $f(C)$ to a divisor $\redct_{\hat Y}(\zeta(0, \abs x^{1/34}))$ requires blowing up $E_0 \cap E_1$ to obtain $E_{\frac 12}$, blowing up $E_0 \cap E_{\frac12}$ to obtain $E_{\frac 13}$, \dots, blowing up $E_0 \cap E_{\frac1{33}}$ to obtain $E_{\frac 1{34}}$. As described, this is $34$ blowups. Each one is centred at the present image of $C$, $E_0 \cap E_{\frac 1m}$, associated to the $\Gamma$-annulus bounded by $\zeta(0, 1)$ and $\zeta(0, \abs x^{1/m})$ containing $\zeta(0, \abs x^{1/34})$.
 
 (ii) The corresponding problem for $g(x, y) = (x^{34}, x^{11}y)$ is less straightforward but needs fewer blowups. Here, $g_*(\zeta(0, 1)) = \zeta(0, \abs x^{11/34})$, so we need a Farey sequence containing $0$ and $\frac {11}{34}$. Inspecting the number line. We find the following.
 
 \begin{figure}[ht]
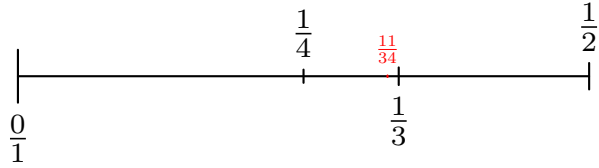

\begin{center}
\def\ExtraFracList{1/3, 1/4, 11/34}%
 \SetFracColor{11}{34}{red}
 \SetLabelScaleExpr{min(1.5, 0.3 + 25/(\Den+10)}
\fareyline{0.5\textwidth}0{1/2}2
\caption{After three blowups to the `left' of $\frac 11$, $\frac{11}{34}$ is in $(\frac14, \frac 13)$ very close to $\frac 13$.}
\label{fig:1134}
\end{center}
\end{figure}

 \begin{figure}[ht]
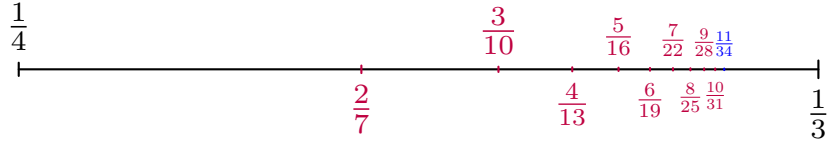

\begin{center}
\def\ExtraFracList{1/3, 1/4, 11/34, 2/7, 3/10, 4/13, 5/16, 6/19, 7/22, 8/25, 9/28, 10/31}
\SetFracColor{11}{34}{blue}
\SetFracColor{10}{31}{purple}
\SetFracColor{9}{28}{purple}
\SetFracColor{8}{25}{purple}
\SetFracColor{7}{22}{purple}
\SetFracColor{6}{19}{purple}
\SetFracColor{5}{16}{purple}
\SetFracColor{4}{13}{purple}
\SetFracColor{3}{10}{purple}
\SetFracColor{2}{7}{purple}
 \SetLabelScaleExpr{min(1.5, 0.3 + 20/(\Den+7)}
\fareyline{0.7\textwidth}{1/4}{1/3}2
\caption{Ten more Farey additions to the `right' and we locate $\frac {11}{34}$, whose associated divisor is therefore resolved after blowing up to find the divisors with the Farey parameters above.}
\label{fig:1134deep}
\end{center}
\end{figure}

(iii) Next, consider $g(x, y) = (x^{34}, x^{21}y)$. By now, the reader might guess that this case is special because $\frac{21}{34}$ is a continued fraction approximant to the golden ratio, and indeed $21, 34$ are Fibonnaci numbers. This means that the $\frac{21}{34}$ is at the lowest possible level in the Stern–Brocot tree relative to its height, and takes a symmetric pattern of left-right-left-right Farey additions to be reached.

 \begin{figure}[ht]
\begin{center}
\def\ExtraFracList{2/3, 3/5, 5/8, 21/34}%
\SetFracColor{21}{34}{red}
 \SetLabelScaleExpr{min(1, 0.1 + 5.7/(\Den+2)}
\fareyline{0.7\textwidth}{0}{1}2
\label{fig:2134}
\end{center}
\end{figure}
 \begin{figure}[ht]
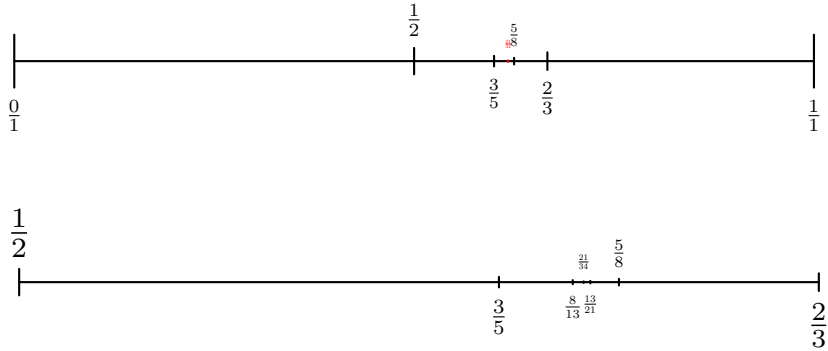

\begin{center}
\def\ExtraFracList{2/3, 3/5, 5/8, 8/13, 13/21, 21/34}
 \SetLabelScaleExpr{min(1.5, 0.3 + 5.5/(\Den+2)}
\fareyline{0.7\textwidth}{1/2}{2/3}2
\caption{A symmetric pattern of Farey additions locates $\frac{21}{34} \simeq \frac{\sqrt 5 - 1}{2}$.}
\label{fig:2134deep}
\end{center}
\end{figure}
\end{ex}

\subsection{Arboreal Example}\label{sec:secondexample}

\begin{ex}\label{ex:secondexample}
Starting with $\zeta(0, 1)$ we will discover (combinatorially or arithmetically) a sequence of blowups to resolve the valuation/norm $\zeta(x^{5/7} + x + x^{4/3}, \abs x^{3/2})$ as a divisor in a modification of $\P^1 \times \P^1$.

In practice, given $\zeta(c_1 x^{q_1} + c_2 x^{q_2} + \cdots + c_n x^{q_n}, \abs x^{q_{n+1}})$, with $q_j$ increasing, one should inductively resolve $\zeta(c_1 x^{q_1} + c_2 x^{q_2} + \cdots + c_m x^{q_m}, \abs x^{q_{m+1}})$, beginning with $\zeta(0, \abs x^{q_1})$. Here, we have $\frac 57 < 1 < \frac 43 < \frac 32$. We start with $\zeta(0, \abs x^{5/7})$. This has multiplicity $\mm = 1$ but generic multiplicity $\mb = 7 > 1$, so it is satellite. 

To smoothly resolve a satellite Type II point $\zeta$ with $\mm(\zeta) = m < \mb(\zeta) = b$ and $\ma(\zeta)= a$, it is necessary to have resolved its neighbouring two free points which can be identified by Farey parameters $\frac{n}{m} < \frac ab < \frac {n+1}m$. After this, satellite blowups pick out the point with parameter $\frac ab$ according to Farey addition, as explained by \autoref{prop:corresp:ordering} and \autoref{thm:farey:addition}.

In this case $m = 1$ so we seek integers; indeed $n=0$, as $\frac 01 < \frac 57 < \frac 11$. We identify $\zeta(0, \abs x^0)$ and $\zeta(0, \abs x^1)$ as targets. We have the former and blowup the origin once to resolve the latter. Now the satellite blowups, which (for now) are similar to previous examples.

 \begin{figure}[ht]
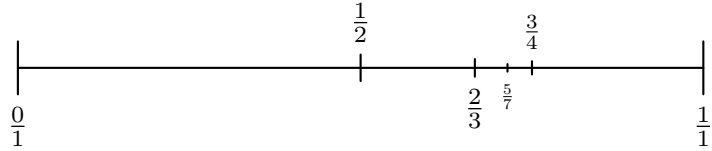

\begin{center}
\def\ExtraFracList{1/2, 2/3, 3/4, 5/7}
\SetLabelScaleExpr{min(1.25, 0.25 + 5/(\Den+2)}
\fareyline{0.6\textwidth}{0}{1}1
\caption{Locating $\zeta(0, \abs x^{5/7})$.}
\label{fig:secondex:1}
\end{center}
\end{figure}

Next, we seek the direction containing $\zeta(x^{5/7}, \abs x^{1/1})$, or equivalently $x^{5/7} + x + x^{4/3}$, or much more simply $x^{5/7}$. This is simply a correct choice of constant $c$ for which $cx^{5/7}$ lands at the divisor corresponding to $\zeta(0, \abs x^{5/7})$. Recall from \autoref{prop:galois:reductionstructure} that directions at this Type II point are $\Gamma$-disks which reduce to particular points on the divisor. In practice, perhaps, this could be found as an intersection with the proper transform of $y^7=x^5$. Note that $\vec v(x^{5/7})$ is necessarily generic with $\mm(\vec v(x^{5/7})) = \mb(\vec v(x^{5/7})) = 7$. The free blowup at this point gives the next free point on $\T_7$ in the direction $\vec v(x^{5/7})$, which will have Farey parameters $\ma = 6$, $\mb = 7 = \mm$. Therefore, we obtain $\zeta(x^{5/7}, \abs x^{6/7})$.

Now we consider free blowups to obtain $\zeta(x^{5/7}, \abs x)$, or if satellite, its two free neighbours. This time, it is free, since $\mb(\zeta(x^{5/7}, \abs x^{1/1})) = \LCM(7, 1) =  7 = \mm(\zeta(x^{5/7}, \abs x^1)$. Since $1 = 7/7$, coming from $6/7$ we need one more free blowup from $\zeta(x^{5/7}, \abs x^{6/7})$ to $\zeta(x^{5/7}, \abs x^{7/7})$. 

Now we repeat, seeking $\zeta(x^{5/7} + x, \abs x^{4/3})$. First we take a free blowup in the direction of $x^{5/7} + x$ at $\zeta(x^{5/7}, \abs x)$, according to a correct choice of (free) closed point. We first obtain $\zeta(x^{5/7} + x, \abs x^{8/7})$. Then, we perform free blowups to obtain $\zeta(x^{5/7} + x, \abs x^{4/3})$ if free itself, or if satellite, then its two free neighbours. In this case, $\mb(\zeta(x^{5/7} + x, \abs x^{4/3})) = \LCM(7, 1, 3) = 21 > 7 =\mm$. We choose $n$ with $\frac{n}{7} < \frac 43 =\frac {28}{21} < \frac {n+1}7$, namely $n=9$. Dividing the denominators in this generalised Farey sequence all by $7$, we begin with $\frac 71, \frac 81$ and seek a Farey sequence including $\frac 91 < \frac {28}{3} < \frac {10}1$. 
The arithmetic suggests two more free blowups to obtain $\zeta(x^{5/7} + x, \abs x^{9/7}), \zeta(x^{5/7} + x, \abs x^{10/7})$, then two satellite blowups giving $\zeta(x^{5/7} + x, \abs x^{19/14})$, then $\zeta(x^{5/7} + x, \abs x^{28/21})$. Again, this is consistent with \autoref{thm:farey:addition}. Modifying the definition of Farey addition for fractions with denominator in $7\Z$:
 \[\frac 77 \oplus \frac 10 = \frac 87, \quad \frac 87 \oplus \frac 10 = \frac 87, \quad \frac 87 \oplus \frac 10 = \frac 97,\quad \frac 97 \oplus \frac 10 = \frac {10}7, \quad \frac 97 \oplus \frac {10}7 = \frac {19}{14}, \quad \frac 97 \oplus \frac {19}{14} = \frac {28}{21}\]

By now, our dual graph looks like the black portion of \autoref{fig:treedualgraph}. As the generic multiplicity is $21$ at the latest divisor, and $ \frac{31}{21} < \frac 32 = \frac {63}{42} < \frac{32}{21}$, we shall need $4$ free blowups from $\zeta(x^{5/7} + x, \abs x^{28/21})$, and a single satellite blowup. These are drawn in blue; the details are left as an exercise.

\begin{figure}[ht]
\tikzset{
  every label/.append style={yshift=-1pt},
}
\begin{center}
\resizebox{0.8\textwidth}{!}{\begin{tikzpicture}
	\begin{pgfonlayer}{nodelayer}
		\node [style=bknode, label={[below=2mm]:$\zeta(0, |x|^0)$}] (0) at (-9, 0) {};
		\node [style=bknode, label={[right=2mm]:$\zeta(0, |x|^1)$}] (1) at (0, -9) {};
		\node [style=bknode, label={[right=2mm]:$\zeta(0, |x|^{\frac{1}{2}})$}] (2) at (0, 0) {};
		\node [style=bknode, label={[right=2mm]:$\zeta(0, |x|^{\frac{2}{3}})$}] (3) at (0, -3) {};
		\node [style=bknode, label={[right=2mm]:$\zeta(0, |x|^{\frac{3}{4}})$}] (4) at (0, -4.5) {};
		\node [style=bknode, label={[right=2mm]:$\zeta(0, |x|^{\frac{5}{7}})$}] (12) at (0, -3.75) {};
		\node [style=bknodesm, label={[above=2mm]:$\scriptstyle\zeta(x^{\frac{5}{7}}, |x|^{\frac{6}{7}})$}] (5) at (-2.5, -3.75) {};
		\node [style=bknodesm, label={[above left=1mm]:$\scriptstyle\zeta(x^{\frac{5}{7}}, |x|^{1})$}] (6) at (-5, -3.75) {};
		\node [style=bknodesm, label={[left=2mm]:$\scriptstyle\zeta(x^{\frac{5}{7}}+x, |x|^{\frac{8}{7}})$}] (7) at (-5, -6.25) {};
		\node [style=bknodesm, label={[right=1mm, yshift=5pt]:$\scriptstyle\zeta(x^{\frac{5}{7}}+x, |x|^{\frac{9}{7}})$}] (8) at (-5, -8.75) {};
		\node [style=bknodesm, label={[below=2mm]:$\scriptstyle\zeta(x^{\frac{5}{7}}+x, |x|^{\frac{10}{7}})$}] (9) at (-5, -11.25) {};
		\node [style=bknodesm, label={[right=2mm]:$\scriptstyle\zeta(x^{\frac{5}{7}}+x, |x|^{\frac{4}{3}})$}] (10) at (-5, -9.5) {};
		\node [style=blnodeft] (11) at (-7.625, -9.5) {};
		\node [style=bknodesm, label={[right=1mm,yshift=-5pt]:$\scriptstyle\zeta(x^{\frac{5}{7}}+x, |x|^{\frac{19}{14}})$}] (13) at (-5, -10) {};
		\node [style=blnodeft] (14) at (-5.75, -9.5) {};
		\node [style=blnodeft] (15) at (-6.5, -9.5) {};
		\node [style=blnodeft] (16) at (-7.25, -9.5) {};
		\node [style=blnodeft] (17) at (-8, -9.5) {};
		\node [style=none, label={[below=-1mm, xshift=-5pt]:$\scriptstyle\zeta(x^{\frac{5}{7}}+x+x^{\frac{4}{3}}, |x|^{\frac{3}{2}})$}] (18) at (-8, -10.25) {};
	\end{pgfonlayer}
	\begin{pgfonlayer}{edgelayer}
		\draw (0) to (2);
		\draw (2) to (3);
		\draw (4) to (1);
		\draw (5) to (6);
		\draw (6) to (7);
		\draw (7) to (8);
		\draw (8) to (10);
		\draw (5) to (12);
		\draw (3) to (12);
		\draw (12) to (4);
		\draw (13) to (9);
		\draw (10) to (13);
		\draw (16) to (11);
		\draw (16) to (15);
		\draw (15) to (14);
		\draw (14) to (10);
		\draw (17) to (11);
		\draw [style=label arrow] (18.center) to (11);
	\end{pgfonlayer}
\end{tikzpicture}}
\caption{Dual graph for \autoref{ex:secondexample}, approximately to scale.}
\label{fig:treedualgraph}
\end{center}
\end{figure}
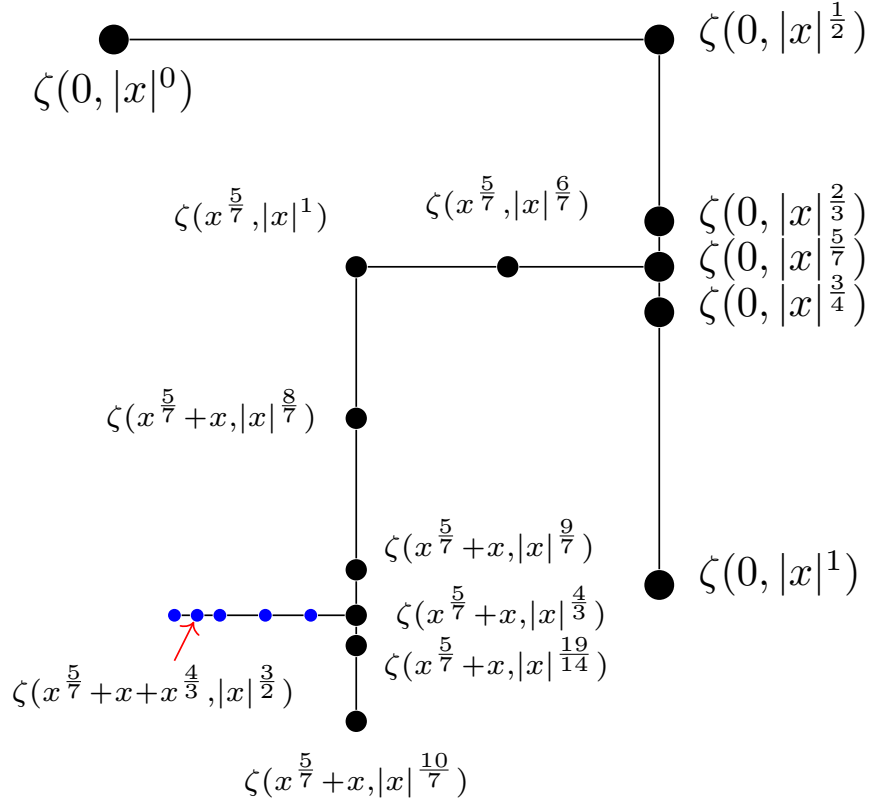

Finally, we remark that including all Galois conjugates of these Type II points would yield a \vsmooth{} vertex set. Indeed, it would be the (unique) smallest one that includes $\zeta(x^{5/7} + x + x^{4/3}, \abs x^{3/2})$.
\end{ex}

\begin{rmk}
 This example demonstrates how on every interval of a fixed multiplicity $\mm = m$ (e.g. $[\zeta(0, \abs x^{5/7}), \zeta(x^{5/7} + x, \abs x^{10/7})$), the Farey parameters form a (not necessarily complete) Farey sequence of rational numbers after dividing denominators by $m$ (e.g. as below)
 \begin{figure}[ht]
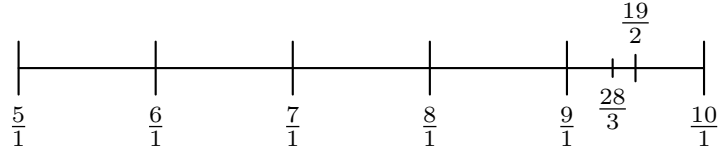

\begin{center}
\def\ExtraFracList{19/2,28/3}
\SetLabelScaleExpr{min(1.25, 0.25 + 5/(\Den+2)}
\fareyline{0.6\textwidth}{5}{10}1
\caption{The multiplicity $7$ limb of the dual graph, Farey parameters multiplied by $7$.}
\label{fig:secondex:7}
\end{center}
\end{figure}

\end{rmk}

\subsection{Algebraic Dynamical Application}

The regularisation of a rational map in the last example is typically too much to hope for, even ``locally'' (as above). In the dynamics of rational maps on higher dimensional varieties, an important concept is \emph{algebraic stability}. This condition can be summarised as $(f^n)^* = (f^*)^n$ on cohomology. On a surface, an equivalent but more geometric and intuitive definition is as follows. Let $f : X \dashto X$ be a rational map on a surface (assumed smooth for simplicity). A \emph{destabilising orbit} is a sequence (or finite orbit) $C \mapsto p_1 \mapsto \cdots \mapsto p_n \mapsto D$ where $C, D$ are curves, meaning $C$ is contracted to $p_1$ and it or a point later in its orbit, say $p_n$, is indeterminate. Both features of a destabilising orbit are more easily detected by our non-Archimedean techniques.

We refer the reader to \cite{DF} or \cite{stability} for a full introduction such as its significance and to \cite{Fav, FJ11, JW, DDG1} for examples of subsequent developments. Indeed, the author's primary application of this material is to understand algebraic stability for skew products on ruled surfaces. The Farey perspective was first used in \cite{Fav} and later Berkovich spaces were employed in \cite{FJ04, FJ07, FJ11}.

The most fruitful feature of the Berkovich perspective is that it allows us to understand the action on divisors independently of the chosen model surface, because $\phi_*$ is the same over all birational models. A key question with regard to algebraic stability is, given a rational map $\phi: X \dashto X$, whether any birational model $Y$ of $X$ exists where the lift of $\phi$ to $Y$ is algebraically stable.

In \thesisarticle{???}{\cite{skewstab}} we will leverage the techniques in the current \thesisarticle{chapter}{article} to answer such questions about algebraic stability for skew products on ruled surfaces. The induced skew products on $\P^1_\an$ exhibit more complicated behaviour than in aforementioned cases. Furthermore, our present classification of smoothness will contribute to successfully finding a stable model $Y$ for a rational skew product which is additionally smooth; this is a feature that has not been achieved in previous works since \cite{DF}. We leave further details to the sequel. 

%

\clearpage


\section{Background}\label{sec:background}

\subsection{Rational Maps}

A rational map $\phi : X \dashto Y$ may have indeterminate points, where it is not continuously defined. We might write $\mathcal I(\phi)$ for the (finite) collection of indeterminate points. In basic terms, these occur where $\frac 00$ appears in a (simplified) formula for $\phi$ in a chart. We can always decompose $\phi$ as $\beta \circ \alpha^{-1}$ where $\alpha : \Gamma_\phi \to X$ is a birational morphism and $\beta : \Gamma_\phi \to Y$ is a morphism of varieties. Assuming $X$ is a surface, the first factor $\alpha$ decomposes as a sequence of point \emph{blowups} (or rather, \emph{blowdowns}) \cite[Proposition 5.3]{Hart}. A blowdown of smooth surfaces $\pi : X_1 \to X_0$ is an isomorphism everywhere except at a rational \emph{exceptional curve} $E_\pi$, whose image $\pi(E_\pi)$ is a closed point in $X_0$. Therefore, using the factorisation $\phi=\beta \circ \alpha^{-1}$ we may interpret the image $\phi(p)$, of an indeterminate point $p$ as a (rational but possibly reducible) curve $\beta(\alpha^{-1}(p))$; this definition of mapping is called the \emph{total transform}. Given a curve $C \subset X$ on the surface, its \emph{proper transform} is defined by $\overline{\phi(C\sm\mathcal I)}$, which is the same as $\phi(C)$ if $C$ has no indeterminate points. We say $C$ is \emph{contracted} by $\phi$ if its proper transform is a closed (i.e.\ $0$-dimensional) point. 
 Note that total transform can produce unexpected results in the form of `extra' curves, when we map a curve $C$ which happens to contain an indeterminate point of $\phi$. In the worst case, the total transform of a contracted curve can also be a curve.\\\textbf{Conventions:} In this \thesisarticle{thesis}{article}, we shall by default mean proper transform when discussing the image of curves under a rational map. Further, all rational maps will be assumed dominant unless explicitly stated otherwise.

\subsection{Primer on Blowups}

Now we will give descriptions of blowups of $X$ and coordinates. For simplicity, we often work on surfaces birational to $\P^1 \times \P^1$ and focus on blowups centred on the origin.

Intuitively, a blowup at the origin in some local coordinates $(z, w)$ is the surface obtained by replacing $P = (0, 0)$ with a projective line ($\P^1$) of points, one point $P$ for each direction passing through $(0, 0)$; indeed, the gluing is made so that the original line $t_0w = t_1z$ lifts to a new line containing $[t_0 : t_1] \in \P^1$ (replacing $(0, 0)$ in the former. The $\P^1$ is referred to as the \emph{exceptional curve} of the blowup. We also say that the blowup \emph{has centre $P$}. 

\begin{defn}
 Given a rational curve $E$ and local coordinates $(z, w)$ around a point $P \in E$ such that $E = \set{w = 0}$, we will call $z$ is \emph{abscissa} and $w$ the \emph{ordinate} with respect to $E$. 
\end{defn}

Blowups are described algebraically in several ways; here is a `hands-on' description with local coordinates.
Consider $\C^2$ with $E_0 = \set{x=0}, H = \set{y=0}$. Here, $x$ is an abscissa for $H$ (the $x$-axis) and $y$ is the abscissa for $E_0$, the ordinates are vice versa; $(x, y)$ provide coordinates. Suppose we blow up the origin $(0, 0)$ (in $\C^2$), the intersection of $E_0$ and $H$. The origin is replaced by a $\P^1 = E_1$ which has coordinates in two charts. In one, we have coordinates $\left(\frac{x}{y}, y\right)$ and in the other $\left(x, \frac{y}{x}\right)$. In the first chart we see proper transforms $\hat E_0 = \set{\frac xy = 0}$ of $E_0$ and $E_1$ as axes of the coordinates; this chart excludes the proper transform $\hat H$ of $H = \set{y=0}$. Note that $y = 0$ in the new coordinates defines $E_1$, or in other words $y$ is the ordinate for $E_1$ in the first chart. In the second chart, we see $\hat H$ and $E_1$, but $E_1$ is defined locally by the vanishing of $x$. In both cases, we keep in mind that $\frac{x}{y}$ or $\frac{x}{y}$ gives an abscissa along $E_1$ whereas the other coordinate measures distance from $E_1$. Moreover, note that in the first chart, $y$ is `still' the abscissa for $\hat E_0$ and in the second chart, $x$ is the abscissa for $\hat H$.

All this generalises to an arbitrary point blowup at the origin of any local coordinates $(\phi_1, \phi_2)$ on any surface. We get two new charts coordinates $(\psi_1, \psi_2)$ given by $\left(\frac{\phi_1}{\phi_2}, \phi_2\right)$ and $\left(\phi_1, \frac{\phi_2}{\phi_1}\right)$. Note that the geometric map from blowup to the old surface is a birational morphism.
\[\C[\phi_1, \phi_2] \longrightarrow \C[\psi_1, \psi_2]\]
\begin{align*}
 \phi_1 \longmapsto& \psi_1 \cdot \psi_2\\
 \phi_2 \longmapsto& \psi_2
\end{align*}

Later, we will discuss dual graphs (trees) representing distinguished exceptional curves and their intersections. When we blow up a point lying on a single distinguished curve (vertex of the tree) we call it a \emph{free blowup}, otherwise if it lies on the intersection of two others we will call it a \emph{satellite blowup}. Beginning with $E_0$ and $H$ above, $E_1$ was a satellite blowup. It is instructive to consider the curves obtained by repeated satellite blowups between $E_0$ and $H$.

\begin{ex}\label{ex:threehalves}
 Continue the example above. The first blowup at the intersection of $E_0 = \set{x=0}$ and $H = \set{y=0}$ yields $E_1 \cong \P^1$. The proper transforms $\hat E_0, \hat H$ of the first two curves are now disjoint, but both intersect $E_1$ transversely. Near the intersection of $E_1$ and $\hat H$ we have coordinates $\left(x, \frac{y}{x}\right)$. Now blowup here to obtain exceptional curve $E_2$ which has two coordinate pairs $\left(\frac{x^2}{y}, \frac{y}{x}\right)$ and $\left(x, \frac{y}{x^2}\right)$. Notice that these are monomial and (birationally) invertible. If we blow up at the intersection of $E_1$ and $E_2$ we get $E_{\frac{3}{2}}$ with charts $\left(\frac{x^3}{y^2}, \frac{y}{x}\right)$ and $\left(\frac{x^2}{y}, \frac{y^2}{x^3}\right)$.
 
Observe that in the final surface, (the proper transform of) $E_1$ has two charts, $\left(\frac{x}{y}, y\right)$ (near $E_0$) and $\left(\frac{x^3}{y^2}, \frac{y}{x}\right)$ (near $E_{\frac{3}{2}}$). Also, $E_2$ has charts $\left(\frac{x^2}{y}, \frac{y^2}{x^3}\right)$ (near $E_{\frac{3}{2}}$) and $\left(x, \frac{y}{x^2}\right)$ (near $H$).

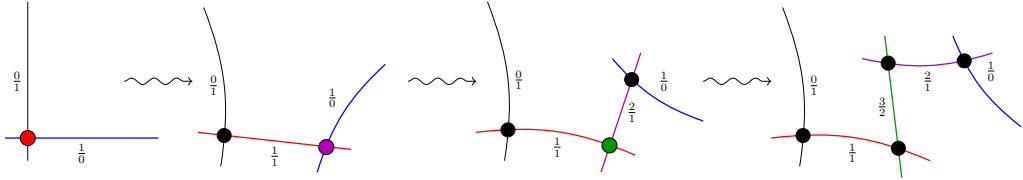
\begin{figure}[ht]
\begin{center}
\resizebox{0.9\textwidth}{!}{\begin{tikzpicture}
	\begin{pgfonlayer}{nodelayer}
		\node [style=none] (0) at (-3, 6.25) {};
		\node [style=none] (1) at (-2.25, -0.75) {};
		\node [style=none] (20) at (5, 3.75) {};
		\node [style=none] (21) at (2, -1) {};
		\node [style=none] (28) at (-3.25, 0.75) {};
		\node [style=none] (29) at (3.5, 0) {};
		\node [style=black vertex] (32) at (-2.1, 0.61) {};
		\node [style=mag vertex] (41) at (2.4, 0.1) {};
		\node [style=none] (45) at (-6.5, 3) {};
		\node [style=none] (46) at (-3.5, 3) {};
		\node [style=none] (47) at (9.5, 6.5) {};
		\node [style=none] (48) at (10.25, -0.5) {};
		\node [style=none] (49) at (15, 4) {};
		\node [style=none] (50) at (19, 1.25) {};
		\node [style=none] (51) at (9, 0.75) {};
		\node [style=none] (52) at (16, -0.25) {};
		\node [style=black vertex] (55) at (10.4, 0.86) {};
		\node [style=none] (57) at (16.25, 4.25) {};
		\node [style=none] (58) at (14.5, -1) {};
		\node [style={G{n}}] (61) at (14.875, 0.175) {};
		\node [style=black vertex] (62) at (15.85, 3.075) {};
		\node [style=none] (66) at (6, 3) {};
		\node [style=none] (67) at (9, 3) {};
		\node [style=none] (68) at (-10.75, 6.5) {};
		\node [style=none] (69) at (-10.75, -0.5) {};
		\node [style=none] (70) at (-11.75, 0.5) {};
		\node [style=none] (71) at (-5, 0.5) {};
		\node [style=red vertex] (74) at (-10.75, 0.5) {};
		\node [style=none] (75) at (19, 3) {};
		\node [style=none] (76) at (22, 3) {};
		\node [style=none] (77) at (22.5, 6.25) {};
		\node [style=none] (78) at (23.25, -0.75) {};
		\node [style=none] (79) at (30, 5) {};
		\node [style=none] (80) at (33, 1) {};
		\node [style=none] (81) at (22, 0.5) {};
		\node [style=none] (82) at (29, -0.5) {};
		\node [style=black vertex] (83) at (23.4, 0.61) {};
		\node [style=none] (84) at (26, 4) {};
		\node [style=none] (85) at (31.75, 4.25) {};
		\node [style=black vertex] (86) at (27.6, 0.05) {};
		\node [style=black vertex] (87) at (30.5, 3.9) {};
		\node [style=none] (88) at (27.75, -1.25) {};
		\node [style=none] (89) at (27, 5) {};
		\node [style=black vertex] (90) at (27.15, 3.8) {};
	\end{pgfonlayer}
	\begin{pgfonlayer}{edgelayer}
		\draw [bend left=15] (0.center) to node [midway, left] {$\frac 01$} (1.center);
		\draw [style=bluesolid, bend right=15] (20.center) to node [midway, above left=-0.2] {$\frac 10$} (21.center);
		\draw [style=rededge] (28.center) to node [midway, below] {$\frac 11$} (29.center);
		\draw [style=snakey] (45.center) to (46.center);
		\draw [bend left=15] (47.center) to node [midway, right] {$\frac 01$} (48.center);
		\draw [style=bluesolid, bend right=15] (49.center) to node [midway, above right] {$\frac 10$} (50.center);
		\draw [style=rededge, bend left=15] (51.center) to node [midway, below] {$\frac 11$} (52.center);
		\draw [style=magentaedge] (58.center) to node [midway, right] {$\frac 21$} (57.center);
		\draw [style=snakey] (66.center) to (67.center);
		\draw (68.center) to node [midway, left] {$\frac 01$} (69.center);
		\draw [style=bluesolid] (70.center) to node [midway, below] {$\frac 10$} (71.center);
		\draw [style=snakey] (75.center) to (76.center);
		\draw [bend left=15] (77.center) to node [midway, right] {$\frac 01$} (78.center);
		\draw [style=bluesolid, bend right=15] (79.center) to node [midway, above right] {$\frac 10$} (80.center);
		\draw [style=rededge, bend left=15] (81.center) to node [midway, below] {$\frac 11$} (82.center);
		\draw [style=magentaedge, bend right=345] (85.center) to node [midway, below] {$\frac 21$} (84.center);
		\draw [style=greenedge] (88.center) to node [midway, left] {$\frac 32$} (89.center);
	\end{pgfonlayer}
\end{tikzpicture}}
\caption{A sequence of blowups of $\P^1\times\P^1$. Divisors $E_{\frac ab}$ labelled by their rational parameter $\frac ab$.}
\label{fig:threehalves}
\end{center}
\end{figure}
\end{ex}

The above example shows a very useful property: The `abscissa' of exceptional curves never changed under blowup and proper transform (up to reciprocal). $E_1$ retained $\frac{x}{y}$ and $E_2$ retained $\frac{x^2}{y}$, despite various blowups happening since their creation. One can also observe that the ordinates changed, but only by multiplying (or dividing) by the abscissa.

This observation suggests a parametrisation of curve components by their abscissa or by the (local germs of) curves which intersect them. In the example above we saw that each curve $E_{\frac ab}$ had an abscissa $\frac{x^a}{y^b}$. Suppose we blowup the intersection point $P = E_{\frac ab} \cap E_{\frac cd}$. Locally we have coordinates $(\frac{x^a}{y^b}, \frac{y^d}{x^c})$. It is not hard to see that the abscissa for the new exceptional curve over $P$ is $\frac{x^a}{y^b} \cdot \frac{x^c}{y^d} = \frac{x^{a+c}}{y^{b+d}}$, so we would label it $E_{\frac {a+c}{b+d}}$. This shows how Farey addition plays a key r\^ole in the parametrisation of curves on smooth surfaces. 

This monomial parametrisation is exactly the one which appears in \cite{Fav}. The Farey addition for smooth blowups was noticed in \cite{FJ04}. In general, this geometric parametrisation becomes more difficult to follow when we leave the purely satellite or monomial context. For one, the dual graph grows from a line into a tree, so the linear ordering by Farey parameters will not suffice alone. Abscissae are given by polynomials or more locally by Puiseux germs. For instance, the monomial abscissa $\frac{x^a}{y^b} = C$ can be rewritten as $y = c x^\frac ab$, (where $C = c^{-b}$). 

Given a smooth point $p \in X$, there always exists a unique point blowup $\pi : \hat X \to X$ of $p$ which produces an exceptional curve, say $E \cong \P^1_\k$; this $E$ is smooth, rational and has self-intersection $E \cdot E = E^2 = -1$. Furthermore, $\pi^*C = \hat C + E$ and $\hat C \cdot E = 1$ where $\hat C$ is the proper transform of $C$ which is non-singular at $p$. See \cite[Proposition 3.2]{Hart} and \cite[Proposition 3.6]{Hart}. The converse is also true.

\begin{thm}[Castelnuovo]\label{thm:castelnuovo}
 Let $Y$ be a smooth projective surface and $E \cong \P^1$ be a smooth curve on $X$ such that $E^2 = -1$. Then there exists a smooth surface $X$, a birational morphism $\pi : Y \to X$, and a point $p \in X$ such that $Y$ is isomorphic via $\pi$ to the point blowup of $X$ with centre $p$.
\end{thm}
One can also use the push-pull formula to check that for any $D$ passing through $p$, we have $\hat D^2 = D^2 - 1$, where $\hat D$ is the proper transform of $D$. We shall often abuse notation and give a divisor the same label as its proper transforms.

Beyond smooth models, the divisors in our models are $\Q$-Cartier and we can rely on an intersection form which obeys the projection formula 
\[\pi^* C \cdot \hat D = C \cdot \pi_*\hat D.\]%

\begin{thm}[Adjunction Formula]\label{thm:adjunction}
 Let $X$ be a proper normal surface and $K$ a $\Q$-Cartier canonical divisor with $\omega_X = \bO_X(K)$. Then for every effective Cartier divisor $C \subset X$, we have $\omega_C \cong \omega_X(C)|_C$. In particular
 \[(K + C) \cdot C = 2p_a(C) - 2,\] where $p_a(C)$ is the arithmetic genus of $C$.
\end{thm}

\subsection{Skew Products on a Ruled Surface}

\begin{defn}
 We say that a (normal) surface $X$ is a \emph{birationally ruled surface over $B$}, iff $X$ is a $2$ dimensional variety with a dominant rational map $h : X \dashto B$ such that $h^{-1}(b) \cong \P^1$ for all but finitely many $b \in B$.
\end{defn}

By \cite[\S V.2]{Hart}, equivalently one can simply say $X$ is birational to $B \times \P^1$, where $B$ is a smooth curve. If $h$ is not a morphism, we can blowup $X$ finitely many times to resolve these indeterminate points. For the rest of this \thesisarticle{thesis}{paper}, whenever we write that $h : X \to B$ is birationally ruled, we will assume the fibration $h$ is a morphism.

\begin{defn}
 Let $X$ be a normal projective surface, then we say $\phi: X \dashrightarrow X$ is a \emph{(rational) skew product over $B$} if and only if $\phi$ is a dominant rational map, there exists a projective curve $B$ such that $h : X \to B$ is a birationally ruled surface, and the following diagram of rational maps commutes.
\[
\begin{tikzcd}
 X \arrow[dashed]{r}{\phi} \arrow[swap]{d}{h} & X \arrow{d}{h} \\
 B \arrow[swap]{r}{\phi_1} & B
\end{tikzcd}
\]
\end{defn}

We also call $B$ the \emph{base curve}. The word `rational' is used to stress that $\phi$ may not be a morphism. Note that the map $\phi_1$ must always be a morphism since $B$ is $1$-dimensional. 

\begin{lem}\label{fibretofibre}
 Let $\phi$ be a skew product on $X$ over $B$.
\begin{itemize}
 \item If $\phi$ contracts the curve $C$ in $X$. Then $C \subseteq h^{-1}(z)$ for some $z \in B$.
 \item If $\phi(p) = C$ (i.e.\ $p \in I(f)$) then $C \subseteq h^{-1}(z)$ for some $z \in B$.
\end{itemize}
 \end{lem}
 
\begin{proof}
 Suppose $\phi$ contracts the curve $C$ in $X$. Either $h(C) = z$ or $h(C)$ is dense in $B$, in which case $h(C) = B$ since $h$ is proper. \[w = h(z) = h(\phi(C)) = \phi_1(h(C)) = \phi_1(B)\] Therefore $\phi_1$ sends every point to $w \in B$; this means that $h(\phi(p)) = \phi_1(h(p)) = w$ for every $p \in X$ so $\phi$ sends every point to $h^{-1}(w)$ and is not dominant $\contra$.
 
 Similarly if $\phi(p) = C$ then $h(C) = h(f(p)) = \phi_1(h(p))$, a point in $B$ (as $\phi_1, h$ are continuous).
\end{proof}

Consider the situation similar to a skew product, except fibres of $h$ are not connected or $B$ is not smooth. The following proposition says that there is a better choice of smooth curve where $\phi$ is a skew product.

\begin{prop}\label{prop:galois:betterskewprod}
 Suppose that $X$ is a surface, $B$ a curve, and $h : X \dashto B$ a rational map such that all but finitely many of the fibres of $h : X \dashto B$ are (possibly disconnected) rational curves. Let $\phi: X \dashrightarrow X$ be a rational map such that the following diagram commutes.
 \[
\begin{tikzcd}
 X \arrow[dashed]{r}{\phi} \arrow[swap, dashed]{d}{h} & X \arrow[dashed]{d}{h} \\
B \arrow[swap]{r}{\phi_1} & B
\end{tikzcd}
\]
Then after replacing $X$ with its smooth desingularisation $\tilde X$, we can also replace $B$ with a smooth curve $\tilde B$ and a fibration $\tilde h : \tilde X \dashto \tilde B$ which is a birational ruling of $\tilde X$ i.e. $\tilde h$ has \emph{connected fibres}, and the induced $\phi$ is a rational skew product over $B$. After further blowup, we may assume $h$ is continuous.
\end{prop}

\begin{proof}
 First we replace $X$ with its smooth desingularisation $\tilde X$. We may further blowup $\tilde X$ until the fibration over $B$ is continuous; for notational simplicity, we will assume this for the rest of the proof. This modification $\rho : \tilde X \to X$ induces a similar diagram of rational maps by conjugation $\tilde \phi = \rho \circ \phi \circ \rho^{-1}$. 
 Now we have a fibration $h' = h \circ \rho : X \to B$ with $\tilde X$ smooth.
Then by Stein Factorisation there is a curve $\tilde B$, a morphism $\tilde h : \tilde X \to \tilde B$ and a morphism $g : \tilde B \to B$ such that $\tilde h$ has connected fibres and $g$ is finite. Since $\tilde X$ is smooth and the fibres are connected, $\tilde B$ must already be smooth. Now $\tilde h : \tilde X \to \tilde B$ is a birationally ruled surface and so it has a section $s : B \to X$ by Tsen's Theorem \cite[\S V.2]{Hart}. 
 \[
\begin{tikzcd}
 \tilde X \arrow[dashed]{r}{\phi} \arrow[swap]{d}{\tilde h} & \tilde X \arrow{d}{\tilde h} \\
\tilde B \arrow[swap]{d}{g} \arrow[swap, red]{r}{\phi_1} \arrow[blue, bend left = 60]{u}{s}& \tilde B \arrow{d}{g}\\
B \arrow[swap]{r}{\phi_1} & B
\end{tikzcd}
\]
Therefore the map on $B$ we need to construct $\phi_1 : B \to B$ is given by $h \circ \phi \circ s$, and the whole diagram above commutes.
\end{proof}

Since this allows us to extract smooth skew products from singular ones, from now on we will assume that all surfaces and curves are smooth, except when stating theorems.

\subsection{The Puiseux Series}\label{sec:puiseux}
In the following few sections we discuss Puiseux series. These power series with fractional exponents allow us to split polynomials in several variables. This makes them the main field of interest for applications of our non-Archimedean theory to complex surfaces. Some proofs are deferred to \autoref{sec:puiseuxappx} and other omitted; a concise account is given in \cite{Nov}. 
 Assume the ground field $\k$ is any characteristic $0$ algebraically closed field, like the complex numbers $\C$. %
 
 \thesisarticle{}{

\begin{defn}\label{defn:puiseux:seminorm}
 Let $G$ be a group. A \emph{seminorm} is a function $\norm\cdot : G \to \R_{\ge 0}$ such that $\norm 0 = 0$, $\norm{a} = \norm{-a} \quad\forall a \in G$, and $\norm{a + b} \le \norm{a} + \norm{b}\quad\forall a, b \in G$.
\begin{itemize}
 \item This is a \emph{norm} iff additionally $\norm{a} = 0 \implies a=0$.
 \item $\norm\cdot$ is said to be \emph{non-Archimedean} iff $\norm{a+b} \le \max \set{\norm a, \norm b} \quad \forall a, b \in G$.
 \item A seminorm $\norm\cdot$ on a ring $R$ is \emph{multiplicative} iff $\norm{a\cdot b} = \norm{a}\cdot\norm{b} \quad\forall a, b \in R$.
 \item We say a field $(K, \abs\cdot)$ is \emph{non-Archimedean} iff $\abs\cdot$ is a multiplicative non-Archimedean norm on $K$. In this case we refer to $\abs\cdot$ as an \emph{absolute value}.
\end{itemize}
\end{defn}

 Let $(K, \abs\cdot)$ be a non-Archimedean field. The \emph{ring of integers} of $K$ is given by $\bO_K = \set[a \in K]{\abs a \le 1}$. This has a unique maximal ideal, $\mathcal M_K = \set[a \in K]{\abs a < 1}$. The \emph{residue field} of $K$ is the quotient field $\k = \bO_K / \mathcal M_K$. Finally, the \emph{value group}, $\abs{K^\times} \le (0, \infty)$ is the range of $\abs\cdot$ on $K^\times = K \sm \set 0$.
}

\begin{defn}
Let $\k$ be a field. Let $\Kk = \K$ denote the field of \emph{(Newton-)Puiseux series} over $\k$ with variable $x$. For $a = a(x) \in \K$ there are $n \in \N$, $m_0 \in \Z$ and coefficients $(c_m) \subset \k$ such that \[a = \sum_{m = m_0}^{\infty} c_{m}x^{\frac{m}{n}}\] with $c_{m_0} \ne 0$. The norm $\abs\cdot$ on $\K$ is given by \[\abs a = \abs x^{\frac{m_0}{n}}\] with $a$ as above.
 The completion of $\K$, the \emph{Levi-Civita field} $\hk$, is the field with elements of the form \[\gamma = \sum_{j = 0}^{\infty} c_{j}x^{r_j}, \text{ where }(r_j) \subset \Q,\ r_j \to \infty.\]
\end{defn}

The Laurent series, Puiseux series and the Levi-Civita field are non-Archimedean fields as defined above.

\begin{thm}[Puiseux's Theorem]\label{thm:puiseux}
 Let $\k$ be any characteristic $0$ field. Then the Puiseux series $\K(\overline\k)$ over $\overline \k$, is the algebraic closure of the formal Laurent series $\k((x))$.
\end{thm}


Here we explore the Galois theory of the Puiseux series $\K$ over the formal Laurent series $\k((x))$ and provide a construction that will help us write down Puiseux skew products later in \autoref{sec:galois:endo}. %

\begin{defn}
 Define $G = \mathrm{Gal}(\K/\k((x)))$ as the Galois group of $\K$ over $\k((x))$; these are the field automorphisms of $\K$ which preserve $\k((x))$.
\end{defn}

\begin{defn}
 Let $\omega = (\omega_n)_{n=1}^\infty$ be a sequence of complex numbers such that $\omega_1 \ne 0$ and $\omega_{mn}^m = \omega_n$ for every $m, n \in \N$. Then we call $\omega$ a \emph{sequence of roots}. If furthermore $\omega_1 = 1$ then we call $\omega$ a \emph{sequence of roots of unity}.
 
 Define a function $\omega^* : \K \to \K$ using $\omega$, by
 \[\sum_{m \in \Z} c_m x^{\frac{m}{n}} \longmapsto \sum_{m \in \Z} c_m \omega_n^m x^{\frac{m}{n}}.\]
 We extend this definition to $\hk$ in the obvious (continuous) way.
\end{defn}

The reader should think of a sequence of roots of unity precisely as a sequence of choices for $n$-th roots of unity $\omega_n^n = 1$ for every $n$. Moreover these are consistent in the sense of $\omega_{102}^{34} = \omega_6^2 = \omega_3$ is unambiguously a choice of cube root.

\begin{defn}\label{def:rootopsprim}
 We define the following basic operations on sequences of roots. 
 \[\omega\sigma = (\omega_n\sigma_n)_{n=1}^\infty\qquad 1_\bullet = \left(1\right)_{n=1}^\infty\qquad \omega^{-1} = \left(\omega_n^{-1}\right)_{n=1}^\infty\]
  If $\k = \C$ let $\lambda = re^{i\theta} \in \C^*$ with $0 \le \theta < 2\pi$.
 \[\lambda_\bullet = \left(r^{\frac{1}{n}}e^{\frac{i\theta}{n}}\right)_{n=1}^\infty\qquad
 \prim = \left(e^{2\pi i/n}\right)_{n=1}^\infty\]
 Otherwise for any $k \ne \C$, using the axiom of choice, we can pick a sequence of roots of unity $\prim$ such that $\prim_n$ is a primitive $n$th root for every $n$. Similarly, one can choose a consistent sequence of $n$th root for any constant $\lambda \in k = \overline k$.
\end{defn}

\begin{prop}\label{prop:rootopsprim}
 The sequence $\prim$ is a sequence of roots of unity and $\lambda_\bullet$ defines a sequence of roots (of unity if $\lambda = 1$). Let $\omega, \sigma$ be sequences of roots (of unity), then so are $\omega\sigma, \omega^{-1}$. Moreover
 $\omega\sigma = \sigma\omega$, $1_\bullet \omega = \omega$, and $\omega \omega^{-1} = 1_\bullet$. It follows that 
 $(\omega\sigma)^* = \omega^*\sigma^*$, $1_\bullet^* = 1^* = \id$ and thus $(\omega^*)^{-1} = (\omega^{-1})^*$.
 
 If $k = \C$ then $\lambda_\bullet \tilde \lambda_\bullet = (\lambda \tilde\lambda)_\bullet$ holds if and only if $\arg(\lambda) + \arg(\tilde \lambda) < 2\pi$. Otherwise $\lambda_\bullet \tilde \lambda_\bullet = (\lambda \tilde\lambda)_\bullet\prim$.
\end{prop}

Proof left as an exercise. We remark that the notation $\omega^{*n}$ or $\omega^{n*}$ is unambiguous for every $n \in \Z$.

\begin{rmk}[Warning]
 If $\k = \C$ but $\lambda \nin \R_+$ then $\lambda_\bullet \tilde \lambda_\bullet \ne  (\lambda \tilde\lambda)_\bullet$ and $(\lambda_\bullet^*)^{-1} \ne (\lambda^{-1})_\bullet^*$. For example, let $\lambda = -1 = \lambda^{-1}$ then \[\lambda_\bullet^*(\lambda_\bullet^*(x^\half)) = \lambda_\bullet^*(ix^\half) = i^2x^\half = -x^\half.\]
\end{rmk}

\begin{prop}\label{prop:rootsauto}
 Let $\omega$ be a sequence of roots. Then $\omega^*$ is a well defined field automorphism of $\K$ (and $\hk$) which is isometric and extends $x \mapsto \omega_1x$ on $\k((x))$. Hence if $\omega$ is a sequence of roots of unity, then $\omega^* \in G$ is a Galois action. Conversely, every Galois action in $G$ is the action by a sequence of roots of unity.
\end{prop}

\begin{prop}\label{prop:multiplicitydefn}
 Let $\gamma \in \hk$, then the following definitions of $m$ are equivalent.
\begin{enumerate}[label=(\roman*)]
 \item If $\gamma \in \K$ then say it has minimal polynomial $P(y) \in \k((x))[y]$, of degree $m$. Then there is a sequence of roots of unity $\omega$ ($\prim$ suffices) such that the roots of the $P$ are \[\gamma, \omega^*(\gamma), (\omega^*)^2(\gamma), \dots, (\omega^*)^{m-1}(\gamma),\] with $\gamma$ being any such root. 
 Otherwise if $\gamma \in \hk \sm \K$ then $m = \infty$.
 \item Write
 \[\gamma(x) = \sum_{j = 1}^\infty c_j x^{\frac{p_j}{q_j}}\] with $p_j, q_j$ pairwise coprime. Then 
 $m$ is the LCM of $(q_j)_{j=1}^\infty$.  \label{prop:multiplicitydefn:LCM}
 \item \label{prop:multiplicitydefn:subfield} $m$ is the smallest integer such that $\gamma(x) \in \k((x^{\frac{1}{m}})) < \K$, or $\infty$ otherwise.
 \item $m = |\Orb_G(\gamma)|$ and $\Orb_G(\gamma) = \Orb_\prim(\gamma)$.
\end{enumerate}
 \end{prop}
 
 The proof is deferred to \autoref{appx:prop:rootsauto}. 
 
 \begin{defn}\label{defn:multiplicity}
 We call the integer $m$ given in \autoref{prop:multiplicitydefn} the \emph{multiplicity} $\mm(\gamma)$ of $\gamma$.
\end{defn}

\subsection{Puiseux for Two Variables}\label{sec:twovarpuiseux}

In this discussion, let $\k((x^*)), \k((w^*))$ denote the fields of Puiseux series respectively in $x$ and $w$ over $\k$. So $\k((w^*))((x^*))$ is the field of Puiseux series in $x$ over $\k((w^*))$. Note that this is not the same field as vice versa; for example, the series $\sum_{n \ge 0} w^{-n}x^n$ belongs to this field but not $\k((x^*))((w^*))$.

Suppose that $P(w, x, y) \in \k[w, x, y]$ is a polynomial of degree $d$ with respect to $y$.
Consider $P$ as a polynomial in $y$ with coefficients which are Laurent series in $x$ over the algebraically closed field $\k((w^*))$. Then Puiseux's Theorem (\ref{thm:puiseux}) tells us that $P$ splits as \[P = \prod_{i=0}^d (y- \gamma_j)\] where $\gamma \in \k((w^*))((x^*))$. If $P$ is irreducible over $\k(x, w)$ then using the same reasoning as in \autoref{prop:multiplicitydefn} one can deduce that $\gamma \in \k((w^{\frac{1}{n}}))((x^{\frac{1}{m}}))$ with $m, n \mid d$; however, it does not follow that $m = n = d$. Allowing $P$ to have Puiseux coefficients shows that $\k((w^*))((x^*))$ and similarly $\k((x^*))((w^*))$ are algebraically closed.

\begin{ex}
 Let $\gamma = x^{\frac 13} + w^\half x^\half \in \C(w^{\frac 12})(x^{\frac 16})$. Its minimal polynomial is 
 \[y^6 - 3wxy^4 -2xy^3 + 3w^2x^2y^2 - 6wx^2y + x^2 - w^3x^3.\]
 In this case, the Galois action of $\Gal(\C((x^*))/ \C(x))$ is transitive over the roots, but \sloppy\mbox{$\Gal(\C((w^*))/ \C(w))$} has three orbits. Notice that $\gamma(w, x) = \gamma(1, w^3x)/w$. The next example is actually a further monomial transformation of this one, namely $\gamma(1, w^3x^4)/wx$.
\end{ex}

\begin{ex}
  The minimal polynomial for $\gamma = x^{\frac 13} + w^\half x \in \C(w^{\frac 12})(x^{\frac 13})$ is 
 \[P(w, x, y) = y^6 - 3wx^2y^4 -2xy^3 + 3w^2x^4y^2 - 6wx^3y + x^2 - w^3x^6\]
 Which splits both over $\C[w^\half][x]$ and $\C[w][x^{\frac 13}]$.
 \[=\left(y^3 + 3w^\half xy^2 + 3x^2y + w^{\frac 32}x^3 - x\right)\left(y^3 - 3w^\half xy^2 + 3x^2y - w^{\frac 32}x^3 - x\right)\]
 \[=\prod_{i = 0}^2 \left(y^2 - 2 \omega_3^i x^{\frac 13}y + \omega_3^{2i} x^{\frac 23} - wx^2\right)\]
 where $\omega_3$ is a cube root of unity. In particular reducing to any constant $w = c$ makes $P(c, x, y)$ factor into two irreducible cubics over $\C((x))$ despite $P$ being irreducible over $\C((w, x))$. Neither action by $\Gal(\C((x^*))/ \C((x)))$ or $\Gal(\C((w^*))/ \C((w)))$ is transitive; however, the product is a $\Z/3 \times \Z/2$ transitive action on the roots.
\end{ex}

\begin{rmk}[Warning]
We caution that the algebraic closure of $\k((w, x))$ is not $\k((w^*))((x^*))$; indeed, as above $\k((x^*))((w^*)) \ne \k((w^*))((x^*))$ are both algebraically closed. We refer the reader to \cite{Ray} for a concise discussion on this matter. In op.\ cit., Raynaud provides the example $y^2 - wx - w^2 = 0$ showing that one can have a polynomial factor in distinct ways over $\k((w^*))((x^*))$ versus $\k((x^*))((w^*))$.
\[y = \pm\left(w + \half x - \frac 18 w^{-1}x^2 + \cdots\right)\]
\[y = \pm\left(w^\half x^\half + \half w^{\frac 32}x^{-\half} - \frac 18 w^{\frac 53}x^{-\frac 32}  + \cdots\right)\]
\end{rmk}

\subsection{Berkovich Projective Line}\label{sec:backberk}

In this section we recall the bare essentials about the Berkovich projective line. We will utilise the notation from \cite[Section 3]{berkskew} as inherited from \cite{Bene}. Throughout this subsection, let $K$ be a non-Archimedean field (such as $\K$ or the $p$-adic numbers) with norm $\abs\cdot$.

We can define the open and closed disks of radius $r$ centred at $a \in K$, respectively below.
 \[D(a,r):=\set[b\in K]{|b-a|<r},\qquad\overline{D}(a,r):=\set[b\in K]{|b-a|\le r}\]

\begin{defn}
 The Berkovich affine line $\A^1_\an = \A^1_\an(K)$ is the set of non-Archimedean multiplicative seminorms on $K[y]$ extending $(K, \abs\cdot)$. (Meaning $\norm{a} = \abs a \quad \forall a \in K$.) A topology is given to $\A^1_\an(K)$ by taking the coarsest topology for which $\norm[\zeta]\cdot \mapsto \norm[\zeta]f$ is continuous for every $f \in K[y]$.
\end{defn}

There are four types of seminorm than can appear. Whenever we define a seminorm $\norm[\zeta]\cdot$, notationally we will use $\norm[\zeta]\cdot$ and $\zeta$ interchangeably. The latter will be more succinct when dealing with points on Berkovich space.

\begin{defn}[Type I Seminorm]
Given $a \in K$, we define a function called a \emph{Type I seminorm} $\norm[a]\cdot : K[[y-a]] \to [0, \infty)$ by
\[\norm[a]f = \abs{f(a)}.\]
\end{defn}

\begin{defn}[Type II/III Norm]
Let $a \in K$, $R > 0$. %
We define a function
\[\norm[\zeta(a, R)]f = \abs{c_d}R^d\]
where \[f(y) = \sum_{n} c_n (y-a)^n, \quad \wdeg_{a, R}(f) = d.\]
We say $\zeta(a, R)$ is a \emph{Type II} or \emph{Type III norm} if $R$ is rational or not, respectively.
\end{defn}

There is another type of point, Type IV, which derives from a limit of nested closed disks $\CD(a_1, r_1) \supsetneq \CD(a_2, r_2) \supsetneq \CD(a_3, r_3) \supsetneq  \cdots$ whose intersection is empty. $\zeta = \lim_{n\to\infty} \zeta(a_n, r_n)$. We shall rarely need to concern ourselves with these. A theorem of Berkovich says that these four types classify the seminorms.

To conclude our very brief construction of the Berkovich projective line, we only need to add the point at infinity, $\infty$ of Type I. Formally, we can glue two affine lines together.

\begin{defn}
  We define the \emph{Berkovich projective line} $\P^1_\an = \P^1_\an(K)$ with two charts given by $\A^1_\an(K)$ using the homeomorphism \[\A^1_\an(K)\sm\set{0} \to \A^1_\an(K)\sm\set{0}\] given by \[\norm[\frac 1\zeta]{f(y)} = \norm[\zeta]{f\left(\frac{1}{y}\right)}.\] 
 \end{defn}
 
 This space is an $\R$-tree, compact and uniquely path connected. We shall denote the closed interval in $\P^1_\an$ between $\alpha, \beta$ by $[\alpha, \beta]$ and use round brackets for open ends $(\alpha, \beta)$.
 
  In this \thesisarticle{chapter}{article} we shall mainly be concerned with Type I and II points. The Type I points derive from $\P^1(K)$ form the good endpoints to $\P^1_\an(K)$ (the bad ones being Type IV which we try to ignore). The other Type II and III points can be found on lines between Type I points. For instance the points $\zeta(a, R)$ with $0 < R < \infty$ make up the open interval $(a, \infty)$. Taking logs of radii, the map $t \to \zeta(a, \eps^t)$ is a homeomorphism to the same open interval from $\R$, which is isometric with respect to the following metric which gives a second, strong topology in $\bH$. Note that when $K = \K$ with the variable $x$, we will take the convention that $\eps = \abs x$.
  
 \begin{defn}
The set $\bH = \P^1_\an(K) \sm \P^1(K)$ is the \emph{hyperbolic space} over $K$. We define a \emph{hyperbolic metric} $d_\bH : \bH \times \bH \to [0, \infty)$ given by \[d_\bH(\alpha, \beta) = 2 \log \left(\diam(\alpha \vee \beta)\right) - \log \left(\diam(\alpha)\right) - \log \left(\diam(\beta)\right).\]
\end{defn}

The closure a closed disk $\CD(a, R) \subset \P^1$ within $\A^1_\an$ is the Berkovich closed disk \[\CD_\an(a,r):=\set[\zeta \in \A_\an(K)]{\norm[\zeta]{y-a}\le r}\] whose only boundary point is $\zeta(a, r)$. Similarly, the Berkovich open disk is defined as \[D_\an(a,r):=\set[\zeta \in \A_\an(K)]{\norm[\zeta]{y-a}<r}\] and contains the classical open disk $D(a, r)$. In general, a Berkovich open disk in $\P^1_\an$ is as above or the complement of a closed disk. Another topological definition to note is that of an annulus, which is the complement of two disks, open or closed as appropriate; for example the open annulus.
\[\set[\zeta \in \A^1_\an]{r < \norm[\zeta]{y-a} < R} = D_\an(a, R) \sm \CD_\an(a, r)\]

\begin{defn}\label{defn:berk:dirn}
 Let $\zeta \in \P^1_\an$. The connected components of $\P^1_\an \sm \set\zeta$ are called the \emph{directions}, or \emph{tangent vectors} at $\zeta$. The set of directions at $\zeta$ is denoted $\Dir\zeta$. For any $\xi \in \P^1_\an \sm \set\zeta$ we define $\vec v(\xi)$ to be the (unique) direction at $\zeta$ containing $\xi$.
\end{defn}

Around a Type II point, there is a bijective correspondence between directions and \emph{residue classes} of the field. As a key example, the Puiseux series $\Kk$ has residue field $\k$ and at $\zeta(\gamma, \abs x^n)$ there is exactly one direction for every $c \in \k$, given by $\vec v(\gamma(x) + cx^n) = D_\an(\gamma(x) + cx^n, \abs x^n)$; there is additionally $\vec v(\infty)$ which as a set is $\P^1_\an(\K) \sm \CD_\an(\gamma, \abs x^n)$. These open disks (directions) and annuli will be used repeatedly to understand multiplicities, blowups, and smoothness.

\subsection{Non-Archimedean Skew Products}\label{sec:ratskewprops}

Let us briefly recall some general properties of non-Archimedean skew products, especially pertaining to the discussion in the present \thesisarticle{chapter}{paper}. For the full account, we refer the reader to \cite[Section 3]{berkskew}.

 \begin{defn}
  Let $K$ be a field and $\Psi$ be an endomorphism of $K(y)$ extending an automorphism of $K$, i.e. the following diagram commutes:
  \[
\begin{tikzcd}
 K(y) & K(y) \arrow[swap]{l}{\Psi} \\
K \arrow[hook]{u} & K \arrow[hook]{u} \arrow{l}{\Psi_1}
\end{tikzcd}
\]
In this case we will call $\Psi : K(y) \to K(y)$ a \emph{skew endomorphism} of $K(y)$. We will typically denote the restriction $\left.\Psi\right|_K$ by $\Psi_1$.
\end{defn}

\begin{defn}[Non-Archimedean Skew Product]\label{defn:maps:skew}
 Suppose that $\Psi : K(y) \to K(y)$ is a skew endomorphism of $K(y)$ over a non-Archimedean field $(K, \abs\cdot)$ and there is a $\q$ such that $\abs{\Psi(a)} = \abs{\Psi_1(a)} = \abs{a}^{\frac 1\q}$ for every $a \in K$. Then we say $\Psi$ is \emph{\equivar} with \emph{scale factor} $\q$. %
 Given such a $\Psi$, we define $\Psi_*$, a \emph{skew product over $K$}, as follows.
\begin{align*}
 \Psi_* : \P^1_\an(K) &\longrightarrow \P^1_\an(K)\\
 \zeta &\longmapsto \Psi_*(\zeta)\\
 \text{where } \norm[\Psi_*(\zeta)]{f} &= \norm[\zeta]{\Psi(f)}^\q
\end{align*}
If $\q=1$ then we call $\Psi_*$ a \emph{simple} skew product. Otherwise, if $\q < 1$ we say it is \emph{superattracting}, and if $\q > 1$ we may say it is \emph{superrepelling}.
\end{defn}

\begin{prop}[\refproplowerstarfunctorial]\label{prop:galois:lowerstarfunctorial}
 Let $\Phi, \Psi$ be \equivar{} skew endomorphisms of $K(y)$ with scale factors $\q, \q'$ respectively. Then $\Phi \circ \Psi$ is \anequivar{} skew endomorphism with scale factor $\q\q'$, and $(\Phi \circ \Psi)_* = \Psi_* \circ \Phi _*$. Whence $( \cdot )_*$ is a contravariant functor.
\end{prop}

\begin{prop}[\refpropskewcomp]\label{prop:galois:skew:comp}
 Let $\Psi$ be \anequivar{} skew endomorphism. Then there exists a unique decomposition $\Psi = \Psi_2 \circ \Psi_1$ defined as follows. 
  \begin{align*}
 \Psi_1 : K(y) &\longrightarrow K(y) &  \Psi_2 : K(y) &\longrightarrow K(y)\\
  a &\longmapsto \Psi(a) \quad \forall a \in K & a &\longmapsto a \qquad\ \ \forall a \in K\\
  y &\longmapsto y &   y &\longmapsto \Psi(y)
  \end{align*}
This descends to a decomposition of skew products $\Psi_* = \Psi_{1*} \circ \Psi_{2*}$ where $\Psi_{2*}$ is a rational map on $\P^1_\an(K)$ and $\Psi_{1*}$ acts as $\Psi_1^{-1}$ on $\P^1(K)$. Furthermore, there is a one-to-one correspondence between \equivar{} skew endomorphisms and decompositions of skew products $\Psi \leftrightarrow \Psi_{1*} \circ \Psi_{2*}$.
\end{prop}

This decomposition is completely natural in the case of a rational skew product $\phi$. We have $\Phi = \phi^*$, so $\phi_1(x) = \Phi_1(x)$, $\phi_2(x, y) = \Phi_2(y)$ whilst $\Phi_1(y) = y$ and $\Phi_2(x) = x$. This is equivalent to the decomposition
\[\phi(x, y) = (\phi_1(x), \phi_2(x, y)) = (\phi_1(x), y) \circ (x, \phi_2(x, y)).\]
Just as we might write $\phi = \phi_1 \circ \phi_2$ or $\phi = (\phi_1, \phi_2)$, it is natural to write $\phi_{1*}$ for $\Phi_{1*}$, $\phi_{2*}$ for $\Phi_{2*}$, and moreover decompose non-Archimedean skew products as $\phi_* = \phi_{1*} \circ \phi_{2*}$.

\begin{rmk}[Warning]
 There is a correspondence between \equivar{} skew endomorphisms, their decompositions, and skew product decompositions. However, in positive characteristic, a single skew product on the Berkovich projective line has different decompositions, each descending from a different \equivar{} skew endomorphisms, due to Frobenius and anti-Frobenius actions in $\Phi_1$ versus $\Phi_2$. In particuilar, there could be $\Phi \ne \Psi$ with $\Phi_* = \Psi_*$. See \thesisarticle{\autoref{sec:separable}}{\cite[Section 3.5]{berkskew}}.
\end{rmk}

\begin{thm}%
Let $\phi^*$ be \anequivar{} skew automorphism of $K(y)$ over $K$ with scale factor $\q$. \\Then the first factor $\phi_{1*} : \P^1_\an \to \P^1_\an$  %
\begin{enumerate}
 \item is a homeomorphism on $\P^1_\an(K)$;
 \item dilates hyperbolic distances by a factor of $\q$;
 \item is the unique continuous extension of $\Psi^{-1}$ on $\P^1(K) \subset \P^1_\an(K)$;
 \item maps Berkovich points to those of the same type; and
 \item is order preserving on the poset $(\P^1_\an, \preceq)$.
 \end{enumerate}
\end{thm}

\begin{thm}[\refthmskewopencts]%
 Let $\phi_*$ be a non-constant skew product over $K$. Then $\phi_*$ is a continuous, open mapping, which is the unique continuous extension of $(\phi_1^*)^{-1} \circ \phi_2$ on $\P^1(K) \subset \P^1_\an(K)$; and it preserves the types of each Berkovich point.
\end{thm}

\thesisarticle{\autoref{sec:skew:props}}{\cite[Section 3.4]{berkskew}} goes on to detail topological, algebraic and geometric facts one might expect on the Berkovich projective line. The subsection thereafter unravels the issue of separability and Frobenius morphisms.

Next, we give a very brief overview of local degrees. However, we omit a full discussion on monomial annuli and annuloids, deferring to \thesisarticle{\autoref{sec:monoannuloids}}{\cite[Section 3.6]{berkskew}}. 
An \emph{annuloid} is like a Berkovich annulus, except we allow for (one or both of) the boundary points to be of arbitrary Type (I or IV, not just II or III). Whence, any two points $\alpha, \beta$ define a unique annuloid $A$ with this pair as a boundary; we call $(\alpha, \beta)$ the \emph{\spine{}} of $A$. We might refer to $d_\bH(\alpha, \beta)$ as the \emph{modulus} of $A$.

Next, we say a skew product $\phi_*$ is \emph{\mono} on an \annuloid{} $A$ if $\phi_*(A)$ is also an \annuloid{}; equivalently, $\phi_*$ maps $A$ properly, $\deg_A(\phi)$-to-$1$ on all points. A special property of \anannuloid{} $A$ on which $\phi_*$ is \mono{} is that it stretches $A$ by a \emph{\fmult{}} factor $\mpt A = \q d$, where $\q\ (= \frac 1n)$ is the scale factor and $d = \deg_A(\phi)$. Moreover, the stretch is uniform along the spine (and $\q$ elsewhere on $A$ if separable). 

\Mono{} \annuloid{}s can be used to define the \emph{\fmult{}}, $\mt\zeta{\bvec v}$, of $\phi_*$ in a direction. Given a point $\zeta$ and a direction $\bvec v$ at $\zeta$, there is always \anannuloid{} $A$ bounded by $\zeta$ and some $\xi \in \bvec v$ on which $\phi_*$ is \mono{}. It follows that $\phi_*$ perfectly dilates an interval emanating from $\zeta$ extending out in the direction of $\bvec v$ by some factor, namely $\mt\zeta{\bvec v} = \mpt A$. This \fmult{} matches the local degree in the direction of $\bvec v$, in the following sense \[\mt\zeta{\bvec v} = \q\deg_{\zeta, \bvec v}(\phi) = \q\deg_{\zeta, \bvec v}(\phi_2).\]
The following theorem summarises the most important \mono{} \annuloid{}s (abridged).

\begin{thm}[\restatemonoannuloidthm{}]\thesisarticle{}{\label{thm:galois:skew:ultimatemonoannuloid}}~\newline
 Let $\phi_*$ be a skew product and suppose $\phi_*$ is \mono{} on the open \annuloid{} $A$ bounded by $\alpha, \beta \in \P^1_\an$. Let $m = \mpt A$ be the \fmult{} of $\phi_*$ on $A$. The following hold:
 \begin{enumerate}
 \item $\displaystyle \forall \zeta \in [\alpha, \beta)\ \mt\zeta{\vec v(\beta)} = m$
 \item $\displaystyle \forall \zeta \in (\alpha, \beta]\ \mt\zeta{\vec v(\alpha)} = m$
 \item $\displaystyle \forall \zeta \in (\alpha, \beta)\ \mpt\zeta = m$
 \item For every $\zeta$ in the \spine{} of $A$, every direction $\bvec u \in \Dir\zeta \sm \set{\vec v(\alpha), \vec v(\beta)}$ is a good direction with image in $\phi_*(A)$ disjoint from $\vec v(\phi_*(\alpha)), \vec v(\phi_*(\beta))$. %
 \item Moreover, $\overline A \cap \phi_*^{-1}[\phi_*(\alpha) , \phi_*(\beta)] = [\alpha , \beta]$ and $\phi_* : [\alpha , \beta] \to [\phi_*(\alpha) , \phi_*(\beta)]$, is a homeomorphism between \spine{}s scaling distances by a factor of $m$, i.e. $\forall\ \zeta, \xi \in \bH \cap [\alpha, \beta]$ \[d_\bH(\phi_*(\zeta), \phi_*(\xi)) = md_\bH(\zeta, \xi).\]
\end{enumerate}
\end{thm}

From our decomposition $\phi_* = \phi_{1*} \circ \phi_{2*}$, the \fmult{} breaks down, with the (scale) factor $\q$ contributed by $\phi_{1*}$ (it scales \emph{all} distances this much) and the local degree from $\phi_{2*}$. Likewise, we further define the \emph{local \fmult{}}, of $\phi$ at $\zeta$ as $\mpt\zeta =\q\deg_\zeta(\phi)$m, where $\deg_\zeta(\phi) = \deg_\zeta(\phi_2)$ is the local degree for the Berkovich rational map $\phi_2$. The expected versions of chain rule and degree counting `multiplicity' apply. See \reflocalmults{} for details.

\begin{align*}
 \sum_{\phi_\#(\bvec v) = \bvec w} \mt\zeta{\bvec v} &= \mpt\zeta &
 \sum_{\phi_\#(\bvec v) = \bvec w} \deg_{\zeta, \bvec v}(\phi) &=\deg_\zeta(\phi) \\
 \sum_{\xi \in \phi^{-1}(\zeta)} \mpt\xi &= \mmp\phi&
 \sum_{\xi \in \phi^{-1}(\zeta)} \deg_\xi(\phi) &= \rdeg(\phi)\\
 \mpt\psi{\phi_*(\zeta)} \cdot \mpt\zeta&= \mpt[\psi \circ \phi]\zeta &
 \deg_{\phi_*(\zeta)}(\psi) \cdot \deg_\zeta(\phi) &= \deg_\zeta(\psi \circ \phi)
\end{align*}

These degrees and \fmult{}s can be computed locally using power series and reduction maps. Along \annuloid{}s, one computes a convergent Laurent series for $\phi_2$. Note that for a Type I, IV (resp.) or III point $\zeta$, we have $\mt\zeta{\bvec v} =\mpt\zeta$ for the (resp.\ either) direction $\bvec v$ at $\zeta$. At a Type II point, all the information can be inferred from reduction of $\phi_2$ after an appropriate change of coordinates.

\begin{thm}[{\cite[Theorem 7.34]{Bene}, \refthmskewreduction}]\label{thm:skew:reduction}
 Let $\phi_*= \phi_{1*} \circ \phi_{2*}$ be a skew product of relative degree $d$.
 Then $\overline \phi_2$ is nonconstant if and only if $\phi_*(\zeta(0, 1)) = \zeta(0, 1)$. In that case,
\[\deg_{\zeta(0, 1)}(\phi) = \deg(\overline \phi_2).\]
 and define $T$ to be the set of bad directions. Then
\begin{enumerate}[label=(\alph*), ref=\theenumi]
 \item $\bvec v \in T$ if and only if $\bvec v$ contains both a zero and pole of $\phi_*$,
 \item $T$ is a finite set,
\item $\phi_*(\bvec v) = \P^1_\an$ for each direction $\bvec v \in T$, and
\item $\phi_\#(\bvec v) = \phi_*(\bvec v) = \vec v(\phi_*(\xi))$ for all directions $\bvec v$ at $\zeta(0, 1)$ not in $T$ and all $\xi \in \bvec v$. For such directions $\bvec v$, we have
\[\deg_{\zeta(0, 1), \bvec v}(\phi) = \wdeg_{\bvec v}(\phi_2)\]%
\end{enumerate}
Finally, $\phi_*$ has explicit good reduction, i.e.,\ $\deg(\overline \phi_2) = d$, if and only if $\overline \phi_2$ is
nonconstant and $T = \emp$.
\end{thm}


\newcommand{\pmat}[4]{\begingroup 
\renewcommand{\arraystretch}{1}
\begin{pmatrix}
 #1 & #3 \\
 #2 & #4
\end{pmatrix}
\endgroup}
\newcommand{\vmat}[4]{\begingroup 
\renewcommand{\arraystretch}{1}
\begin{vmatrix}
 #1 & #3 \\
 #2 & #4
\end{vmatrix}
\endgroup}
\newcommand{\pvec}[2]{\begingroup 
\renewcommand{\arraystretch}{1}
\begin{pmatrix}
 #1 \\
 #2
\end{pmatrix}
\endgroup}

\subsection{Farey Arithmetic}\label{sec:fareyarithmetic}

For brevity we shall denote matrices with round brackets and their determinants with square brackets. Meaning 
\[\det \pmat abcd = \vmat abcd = ad-bc\]

\begin{defn}
 We say that $\frac ab, \frac cd$ are a \emph{Farey adjacent} or \emph{neighbours} iff $\vmat abcd = \pm 1$, assuming they were written in lowest terms.
 
  We will say that \[\frac {a_1}{b_1} < \frac {a_2}{b_2} < \cdots < \frac {a_n}{b_n}\] form a \emph{Farey sequence} iff $\frac {a_i}{b_i}$ and $\frac {a_{i+1}}{b_{i+1}}$ are Farey adjacent for each $i = 1, \dots, n-1$. We will say a sequence of fractions is a \emph{complete Farey sequence of order $B$ between $\frac {a_1}{b_1}$ and $\frac {a_n}{b_n}$} iff these fractions forms the whole collection of rational numbers in $[\frac {a_1}{b_1}, \frac {a_n}{b_n}]$ with denominator at most $B$.%
\end{defn}

\begin{ex}
This is an example of a Farey sequence illustrated on a number line.
 \begin{figure}[H]
\def\ExtraFracList{1/2, -1/2, 3/5, 2/3, 5/2, 7/3, 9/4, 1/3}
\fareyline{0.9\linewidth}{-2}{3}{1}
\stepcounter{figure}
\end{figure}
\end{ex}

It is not immediate that a complete Farey sequence of order $B$ (without adjacency assumption) between two fractions is also a Farey sequence (with adjacency). However, this follows by a nice inductive exercise or directly from \autoref{cor:fareyaux:complete} below. This phenomenon was the original interest in `Farey series'.

\begin{ex}
  One can check that every consecutive pair of fractions are Farey adjacent.
  \begin{figure}[H]
\fareyline{0.9\linewidth}{0}{3}{5}
\caption{The complete Farey sequence between $0$ and $3$ of order $5$.}
\end{figure}
\end{ex}

These definitions may be extended to allow for zero denominators, by considering coprime pairs of integers i.e.\ $\frac ab \sim (a, b) \in \Z \times \N$. Often authors will define `Farey sequence' as we have a complete Farey sequence (of some order) between $\frac 01$ and $\frac 11$; however, many of the well-known facts hold in this generality.
It is easy to see that if $\frac ab < \frac cd$ are Farey adjacent, then $\vmat abcd = -1$ and $\frac cd - \frac ab = \frac 1{bd}$. 

Central to the subject of Farey sequences is the \emph{mediant}, or \emph{Farey sum}, defined as follows. Given two fractions $\frac ab < \frac cd$ with no other assumptions, write \[\frac ab \oplus \frac cd = \frac{a+c}{b+d}.\]
It is a simple exercise to show the following. (Hint: $\vmat ab{a+c}{b+d} = \vmat {a+c}{b+d}cd = \vmat abcd < 0$.)

\begin{prop}[classical]\label{prop:fareyaux:mediant}
Suppose $\frac ab < \frac cd$ be rational numbers in lowest terms; then
 \[\frac ab < \frac ab \oplus \frac cd < \frac cd.\] If additionally $\frac ab, \frac cd$ are Farey adjacent, then they are both also Farey adjacent to $\frac ab \oplus \frac cd$.%
\end{prop}

\begin{ex}
\begin{align*}
  \frac 12 &= \frac 01 \oplus \frac 11 & \frac 13 &= \frac 01 \oplus \frac 12\\[2mm]
  \frac 25 &= \frac 13 \oplus \frac 12 & \frac 37 &= \frac 25 \oplus \frac 12\\[2mm]
  \frac 5{12} &= \frac 25 \oplus \frac 37 & \frac 7{17} &= \frac 25 \oplus \frac 5{12}
\end{align*}
\begin{figure}[h]
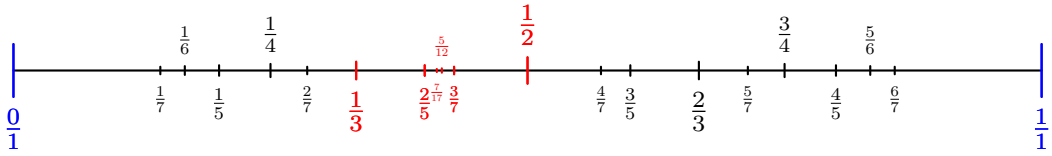

\def\ExtraFracList{0/1,1/1,1/2,1/3,2/5,3/7,5/12, 7/17}
  \SetFracColor{0}{1}{blue}
  \SetFracColor{1}{1}{blue}
  \SetFracColor{1}{2}{red}
  \SetFracColor{1}{3}{red}
  \SetFracColor{2}{5}{red}
  \SetFracColor{3}{7}{red}
  \SetFracColor{5}{12}{red}
  \SetFracColor{7}{17}{red}
\noindent\fareyline{0.9\linewidth}017
\caption{We can find $\frac 7{17}$ by repeated Farey addition between $\frac 01$ and $\frac 11$.}
\end{figure}
\end{ex}

The study of the mediant and Farey sequences is practically ancient, with the above observation apparently being made by Plato in the \emph{Parmenides} dialogue, mid fourth century BCE. Names of mathematicians who have worked on these sequences this abound. However, they take their name from John Farey, a British geologist who essentially conjectured that the mediant is the fraction with smallest denominator between two adjacent fractions, and hence that one can recursively produce any complete Farey sequence (and indeed \emph{all} fractions) by taking successive mediants. These latter properties were actually proven by Charles Haros 14 years earlier. For more history and background, see \cite{Dickson, Fowler, Guthery}.

\begin{prop}[Haros]\label{prop:fareyaux:haros}
 Suppose that $\frac ab < \frac cd$ are Farey adjacent fractions, and $\frac pq \in (\frac ab, \frac cd) \cap \Q$ is any other (each in lowest terms). Then
 \[\pvec pq = \vmat cdpq \pvec ab + \vmat pqab \pvec cd.\]
 In particular, $q \in \N^+\langle b, d\rangle$, so $q \ge b+d > \max(b, d)$.
\end{prop}

\begin{cor}
 Suppose that $\frac ab < \frac cd$ are Farey adjacent fractions in lowest terms. Amongst the fractions in $(\frac ab, \frac cd) \cap \Q$, $b+d$ is the smallest possible denomenator, and this is attained uniquely by $\frac ab \oplus \frac cd = \frac{a+c}{b+d}$.
\end{cor}

We will indulge in the lemmas and proofs as far as they aid the present article. Certain techniqiues will extend to the tree setting, whereas this is merely a numberline. The following algebraic lemma is an exercise in linear algebra and will hasten our progress. In fact, the formula in \autoref{prop:fareyaux:haros} is obtained simply by combining it with the fact $\vmat abcd = -1$. 

\begin{lem}
 Let $R$ be a ring and suppose $a, b, c, d, e, f \in R$. Then
 \[\vmat abcd \pvec ef + \vmat cdef \pvec ab + \vmat efab \pvec cd = \pvec 00.\]
\end{lem}

In practice, we are more likely to be given a Farey sequence $\frac ab < \frac pq < \frac cd$ with $q > \max(b, d)$ and wish to know, conversely to \autoref{prop:fareyaux:mediant} that $\frac ab$ and $\frac cd$ are Farey adjacent. %

\begin{prop}\label{prop:fareyaux:adjtriple}
 Suppose that $\frac ab < \frac pq < \frac cd$ is a Farey sequence written in lowest terms, meaning both successive pairs of fractions are Farey adjacent. Then $\frac pq = \frac ab \oplus \frac cd$ and TFAE
\begin{enumerate}[label = (\roman*)]
 \item $p = a + c$\label{item:fareyaux:adjtriple:p}
 \item $q = b + d$\label{item:fareyaux:adjtriple:q}
 \item $\displaystyle \frac ab$ and $\displaystyle \frac cd$ are Farey adjacent.\label{item:fareyaux:adjtriple:adj}
 \item $q > \max(b, d)$.\label{item:fareyaux:adjtriple:gr}
\end{enumerate}
\end{prop}

\begin{proof}
From the algebraic lemma, 
  \[\vmat cdab \pvec pq = \vmat cdpq \pvec ab + \vmat pqab \pvec cd.\]
  Since the fractions form a Farey sequence $\vmat cdpq = \vmat pqab = 1$; define $n = \vmat cdab$. Whence,
  \[n \pvec pq = \pvec ab + \pvec cd,\]
 $np = a+c$, $nq = b+d$, and so \[\frac{a+c}{b+d} = \frac{np}{nq} = \frac pq.\] Now it is clear that \ref{item:fareyaux:adjtriple:adj} $\iff$ $n=1$ $\iff$ \ref{item:fareyaux:adjtriple:p} $\iff$ \ref{item:fareyaux:adjtriple:q}. In the affirmative case, $q = b+d > \max(b, d)$. Otherwise, if $n \ge 2$ then $q = \frac{b+d}n \le \frac{b+d}2 \le \max(b, d)$. Thus, \ref{item:fareyaux:adjtriple:q} is equivalent.
\end{proof}

\begin{prop}\label{prop:fareyaux:parents}%
 Suppose $p, q \in \Z$ are coprime with $q \ge 1$. On each side of $\frac pq$ there exist unique \emph{`parents'} $\frac ab < \frac pq < \frac cd$ which form a Farey sequence and are such that either $q = b = d = 1$ or $q > \max(b, d)$, the latter implying that $\frac ab$ and $\frac cd$ are themselves Farey neighbours.
\end{prop}

\begin{proof}
 B\'ezout's Theorem provides integers $a, b$ to the equation $bp - aq = 1$, i.e.\ $\vmat abpq = -1$. Now, by replacing $\pvec ab$ with $\pvec ab + n \pvec pq$ if and as necessary, we can ensure also that $1 \le b \le q$. One can check that this pair $(a, b)$ is unique. Note that if $b = q$, then $q \mid bp - aq = 1$; therefore $q > 1 \iff q > b$. Similarly, there exists a unique integer pair $(c, d)$ with $\vmat pqcd = -1$ and $q = d = 1$ or $1 \le d < q$. 
\end{proof}

\begin{thm}\label{prop:fareyaux:iff}
 Suppose that $\frac ab < \frac cd$ are fractions in lowest terms. Then they are adjacent if and only if every fraction $\frac pq \in \left(\frac ab, \frac cd\right)$ has a larger denominator, i.e.\ $q > \max(b, d)$.
\end{thm}

\begin{proof}
The forward direction is given by \autoref{prop:fareyaux:haros}. Assume that $b \ge d$; the other case is a similar exercise. 
For $\frac ab$, \autoref{prop:fareyaux:parents} provides a Farey neighbour and parent $\frac ab < \frac pq$ with $1 \le q \le b = \max(b, d)$. 
 We also want to show $\frac pq < \frac cd$. A priori, these are not equal because $\vmat abcd \ne -1$. Suppose, otherwise, that $\frac ab < \frac cd < \frac pq$. Then by \autoref{prop:fareyaux:haros} we have $d \ge b+q > b$, contradicting our initial assumption. 
\end{proof}

\begin{cor}\label{cor:fareyaux:complete}
 Any complete Farey sequence is a Farey sequence. 
\end{cor}

\begin{defn}
 We say that a Farey sequence $S = \left(\frac {a_i}{b_i}\right)_{i=1}^N$ is \emph{generated by Farey addition between} a Farey subsequence $T = \left(\frac {c_i}{d_i}\right)_{i=1}^M$ iff there is a succession 
 $T = S_M \subset S_{M+1} \subset \cdots \subset S_N = S$ such that for each $M \le i < N$, $S_{i+1} = S_i \cup \set{\frac ab \oplus \frac cd}$ for some $\frac ab, \frac cd \in S_i$. In other words, $S$ can be generated by repeated Farey additions between existing fractions, beginning with $T$. Naturally, we always require mediants to be taken between fractions in lowest terms.
\end{defn}

This is the most general form of the Haros result for Farey sequences which need not be complete.

\begin{prop}
 Let \[\frac {a_1}{1} < \frac {a_2}{b_2} < \cdots <  \frac {a_{n-1}}{b_{n-1}} < \frac {a_n}{1}\] be a Farey sequence (note the end fractions are integers $a_1 < a_n$). Then this is generated by Farey addition between the subsequence $a_1, a_1+1, \dots, a_n - 1, a_n$. 

 Moreover, every Farey sequence of the form \[\frac{-1}0 < \frac {a_1}{b_1} < \frac {a_2}{b_2} < \cdots <  \frac {a_{n-1}}{b_{n-1}} < \frac {a_n}{b_n} < \frac10\] is generated by Farey addition between $\frac{-1}0, \frac a1, \frac 10$ for some $a \in \Z$.
\end{prop}

\begin{proof}
 Induction on maximum denominator. If the maximum is $1$, it is not hard to see that neighbouring terms in the sequence must be consecutive integers. Otherwise, let $\frac pq$ be a fraction in $S$ with maximal denominator and suppose $\frac ab < \frac pq < \frac cd$ are its two neighbours in $S$. Then $q \ge \max(b, d)$ but $q$ cannot equal $b$ (or $d$), because then $q > 1$ divides $\vmat abpq = 1$ (or $\vmat pqcd$ respectively). So $q > \max(b, d)$, and thus $\frac ab$ and $\frac cd$ are adjacent by \autoref{prop:fareyaux:adjtriple}. Therefore $S \sm \set{\frac pq}$ is a Farey sequence, which by induction is generated by Farey addition over $a_1, a_1+1, \dots, a_n - 1, a_n$. We leave details of the addendum as an exercise; note that necessarily $b_1 = b_n = 1$ and for instance $\frac a1 \oplus \frac 10 = \frac {a+1}1$.
\end{proof}

Finally, we observe that complete Farey sequences are very easy to produce and characterise. Indeed, let $M < N$ be integers and fix $B \in \N^+$. Then the set $\set[\frac ab]{1 \le b \le B, M \le \frac ab \le N}$ always forms a complete Farey sequence of order $B$ between $M$ and $N$. To see this, suppose that $\frac ab < \frac cd$ are consecutive fractions in the above set with respect to $<$; then for any rational number $\frac pq$ between them, we must have $q > B \ge \max(b, d)$; now by \autoref{prop:fareyaux:iff}, $\frac ab < \frac cd$ must be Farey adjacent.

\section{Space}\label{sec:space}


In this section, we focus on the Berkovich projective line $\P^1_\an(\K)$ over the Puiseux (Levi-Civita) series. This space will function as a universal dual graph for divisors in a subset of a surface. Essential definitions were given \autoref{sec:backberk}, with significantly greater detail in \cite{Bene}. The main goal is to understand the Puiseux parametrisation of Type II points. Note the if $\zeta(\gamma, r)$ is a Type II point of $\P^1_\an(\K)$ then $r$ is `rational', meaning $r \in \abs\K = \set[\abs x^t]{t \in \Q}$. Therefore we can always write Type II points as $\zeta(\gamma, \abs x^{\frac ab})$ with $a \in \Z$, $b\in \N$. These numbers $a, b$ along with the multiplicity of $\gamma$ will play a crucial r\^ole later. Many details of this subsection can be found in Favre \& Jonsson \cite{FJ04}, such as their Section 3.4.

\subsection{Multiplicity}\label{sec:galois:multiplicity}

Previously in \autoref{defn:multiplicity} we defined a number $\mm(\gamma)$ for $\gamma \in \hk$ and in \autoref{prop:multiplicitydefn} we saw it was the maximum length of an orbit under $G$, conveniently given by $|\Orb_\prim(\zeta)|$.
Now however we shall want $G$ to act as in \autoref{prop:GaloisactiononP1}. On $\gamma \in \hk$ this changes little since $\omega_*(\gamma) =  (\omega^*)^{-1}(\gamma)$, so $\Orb_G(\gamma) = \Orb_\prim(\gamma)$ is the same. Note that $\omega_*(\infty) = \infty$ for any $\omega^* \in G$.

\begin{defn}%
 Let $\zeta \in \P^1_\an(\hk)$. Define its \emph{multiplicity} $\mm(\zeta)$ to be $|\Orb_\prim(\zeta)|$. We also define the \emph{multiplicity} of a subset $U \subset \P^1_\an$, $\mm(U) = \min_{\zeta \in U} \mm(\zeta)$ as the minimum multiplicity in $U$.
\end{defn}

\begin{prop}\label{prop:multiplicity}
Let $\zeta \in \bH \subset \P^1_\an$.
\begin{enumerate}[label = (\roman*)]
 \item \label{prop:multiplicity:II} If $\zeta$ is Type II or III, then we can write $\zeta = \zeta(a, r)$ such that $\mm(a) = \mm(\zeta)$. Moreover $a$ can be obtained by truncating any Puiseux series $b \in \CD(a, r)$ to order $r$, therefore $\mm(a) = \min_{b \in \CD(a, r)} \mm(b)$ and $a \in k\left(x^\frac{1}{\mm(a)}\right)$ has a finite expansion.
 \item \label{prop:multiplicity:IV} If $\zeta$ is Type IV, then $\mm(\zeta) = \infty$.
\end{enumerate}
\end{prop}

\begin{proof}
 \ref{prop:multiplicity:II} We need to determine the period of $\zeta$ under the action of $\prim_*$. If $\zeta = \zeta(b, r)$ then $\prim_*^k(\zeta) = \zeta(\prim_*^k(b), r)$. So we want the minimal $k$ such that $|\prim_*^k(b) - b| \le r$. Write $b = a + c$ where $a$ is the $\bO(r)$ truncation of $b$ and $c$ is the remainder, then \[\prim_*^k(b) - b = (\prim_*^k(a) - a) + (\prim_*^k(c) - c).\] $|\prim_*^k(c)|, |c| \le r$ for any $k$, and $|\prim_*^k(a) - a| \le r$ iff $\prim_*^k(a) = a$; therefore $|\prim_*^k(b) - b| \le r$ iff $\prim_*^k(a) = a$. This proves $\mm(a) = \mm(\zeta)$.
 
 \ref{prop:multiplicity:IV} Suppose $m = \mm(\zeta) \in \N$.
 Describe $\zeta$ as a limit of Type II points $\zeta_n = \zeta(a_n, r_n)$ with $\mm(\zeta_n) = \mm(a_n)$. WLOG shrink $r_n$ to $|a_n - a_{n+1}|$ to minimally have $\bar D_\an(a_{n+1}, r_{n+1}) \subset \bar D_\an(a_n, r_n)$, and remove terms in the sequence to ensure that $r_n$ is strictly decreasing.
 
We have $\zeta \in \bar D_\an(a_n, r_n)$ and $\prim_*^m(\bar D_\an(a_n, r_n))$ is a closed disk $\bar D_\an(\prim_*^m(a_n), r_n)$ containing $\prim_*^m(\zeta) = \zeta$. Since these disks overlap and have the same radius they are equal, proving that $\prim_*^m(\zeta_n) = \zeta_n$. Therefore $\mm(\zeta_n) = \mm(a_n) \mid m$. %

By \autoref*{prop:multiplicitydefn} \ref{prop:multiplicitydefn:subfield}, $a_n \in k((x^\frac 1m))$ for every $n$. Since $r_n = |a_n - a_{n+1}|$ we must have $r_n = |x|^\frac {p_n}{m}$ for some $p_n \in \Z$. But if $r_n$ is strictly decreasing then (for large $n$) we also have $p_n$ strictly increasing. Hence $p_n \to \infty$ and $r_n \to 0$.
\end{proof}

\begin{rmk}
 The best way to think about \autoref{prop:multiplicity} is that a point $\zeta$ has the minimum multiplicity among Type I points $a \preceq \zeta$ below it in the canonical order (with $\min \emp = \infty$). These are the simplest Type I points describing $\zeta$ geometrically, in fact we see in the method of proof for Type IV points, that one could write down a series (not in $\hk$) \[\sum_j c_j x^\frac{p_j}{q_j}\] with $\frac{p_j}{q_j}$ strictly increasing, but not to infinity. This forces the $q_n$ to be unbounded in $\N$, and hence the multiplicity ought to be infinite.
\end{rmk}

\begin{rmk}
 This proposition shows the equivalence between the multiplicity $\mm$ in this \thesisarticle{thesis}{paper} and the multiplicity defined in Favre \& Jonsson \cite[\S 1.4, \S 3.4, \S 6.3.2]{FJ04}. %
\end{rmk}

\begin{prop}
 $\mm(\zeta) = |\Orb_G(\zeta)|$
\end{prop}

This means $\prim_*$ always generates $\Orb_G(\zeta)$. This is why we like $\prim$ - it's a universal generator for $G$ acting on $\P^1_\an$.

\begin{proof}
The result is obvious if $\mm(\zeta) = \infty$. By \autoref{prop:multiplicity} this deals with Type IV points. In the case of Type I points this was proved in \autoref{prop:multiplicitydefn}. Finally we consider $\zeta = \zeta(a, r)$ Type II/III with $\mm(a) = \mm(\zeta)$. Let $\omega^* \in G$ be arbitrary. Since the result holds for $a$, $\omega_*(a) = \prim_*^k(a)$ for some $k \in \N$. \[\omega_*(\zeta(a, r)) = \zeta(\omega_*(a), r) = \zeta(\prim_*^k(a), r) = \prim_*^k(\zeta(a, r))\] This proves that $\Orb_G(\zeta) \subseteq \Orb_\prim(\zeta)$ with the reverse inclusion obvious.
\end{proof}

Recall that $\zeta \prec \xi$ iff $\xi \in (\zeta, \infty]$, and for two Type II/III (or I) points this is most easily understood by $\zeta(a, r) \prec \zeta(a, R)$ where $r < R$.

\begin{prop}\label{prop:multiplicityorder}
 Let $\zeta \prec \xi$ be points in $\P^1_\an$, then $\mm(\xi) \mid \mm(\zeta)$, or $\mm(\zeta) = \infty$.
\end{prop}

\begin{proof}
 Assume $\mm(\zeta) \in \N$. So $\zeta$ is Type I, II, or III, so write $\zeta = \zeta(a, r)$ where $r \ge 0$ and $\mm(a) = \mm(\zeta)$ (using \autoref*{prop:multiplicity} \ref{prop:multiplicity:II}). If $\xi \succ \zeta$ then $\xi$ must be either $\infty$ or $\zeta(a, R)$ with $R > r$. If $\xi = \infty$ then $\mm(\xi) = 1$ which divides anything. If $\xi = \zeta(a, R)$ then $\prim_*^{\mm(a)}(\zeta(a, R)) = \zeta(\prim_*^{\mm(a)}(a), R) = \zeta(a, R)$ which proves that $\mm(\xi) \mid \mm(a) = \mm(\zeta)$.
\end{proof}

\begin{rmk}
 Equivalently, if we consider the poset $(\N_+ \cup \{\infty\}, <_m)$ where $a \le_m b \iff a \mid b$ or $b = \infty$, then the proposition shows that \[m : (\P^1_\an, \prec) \longrightarrow (\N \cup \{\infty\}, <_m)\] is an order reversing map of posets (but highly non-injective).
\end{rmk}

Recall that $\Hull(S)$ denotes the convex hull of a subset $S \subseteq \P^1_\an$.
 
\begin{defn}
 Define the \emph{multiplicity $n$ subtree} by \[\T_n = \set[\zeta \in \P^1_\an(\hk)]{\mm(\zeta) \mid n}.\]
\end{defn}

It is clear from the definition that $\T_m \subseteq \T_n \iff m \mid n$.

\begin{prop}\label{prop:multsubtreehull}
 The subset $\T_n$ is a closed connected subtree of $\P^1_\an(\hk)$. Moreover, \[\T_n = \Hull \left(\set[a \in \P^1(\K)]{\mm(a) \mid n}\right).\]
\end{prop}

\begin{proof}
 First we remark that $[0, \infty] \subset \T_1$. Indeed by \autoref*{prop:multiplicity} \ref{prop:multiplicity:II} $\mm(\zeta(0, r)) = \mm(0) = 1$, and $\mm(\infty) = 1$. $\infty \in \T_n$ for every $n$, so to show $\T_n$ is a convex hull we first show that $[a, \infty] \subset \T_n$ for any $a \in K$ with $\mm(a) \mid n$, then use the fact that a path between any two elements of $\T_n$ travels along such closed intervals.
 
 Suppose $a \in K$ and $\mm(a) \mid n$. If $\zeta \in [a, \infty]$ then $\zeta > a$; by \autoref{prop:multiplicityorder} $\mm(\zeta) \mid \mm(a)$, hence also $\mm(\zeta) \mid n$, so $\zeta \in \T_n$. We have shown $[a, \infty] \subset \T_n$.
 
 Let $\zeta, \xi \in \T_n$, $a \preceq \zeta$ be a Type I point with $\mm(a) = \zeta$, and similarly $b \preceq \xi$ where $\mm(b) = \mm(\xi)$. Then either (WLOG) $\zeta \preceq \xi$ and $[\zeta, \xi]$ is the unique path connecting them, otherwise $[\zeta, \zeta \vee \xi] \cup [\xi, \zeta \vee \xi]$ is the unique path, where $\zeta \vee \xi$ is the least upper bound to $\zeta$ and $\xi$. In either case, the path is a subset of $[a, \infty] \cup [b, \infty]$. By the first part of the proof, this is entirely contained in $\T_n$.
 
 It remains to show that $\T_n$ is closed. $\prim_*^n(\zeta) = \zeta$ for any $\zeta \in \T_n$. Suppose that $\zeta_k \in \T_n$ is a sequence converging to $\zeta$. By continuity of $\prim_*$, \[\prim_*^n(\zeta) = \prim_*^n(\lim_{k \to \infty}\zeta_k) = \lim_{k \to \infty} \prim_*^n(\zeta_k) = \lim_{k \to \infty} \zeta_k = \zeta.\] Therefore $\mm(\zeta) \mid n$.
\end{proof}

\begin{cor}\label{cor:multiplicitylowersemicont}
 The function $\mm : \P^1_\an(\hk) \longrightarrow \N_+\cup\{\infty\}$ is lower semicontinuous, both in the usual order on $\N$, and with respect to the multiplicative order $(\N_+\cup\{\infty\}, <_m)$.
\end{cor}

\begin{prop}\label{prop:multsubtreevalency}
The subtree $\T_n$ is an infinite tree with discrete branching in the following sense: every (non-endpoint) vertex $\zeta \in \T_n$ of valency at least $3$ is of Type II and in every direction at $\zeta$ there is an edge of length $1/n$ (in the sense that $\T_n$ has no further branching).
 
\begin{itemize}
 \item The set of non-endpoint vertices is of the form 
 \[\set[\zeta \in \bH]{\zeta = \zeta(a, |x|^\frac pq),\ \mm(a), q \mid n},\]
 hence $d_\bH(\zeta_1, \zeta_2) \in \tfrac 1n \N$ for any two $\zeta_1, \zeta_2$ in the set.
 \item  Let $\zeta = \zeta(a, |x|^\frac pq)$ with $\mm(a) = \mm(\zeta) = m$, $\GCD(p, q) = 1$, and set $\mg = \LCM(\mm(\zeta), q)$. Then
\begin{enumerate}[label = (\roman*)]
 \item There is a $\zeta' \in (\zeta, \infty]$ such that $\mm(\xi) = m$ for every $\xi \in [a, \zeta')$.
 \item Let $c \in \k^*$. Then $\mm(\xi) \ge \mg$ for every $\xi \in \vec v(a + cx^\frac pq)$, and $\mm(\xi) = \mg$ for every $\xi \in [a + cx^\frac pq, \zeta)$.
\item In particular, $\zeta$ has two directions with possibly lower multiplicities, $\mm(\vec v(\infty)) = 1$, $\mm(\vec v(a)) = m$, and $\mm(\bvec v) = \mg$ for every other direction $\bvec v \in \Dir \zeta$.
\end{enumerate}
\end{itemize}
\end{prop}

\begin{proof}
 First we will show that $\zeta$ is a vertex of $\T_n$ if it is of the form $\zeta(a, r)$ with $a \in k((x^\frac 1n))$ and $r = |x|^\frac pq$ for some $p, q \in \Z$ and $q \mid n$, moreover we will prove (i) and (ii). Second we will show the converse that every non-endpoint vertex is of this form. The rest is left as an exercise. It is immediate that we only need to deal with Type II points.
 
 Using \autoref{prop:multiplicity}, let $\zeta = \zeta(a, r)$ be a Type II point with multiplicity $m$, $a \in k((x^\frac 1m))$, $r = |x|^\frac pq$ for some $p, q \in \Z$ coprime, and $m, q \mid n$. Set $\mg = \LCM(\mm(\zeta), q)$, thus $\mg \mid n$. The directions of $\zeta$ can be written $\vec v(\infty)$ or $\vec v(a + cx^\frac pq)$ where $c \in k$. Observe that $a + cx^\frac pq$ has multiplicity $\mg$ whenever $c \ne 0$; furthermore this finite Puiseux series is the $\bO(s)$ truncation of itself for any $s < r$, so by \autoref{prop:multiplicity}, $\zeta(a + cx^\frac pq, s)$ has multiplicity $\mg$ for any $s < r$ and $c \in k^*$. This concludes the proof for (ii). In the case of $c = 0$, we have simply $a$, which is multiplicity $m$, as is $\zeta$. 
\autoref{prop:multsubtreehull} shows that the set of points where $\mm = m$ is open within $[a, \infty] \subset \T_m$. This both proves (i) and finalises the proof that $\zeta$ is a vertex of $\T_n$ with an edge in every possible direction. 

 Conversely, suppose that $\zeta = \zeta(a, r)$ is a non-endpoint vertex in $\T_n$. %
 There must be at least two more intervals $[\xi_1, \zeta]$ and $[\xi_2, \zeta]$ within $\T_n$ emanating from $\zeta$ in distinct finite directions. By \autoref*{prop:multiplicity} \ref{prop:multiplicity:IV}, the $\xi_j$ are Type I/II/III. Using \autoref*{prop:multiplicity} \ref{prop:multiplicity:II} we can pick Type I points $a_1, a_2 \in k((x^\frac 1n))$ so that $a_j \preceq \xi_j$, then we must have $r = \abs{a_1 - a_2} = \abs x^\frac kn$ for some $k \in \Z$.
\end{proof}

\begin{defn}
 Let $\overline \TV_n$ be the set of vertices of $\T_n$, and $\TV_n = \overline \TV_n \sm \P^1$ be the Type II (non-endpoint) vertices.
\end{defn}

If $\bvec v$ is not a (or the only) direction at $\xi$ with multiplicity $1$, then using \autoref{prop:multsubtreehull} one can show there is an interval $(\xi, \zeta) \subset \bvec v$ consisting purely of multiplicity $\mm(\bvec v)$ points.

\autoref{prop:multsubtreevalency} can be rephrased in terms of multiplicities in directions, as follows in the next result. This corollary is an alteration of \cite[Proposition 3.39]{FJ04}. Note that they choose a different definition for $\mm(\vec v(\infty))$ (which they would denote by $\vec v(\nu_{\mathfrak m})$), that is at least $\mm(\zeta)$; we always have $\mm(\vec v(\infty)) = 1$. Regardless, Favre and Jonsson definition of generic multiplicity $\mg$ below is the same.

\begin{cor}\label{cor:multsubtreevalency}
 Let $\zeta \in \P^1_\an$ be a Type II point. Write $\zeta = \zeta(a, |x|^\frac pq)$ with $\mm(a) = \mm(\zeta)$, $\GCD(p, q) = 1$, and set $\mg = \LCM(\mm(\zeta), q)$. 
 Then $\mg$ is the smallest integer such that $\zeta$ is a vertex in $\T_\mg$. Moreover, $\mm(\vec v(\infty)) = 1$ and exactly one of the following holds
\begin{enumerate}[label = (\roman*)]
 \item $\mm(\bvec v) = \mg$ for every direction $\bvec v \ne \vec v(\infty)$ at $\zeta$; in this case $\mg = \mm(\zeta)$.
 \item $\mm(\bvec v) = \mg$ for every direction $\bvec v$ at $\zeta$ except two: $\vec v(\infty)$ and $\bvec u$, say (distinct).\\ In this case, $\mm(\bvec u) = \mm(\zeta) < \mg$, moreover $\mm(\zeta) \mid \mg$.
\end{enumerate}
\end{cor}

\begin{defn}
 Let $\zeta \in \P^1_\an$. Then define \emph{generic multiplicity}, $\mg(\zeta)$ to be
\begin{enumerate}
 \item $\mm(\zeta)$ if $\zeta$ is Type I,
 \item $\mg$ as in \autoref{prop:multsubtreevalency} and \autoref{cor:multsubtreevalency} if $\zeta$ is Type II,
 \item $\infty$ if $\zeta$ is Type III or IV.
\end{enumerate}
\end{defn}

This is a rather disjointed definition; the next proposition provides some synergy.

\begin{prop}\label{prop:genericmultiplicity}
 The following are equivalent to $\mg(\zeta)$.
\begin{enumerate}[label = (\roman*)]
 \item \label{prop:genericmultiplicity:two} $\inf \set[n]{\zeta \in \overline{\T_n \cap \P^1(\hk)}}$,
 \item \label{prop:genericmultiplicity:three} $\displaystyle\liminf_{a \to \zeta,\ a \in \P^1(\hk)} \mm(a)\quad = \quad \sup_{U \supset \zeta}\quad \inf_{a \in \P^1(\hk) \cap U} \mm(a)$,
 \item \label{prop:genericmultiplicity:four} The smallest $n$ such that $\zeta \in\overline\TV_n$, or $\infty$ otherwise.
\end{enumerate}
\end{prop}

\begin{proof}
 \ref{prop:genericmultiplicity:two} and \ref{prop:genericmultiplicity:three} are different ways to phrase the same topological idea (limit infimum).

 \ref{prop:genericmultiplicity:four}: A Type III point is never a vertex in $\T_n$ and a Type IV point is never in a $\T_n$, therefore we get a value of $\inf \emp = \infty$, matching the generic multiplicity. For Type II, the result is given in \autoref{cor:multsubtreevalency}. For Type I points, the result is immediate because $\mg(\zeta) = \mm(\zeta)$.%

\ref{prop:genericmultiplicity:two}: It is enough to show that $\overline\TV_n = \overline{\T_n \cap \P^1(\hk)}$. If $\zeta$ is a vertex, then it is Type II or I. For a Type I point we immediately have $\zeta \in \T_n \cap \P^1(\hk)$. For a Type II point, let $U$ be any open neighbourhood of $\zeta$. Then $U$ contains all but finitely many directions of $\zeta$. By \autoref{cor:multsubtreevalency} all but two directions have multiplicity $\mg(\zeta)$. Note that the residue field of $\hk$ must be infinite so there are infinitely many directions containing a Type I point with multiplicity $\mg(\zeta)$. This shows that $U \cap \T_n \cap \P^1(\hk) \ne \emp$ for any neighbourhood $U$ of $\zeta$, and therefore $\zeta \in \overline{\T_n \cap \P^1(\hk)}$.
\end{proof}

\begin{cor}\label{cor:verticesisolated}
 $\overline\TV_n = \overline{\T_n \cap \P^1(\hk)} = \overline{\set[a \in \P^1(\K)]{\mm(a) \mid  n}}$ is a perfect set of all vertices in $\T_n$, however the Type II points $\TV_n$ are all isolated in the hyperbolic metric with distances in $\tfrac 1n \Z$.
\end{cor}

The multiplicity of a direction can also be understood in a way similar to points. We clearly have an action of $G$ on directions by $(\prim^*, \bvec v) \longmapsto \prim_\#(\bvec v)$. The following is an exercise.

\begin{prop}
 Let $\zeta \in \P^1_\an$ and $\bvec v$ be a direction at $\zeta$. Then $\mm(\bvec v) = |\Orb_\prim(\bvec v)|$.
 \end{prop}

 \begin{defn}\label{defn:specgendirn}
Let $\zeta$ be a Type II point.
\begin{itemize}
 \item When $\mg(\zeta) = 1$ we say $\zeta$ is \emph{integral}.
 \item We say $\zeta$ is \emph{free} iff $\mg(\zeta) = \mm(\zeta)$, and \emph{satellite} otherwise.
 \item We will say $\bvec v \in \Dir \zeta$ is a \emph{generic direction} iff $\mm(\bvec v) = \mg(\zeta)$, and say it is \emph{special} otherwise.
 \end{itemize}
\end{defn}
 
\begin{rmk}\label{rmk:specgendirn}
  \autoref{prop:multsubtreevalency} says that a Type II point $\zeta$ has at most two special directions, namely $\vec v(a), \vec v(\infty)$ where $a \in \K$ is described in the proposition. Further, there are exactly two special directions if and only if $\zeta$ is satellite. Otherwise, $\zeta$ is free and $\mm(\zeta) = \mg(\zeta) = \mm(a) = \mm(\vec v(a))$. The direction $\vec v(\infty)$ is always special unless $\zeta$ is integral, in which case every direction is generic. 
\end{rmk}

\begin{prop}\label{prop:galois:multindirn}
 Let $\zeta$ be a Type II point and $\bvec v$ one of its directions. For every $\xi \in \bvec v$ we have $\mm(\bvec v) \mid \mm(\xi) \mid \mg(\xi)$.
\end{prop}

\begin{proof}
 If $\bvec v = \vec v(\infty)$ then $\mm(\bvec v) = 1$ so there is nothing to prove. Otherwise $\xi \prec \zeta = \zeta(\gamma, \abs x^\frac ab)$ (WLOG $b = \mg(\zeta)$) and using \autoref{prop:multsubtreevalency} we can choose $\mm(\gamma) = \mm(\bvec v)$, so that $\bvec v = D_\an(\gamma, \abs x^\frac ab)$. Now, $\xi' = \zeta(\gamma, \norm[\xi]{y-\gamma})$ is the join of $\xi$ and $\gamma$. One can check that $\mm(\xi') = \mm(\bvec v)$ and by \autoref{prop:multiplicityorder}, $\mm(\xi') \mid \mm(\xi)$.
\end{proof}

\begin{defn}
 We define the \emph{generic multiplicity} of a subset $U \subset \P^1_\an(\K)$ as $\mg(U) = \min_{\zeta \in U} \mg(\zeta)$.
\end{defn}

\begin{prop}\label{prop:galois:multiplicitysubset}
 The minimum generic multiplicity $\mg$ in $U$ is attained by Type I points. For a direction $\bvec v$ at a point, we have $\mm(\bvec v) = \mg(\bvec v)$. More generally, if $U$ is a connected open set then $U$ has a Type II point $\zeta$ with $\mm(\zeta) = \mm(U) \mid \mg(\zeta) = \mg(U)$.
\end{prop}

\begin{proof}
 The first part is an exercise applying \autoref{cor:multsubtreevalency}. Let $g = \mg(U)$ and $m = \mm(U)$. The endpoints of $\T_m$ are Type I, and any Type II point $\zeta \in U\cap \T_m$ has a generic direction contained in $U$. 
 By \autoref{cor:multiplicitylowersemicont}, if $U$ is an open set of points of multiplicity at least $m$ (attained by some point) then $T = U \cap \T_m$ must be an open subtree of $\T_m$, purely of multiplicity $m$. If $T$ has a vertex $\zeta$ of $\T_m$ then it is free of multiplicity $m$ in $U$, so we are done $\mm(\zeta) = \mm(U) = \mg(\zeta) = \mg(U)$.
 Otherwise, assume $T$ is an open interval shorter than $\frac 1m$ without any points of generic multiplicity $m$. 
 Let $[\zeta, \xi]$ be an interval connecting a point $\xi \in U$ with $\mg(\xi) = \mg(U)$ to $\zeta \in T$ with $\vec v(\xi) \cap \T_m = \emp$. Then $\vec v(\xi)$ is generic and by \autoref{prop:galois:multindirn}
 \[\mg(\zeta) = \mm(\vec v(\xi)) \mid \mm(\xi) \mid \mg(\xi).\]
 In particular, $\mg(\zeta) \le \mg(\xi)$ which was meant to be minimal in $U$, so $\mg(U) = \mg(\zeta)$ and we are done.
\end{proof}

 \subsection{Topological Results over Puiseux Series}\label{sec:galois:top}

 Given $\zeta \in \P^1_\an$, let $\Dir\zeta$ denote the set of directions at $\zeta$. We may also write $\vec v_\zeta(\phi)$ to specify the direction of $\phi$ at $\zeta$.

\begin{lem}\label{lem:infdirconv}
 Let $\zeta \in \P^1_\an$ and $(\zeta_n)_{n=1}^\infty$ be a sequence such that the sequence $\bvec v_n = \vec v_\zeta(\zeta_n) \in \Dir\zeta$ has no constant subsequences (ignoring $n$ for which $\zeta_n = \zeta$). Then $\lim_{n \to \infty} \zeta_n = \zeta$.
\end{lem}

\begin{proof}
 WLOG $\zeta_n \ne \zeta\ \forall n$. It is enough to show that for any open connected affinoid $U \ni \zeta$ there exists an $N \in \N$ such that for every $n \ge N$ we have $\zeta_n \in U$. $\bvec v \subset U$ for all but finitely many directions $\bvec v$ at $\zeta$, namely $\bvec w_1 \dots \bvec w_k$. By hypothesis $(\bvec v_n)$ contains only finitely many terms equal to $\bvec w_i$, therefore we can find a $N_i \in \N$ such that $\bvec v_n \ne \bvec w_i$ for every $n \ge N_i$. Let $N = \max_i N_i$, then for every $n \ge N$, we find $\zeta_n \nin \bvec w_i$ for any $1 \le i \le k$, and hence $\zeta_n \in U$.
\end{proof}

Under the same hypotheses as in the Lemma above and using the definition of the weak topology on $\P^1_\an$ we immediately see that for any $\phi \in \hk(y)$, \[\lim_{n\to \infty} \norm[\zeta_n]{\phi} = \norm[\zeta]{\phi}.\]

\begin{ex}
 For any $\phi \in \hk(y)$, let $(c_n)$ be any sequence of distinct complex numbers. \[\lim_{n\to \infty} \norm[c_n x^k]{\phi} = \norm[\zeta(0, |x|^k)]{\phi}\]
\end{ex}

This can be drastically improved however.

\begin{prop}\label{prop:galois:evcnstreps}
  Let $\zeta \in \P^1_\an$ and $(\zeta_n)_{n=1}^\infty$ be a sequence as in \autoref{lem:infdirconv}. Then for any $\phi \in \hk(y)$, the sequence $\left(\norm[\zeta_n]{\phi}\right)_{n = 1}^\infty$ is eventually constant (equalling $\norm[\zeta]{\phi}$).
  
  In particular, if $\zeta = \zeta(a, r)$ is Type II, and $b_n \in \CD(a, r)$ satisfy $|b_m - b_n| = r\ \forall m \ne n$, then the sequence $\left(\norm[b_n]{\phi}\right)_{n = 1}^\infty$ eventually equals $\norm[\zeta]{\phi}$.
  \end{prop}

\begin{proof}
By taking $f$ to be a product and quotient of linear terms, we may assume WLOG that $f = y - c$, for some $c \in K$. We may also dispense of the finitely many terms of the sequence with $\bvec v_n = \vec v(\infty)$. In any case satisfying the hypotheses $\zeta$ must be Type II; let $\zeta = \zeta(a, r)$.%
 
 Suppose $c \nin \bar D(a, r)$, then for any $b \in \bar D(a, r)$ we have $|b - c| = |c|$. $\zeta_n \in \bar D_\an(a, r)$, therefore $\norm[\zeta]{y-c} = \norm[\zeta_n]{y-c} = |c|$ for every $n$.
 
 Suppose $c \in \bar D(a, r)$, then $\zeta_n \in \bar D_\an(c, r)$ and $\norm[\zeta_n]{y - c} \le r$; on the other hand for $n$ large enough, we have $\zeta_n \nin \vec v(c) = D_\an(c, r) \iff \norm[\zeta_n]{y - c} \ge r$.
\end{proof}

\begin{lem}
  Let $\zeta \in \P^1_\an$ be a point and $(\zeta_n)_{n=1}^\infty \subset \T_m$ be a sequence of points distinct from but converging to $\zeta$. Then either $\set{\vec v(\zeta_n)}$ is infinite, or $\zeta_n \to \zeta$ in the hyperbolic metric.
\end{lem}

\begin{proof}
 Assume that $\set{\vec v(\zeta_n)}$ is finite, WLOG $= \set{\vec v(\gamma_1), \dots, \vec v(\gamma_k)}$ with $\gamma_j \in \T_m$. Let $\zeta = \zeta(\gamma, r)$. Then for every $\eps > 0$ there is an $N \in \N$ such that $\forall n \ge N$,
 \[\zeta_n \in \P^1_\an \sm \bigcup_{j=1}^k \bar D_\an (\gamma_j, r-\eps)\]
 Then since $(\zeta_n) \subset \vec v(\gamma_1)\cup \cdots \cup \vec v(\gamma_k)$, for $n \ge N$
 \[\zeta_n \in \bigcup_{j=1}^k D_\an (\gamma_j, r) \sm \bar D_\an (\gamma_j, r-\eps)\]
 
 By \autoref{prop:multsubtreevalency} and \autoref{cor:verticesisolated}, for $\eps < \frac 1m$, we have that
 \[\T_m \cap \bigcup_{j=1}^k D_\an (\gamma_j, r) \sm \bar D_\an (\gamma_j, r-\eps) = \bigcup_{j=1}^k (\zeta(\gamma, r), \zeta(\gamma_j, r-\eps)] \subset B_\bH(\zeta, \log_{|x|}(1 - \eps/r))\]
\end{proof}

\begin{cor}\label{cor:treevertconv}
 Let $\zeta \in \P^1_\an$ be a point and $(\zeta_n)_{n=1}^\infty \subset \TV_m$ be a sequence of points distinct from but converging to $\zeta$. Then $\set{\vec v(\zeta_n)}$ is infinite.
\end{cor}

\subsection{Endomorphisms of Puiseux Series}\label{sec:galois:endo}

In \autoref{def:rootopsprim} and \autoref{prop:rootsauto} we defined an isomorphism which extended $x \mapsto \lambda x$ from $\k((x))$ to $\hk$. Of course this is not unique. Now, for $0 < \abs{g(x)} < 1$, we wish to carefully extend the map $x \mapsto g(x)$ to a homomorphism $g^* : \hk \to \hk$. This will be particularly useful when $g(x) \in \k[[x]]$ is the first component of a skew product. We shall also treat the uniqueness question.
Caution: A field-theoretic extension using the axiom of choice is not enough; one can show there exist extensions of $g^*$ from $\C((x))$ to $\hk(\C)$ which are not continuous with respect to the non-Archimedean norm.

\begin{lem}[Binomial Theorem for $\hk$]\label{lem:binom}
Given $\abs{\gamma(x)} < 1$ and $q \in \Q$, there exists a unique $q$th power of\ $1 + \gamma(x)$ with constant coefficient $1$, namely
\[(1 + \gamma(x))^{q} = \sum_{j=0}^\infty \frac{q(q-1)\cdots(q-k+1)}{k!} (\gamma(x))^j.\]
This agrees with the usual definition when $q = n \in \N$ and has the property that $(1 + \gamma(x))^{pq} = (1 + \gamma(x))^q)^p$ for another rational $p$. For any complete subfield $L < \hk$, $\gamma(x) \in L \implies (1 + \gamma(x))^q \in L$. Finally, the algebraic $n$th roots of $1 + \gamma(x)$ are \[(1 + \gamma(x))^{\frac{1}{n}},\, \xi(1 + \gamma(x))^{\frac{1}{n}},\, \dots,\, \xi^{n-1}(1 + \gamma(x))^{\frac{1}{n}}\] where $\xi \in \k$ is a primitive $n$th root of unity.
\end{lem}

The proof is essentially the same as for the complex binomial theorem.

\begin{defn}\label{defn:galois:eqauto}
 Let $r \in \Q$, $\lambda \in \k^\times$, and $g(x) \in \hk$ be such that $g(x) = \lambda x^r + \lo(x^r) = \lambda \tilde g(x) = \lambda x^r(1+h(x))$. Let $\lambda_\bullet$ be a sequence of roots with $\lambda_1 = \lambda$. Then we define $g^*$ by %
 \[x^{\frac ab} \mapsto \lambda_b^a x^{r\frac ab}(1+h(x))^{\frac ab}.\]
\end{defn}

\begin{prop}\label{prop:primghom}\label{cor:galois:fieldisometry}
 Let $q \in \Q_+$, $\lambda \in \k^\times$, and $g(x) \in \hk$ be such that $g(x) = \lambda x^{\frac1q} + \lo(x^{\frac1q})$. The map $g^* : \hk \longrightarrow \hk$ above is a well defined ring homomorphism with scale factor $q$. If $g(x) \in \k((x))$, then $g^*$ restricts to $x \mapsto g(x)$ on $\k((x))$. The set of $m$th roots of $g(x)$ in $\hk$ are
 \[g^*(x^{\frac 1m}),\, \xi g^*(x^{\frac 1m}),\, \dots,\, \xi^{m-1} g^*(x^{\frac 1m})\] where $\xi \in \k$ is a primitive $m$th root of unity. When $\frac 1q = n \in \N_+$, then $g^*\omega^{*n} = \omega^*g^*$ for any $\omega^* \in G$. Finally, if $q = 1$ then $g^*$ is an isometric isomorphism on $(\hk, \abs\cdot)$ which commutes with Galois actions.
\end{prop}

The proof is given in \autoref{sec:puiseuxappx}, page \pageref{appx:prop:primghom}. The choice to write $\frac 1q$ as an exponent is a convention carried from \cite{berkskew} so that the scale factor is simply written as $q$.
The next proposition says that our choices didn't matter too much.

\begin{prop}\label{prop:galois:extensionnearlyunique}
Let $\Psi$ be any continuous extension of $g^*$ from $\k((x))$ to $\hk$. Then $\Psi = g^*\circ \omega^*$ for some $\omega^* \in G$.
\end{prop}

\begin{proof}
Since $\Psi$ is continuous, $\Psi\left(\sum_{j = 1}^\infty c_j x^{\frac{p_j}{q_j}}\right) = \sum_{j = 1}^\infty c_j \Psi\left(x^{\frac{1}{q_j}}\right)^{p_j}$. Therefore, it is enough to show on roots of $x$.
It must be that $\Psi(x^{\frac 1n})^n = \Psi(x) = g(x)$ for every $m$. So $\Psi(x^{\frac 1n})$ is \emph{some} $n$th root of $g(x)$ in $\hk$ and so by \autoref{prop:primghom} $\Psi(x^{\frac 1n}) = \omega_n g^*(x^{\frac 1n}) = g^*(\omega_n x^{\frac 1n})$ where $\omega_n$ is some $n$th root of unity. Now we show that $\omega$ is a sequence of roots of unity, and then by definition of $g^*\circ \omega^*$ we are done.
\[\omega_{mn}^m g^*(x^{\frac 1n}) = \left(\omega_{mn} g^*(x^{\frac 1{mn}})\right)^m = \Psi((x^{\frac 1{mn}}))^m = \Psi((x^{\frac 1n})) = \omega_n g^*(x^{\frac 1n})\]
Thus $\omega_{mn}^m = \omega_n$ as required.
\end{proof}

\begin{prop}\label{prop:nearlyfunctorial}
 Let $g(x), h(x) \in \C[[x]]$, such that 
\begin{align*}
 g_1(x) = \lambda x^m + \bO(x^{m+1}) &= \lambda \tilde g_1(x) = \lambda x^m(1+h_1(x))\\
 g_2(x) = \mu x^n + \bO(x^{n+1}) &= \mu \tilde g_2(x) = \mu x^n(1+h_2(x))
\end{align*}
($m, n \ge 1$). Then \[g_2^*\circ g_1^* = 
  \prim^{k*}(g_1\circ g_2)^*\]
 where $2k\pi \le \arg(\lambda) + m\arg(\mu) < 2(k+1)\pi$.
\end{prop}

The proof is given on page \pageref{appx:prop:nearlyfunctorial} of the \autoref{sec:puiseuxappx}.

\subsection{Galois Actions on the Berkovich Projective Line}\label{sec:galois:galoisonberk}

In this subsection we discuss a version of the Berkovich projective line over the Puiseux series which can only be described by Laurent series. The effect of this is to quotient out by every Galois orbit, see \cite[Corollary 1.3.6]{Berk}. For an excellent account of the below facts, including proofs, see Favre \& Jonsson \cite{FJ04}. The Berkovich projective line $\P^1_\an(\k((x))$ over the formal Laurent series $\k((x))$ is the set of seminorms which act on $\k((x))[y]$ and restrict to $\abs\cdot$ on $\k((x))$. In \cite{FJ04} this is called the (relative) \emph{Valuative Tree} $\V$, the set of semivaluations extending the valuation $\ord_0$. %
We shall view the elements of $\V$ as seminorms to keep notation simple, but the reader should understand that if $\zeta \in \V$ is a seminorm, then the associated semivaluation $v_\zeta$ is defined by $\log_{|x|}(\norm[\zeta]{\phi}) = v_\zeta(\phi)$.

Seminorms in $\P^1_\an(\K)$ are defined on $\hk[y] \supset \k((x))[y]$, therefore for any $\zeta \in \P^1_\an(\hk)$, the restriction $\left.\zeta\right|_{\C((x))[y]}$ is a well defined member of $\V$. In this sense, $\V$ is similar but not identical to $\P^1_\an(\K)$. The extension of valuation $v \in \V$ to $\K$ is not unique, for example %
the values $v(y - x^\half)$ and $v(y + x^\half)$, must sum to $v(y^2 - x)$, but are otherwise unconstrained. 

Another way to make elements of $\P^1_\an(\K)$ agnostic to the extension $\hk[y]$ of $\C((x))[y]$ is to consider %
Galois actions. Indeed let $\omega^* \in G$ and $\phi \in \C((x))[y]$, then
\[\omega_*(\zeta)(\phi) = \zeta(\omega^*(\phi)) = \zeta(\phi)\] since $\phi$ was arbitrary we have \[\left.\omega_*(\zeta)\right|_{\C((x))[y]} = \left.\zeta\right|_{\C((x))[y]}.\]

\begin{prop}\label{prop:GaloisactiononP1}
\begin{align*}
 G\times \P^1_\an(\hk) &\longrightarrow \P^1_\an(\hk)\\
 (\omega^* , \zeta) &\longmapsto \omega_*(\zeta)
\end{align*}
is a group action (technically a right group action by an Abelian group). Elements of the same Galois orbit have the same restriction to $\C((x))[y]$.
\end{prop}

\begin{defn}
 Let $\mathfrak p : \P^1_\an(\hk) \to \V$ be the map $\zeta \mapsto \left.\zeta\right|_{\C((x))[y]}$.\\
 Let $\P^1_\an(\hk)/G$ denote the orbit space associated with the above group action, with the quotient denoted $ \pi : \P^1_\an(\hk) \to \P^1_\an(\hk)/G$ with $\pi (\zeta) = [\zeta]_G$.
\end{defn}

By the above discussion and proposition, we know that if $\xi \in G(\zeta)$ then $\mathfrak p(\zeta) = \mathfrak p(\xi)$, the converse is also true.%

\begin{thm}[Favre \& Jonsson {\cite[Theorem 4.17]{FJ04}}]
  \[
\begin{tikzcd}
 \P^1_\an(\hk) \arrow[two heads]{rd}{\mathfrak p} \arrow[swap, two heads]{d}{\pi}&  \\
\P^1_\an(\hk)/G \arrow[swap]{r}{\Theta} & \V
\end{tikzcd}
\]
where $\Theta$ is an isomorphism.
\end{thm}

We take the obvious definition $\Theta([\zeta]_G) = \mathfrak p(\zeta)$. The map has been shown to be well defined and commute already by \autoref{prop:GaloisactiononP1}. %
The idea of the proof is that $\mathfrak p$ is order preserving, and we can use ideas from the next sections to show $\Theta$ is bijective on Type I points; the rest follows by convexity.

\section{Maps}\label{sec:maps}


\subsection{Motivation}

The key idea in our applications is that a skew product $\phi : X \dashto X$ on a ruled surface over $k = \bar k$ induces a non-Archimedean skew product $\phi_*$ on the Berkovich projective line $\P^1_\an$ over the Puiseux series $\K_\k$. 
Moreover, we seek such a correspondence 
which captures the geometry of sections and fibral divisors, where $\phi_*$ represents the action of $\phi$ on divisors and points in (and around) a source and target fibre of $X$.
In \refchapberk{} (specifically in \refsecskew{}) such skew products were discussed but only very briefly.

\begin{minipage}{0.25\textwidth}
 \[
\begin{tikzcd}
 X \arrow[dashed]{r}{\phi} \arrow[swap]{d}{h} & X \arrow{d}{h} \\
 B \arrow[swap]{r}{\phi_1} & B
\end{tikzcd}
\]
\end{minipage}
$\leadsto$\hfill
\begin{minipage}{0.25\textwidth}
   \[
\begin{tikzcd}
 \K(y) & \K(y) \arrow[swap]{l}{\phi^*} \\
\K \arrow[hook]{u} & \K \arrow[hook]{u} \arrow{l}{\phi_1^*}
\end{tikzcd}
\]
\end{minipage}
$\leadsto$\hfill
\begin{minipage}{0.3\textwidth}
 \begin{align*}
 \phi_* : \P^1_\an(\K) &\longrightarrow \P^1_\an(\K)\\
 \zeta &\longmapsto \phi_*(\zeta)\\
 \text{where } \norm[\phi_*(\zeta)]{f} &= \norm[\zeta]{\phi^*(f)}^\q
\end{align*}
\end{minipage}

The advantage being that the Berkovich space carries information about all birational models of $X$ at once (\autoref{sec:space} \& \autoref{sec:corresp}). In \autoref{sec:moremaps} we see by example how to combine the knowledge of mappings in this section with dual graphs to recover geometric information about the rational map $\phi : X \dashto X$ from the non-Archimedean version $\phi_* : \P^1_\an \to \P^1_\an$.

Recall the general notion of a \emph{non-Archimedean skew product} $\Psi_*$ over $K$ from \autoref{sec:ratskewprops}. This requires a good notion of \emph{\equivar{} skew endomorphism} $\Psi$ of $K(y)$, meaning for some fixed \emph{scale factor} $\q$ it extends an automorphism $\Psi_1$ of $K$ such that $\abs{\Psi(a)} = \abs{\Psi_1(a)} = \abs{a}^{\frac 1\q}$ for every $a \in K$. Then we can define the non-Archimedean skew product $\Psi_*$.\\
\begin{minipage}{0.45\textwidth}
   \[
\begin{tikzcd}
 K(y) & K(y) \arrow[swap]{l}{\Psi} \\
K \arrow[hook]{u} & K \arrow[hook]{u} \arrow{l}{\Psi_1}
\end{tikzcd}
\]
\end{minipage}
$\leadsto$\hfill
\begin{minipage}{0.45\textwidth}
 \begin{align*}
 \Psi_* : \P^1_\an(K) &\longrightarrow \P^1_\an(K)\\
 \zeta &\longmapsto \Psi_*(\zeta)\\
 \text{where } \norm[\Psi_*(\zeta)]{f} &= \norm[\zeta]{\Psi(f)}^\q
\end{align*}
\end{minipage}

We want to derive the $\k$-algebra map $\Psi = \phi^*$ from a rational map of ruled surfaces. Classically, a skew product is one of the form \[\phi(x, y) = (\phi_1(x), \phi_2(x, y))\] defined on some product space $B \times C$ such as $\R^2$, $\C^2$ or $\P^1 \times \P^1$.  In further generality, we consider a ruled surface $h : X \to B$ and a skew product to be the following commuting diagram of rational maps and the equivalent algebra maps.
\[
\begin{tikzcd}
 X \arrow[dashed]{r}{\phi} \arrow[swap]{d}{h} & X \arrow{d}{h} \\
 B \arrow[swap]{r}{\phi_1} & B
\end{tikzcd}
\]

Where $\phi(x, y) = (\phi_1(x), \phi_2(x, y))$ on $\P^1 \times \P^1$, $h(x, y) = x$ is simply the first projection map. It is helpful to consider this leading example throughout this \thesisarticle{thesis}{paper}. For instance, $\phi(x, y) = (\phi_1(x), \phi_2(x, y))$ induces a $\k$-algebra homomorphism of function fields.
\begin{align*}
 \phi^* : \k(x, y) & \longrightarrow \k(x, y)\\
x & \longmapsto \phi_1(x)\\
y & \longmapsto \phi_2(x, y)
\end{align*}
After changing coordinates we may assume that $\phi_1(0) = 0$ and look in a neighbourhood of $x=0$ then we obtain a $\k$-algebra map $\Psi = \phi_1^* : \k[[x]] \to \k[[x]]$ which extends to one of the local function field $\phi^* : \k((x))(y) \to \k((x))(y)$. In more algebraic terminology, we took the completion of the local ring $\k[x]_{(x)}$.

  \[
\begin{tikzcd}
 \k((x))(y) & \k((x))(y) \arrow[swap]{l}{\phi^*} \\
\k((x)) \arrow[hook]{u}{h^*} & \k((x)) \arrow[hook, swap]{u}{h^*} \arrow{l}{\phi_1^*}
\end{tikzcd}
\]
This process works in general for an arbitrary base curve $B$ over $\k$. Given $b \in B$, with local ring $\bO_{B, b}$ then its completion is always isomorphic to $\k[[x]]$ by the Cohen Structure Theorem. For background on such completions, we refer to \cite[I.5, II.9]{Hart}. If $\phi_1(b) = c$ then $\phi_1^* : \bO_{B, c} \to \bO_{B, b}$ induces a homomorphism $\phi_1^* : \k[[x']] \to \k[[x]]$. By extension, $\phi$ induces a map $\phi^* : \k((x'))(y) \to \k((x))(y)$ which is a skew endomorphism over $\phi_1^*$.

After taking the algebraic closure of $\k((x))$ to obtain the Puiseux series $\Kk$, this map can be extended to a $\k$-algebra endomorphism $\phi^* : \hk(y) \to \hk(y)$. Doing this carefully is the objective of the following subsections. We can write $\phi_1^*(x) = \phi_1(x) \in \k[[x]]$ with $\phi_1(x) = \lambda x^n + \cdots$ (higher order terms) and $\lambda \in \k^\times$ then this extends to an `\equivar{}' skew endomorphism over $\K$. With $\phi_2 \in \k((x))(y)$ we call such a map a \emph{$\k$-rational skew endomorphism}. Non-Archimedean skew products were defined in larger generality from arbitrary dilating skew endomorphisms in \cite{berkskew}. Here, we take a closer look at the particular case over the Puiseux series.

In the case deriving from a skew product on a ruled surface (after completion of local rings), say $\phi(x, y) = (\phi_1(x), \phi_2(x, y))$, with $\phi_1(x) = \lambda x^n + \cdots$, this non-Archimedean skew product has scale factor $\q = \frac 1n$ and its formula becomes \[\norm[\phi_*(\zeta)]{f(x, y)} = \norm[\zeta]{f(\phi_1(x), \phi_2(x, y))}^{\frac 1n}.\]

\subsection{Rational and Puiseux Skew Products}

\begin{defn}
 Let $\phi^* : \k((x))(y) \to \k((x))(y)$ be a $\k$-algebra endomorphism. Suppose that 
\begin{enumerate}
 \item $\phi^*(x) = \lambda x^n + \bO(x^{n+1}) \in \k((x))$ for some $n \in \N$, $\lambda \in \k^\times$; and
 \item $\phi^*(y) \in \k((x))(y)$
\end{enumerate}
Then we say that $\phi^* : \hk(y) \to \hk(y)$ is a \emph{$\k$-rational skew endomorphism} and that the induced $\phi_* : \P^1_\an(\hk) \to \P^1_\an(\hk)$ is a \emph{$\k$-rational skew product}. We define $\phi_1(x) = \phi^*(x)$ and $\phi_2(y) = \phi^*(y)$. We may denote $\phi$ by $(\phi_1, \phi_2)$ to emphasise this splitting. The \emph{relative degree} of $\phi$ is $\rdeg(\phi) = \deg_y(\phi_2)$.
\end{defn}

Occasionally we may simply refer to $\phi_*$ as a \emph{rational skew product} if the ground (residue) field $\k$ is clear from the context or irrelevant.

\begin{prop}\label{prop:galois:rational}
 Let $\phi^* : \k((x))(y) \to \k((x))(y)$ be a $\k$-rational skew endomorphism. Then indeed $\phi^* : \k((x)) \to \k((x))(y)$ is \adilating{} skew endomorphism which can be extended to $\K(y)$ or $\hk(y)$, and the induced $\phi_* : \P^1_\an \to \P^1_\an$ is a skew product of scale factor $1/n$.
\end{prop}

\begin{proof}%
 It is easy to see that $\phi^*$ restricts to an automorphism of $\k((x))$ because $\phi^*$ is a $\k$-algebra endomorphism and $\phi^*(x) \in \k((x))$. Hence, recalling the definition from \refsecskew{}, $\phi^*$ is a skew endomorphism of $\k((x))$. By \autoref{cor:galois:fieldisometry}, $\phi_1^*$ can be extended to $\hk$ and it remains true that $\abs{\phi^*(x)} = \abs{x}^n$. It is natural to further extend $\phi^*$ to $\hk(y)$ by retaining the original mapping $y \mapsto \phi^*(y)$. This makes $\phi^* : \hk(y) \to \hk(y)$ \adilating{} skew endomorphism and $\phi_* : \P^1_\an(\hk) \to \P^1_\an(\hk)$ a skew product over $\hk$ with scale factor $1/n$.
\end{proof}

\begin{rmk}
 It is not easy to write down $\k$-rational skew products not deriving from rational maps in $x$ and $y$ and it is even harder to determine which may or may not come from a skew product on surface due to the complexity of any ring isomorphism which turns a completed local ring of a curve over $\k$ into $\k[[x]]$. %
\end{rmk}

We wish to have a less strict definition for certain cases where we may have changed coordinates using Puiseux series. It will turn out that these \emph{Puiseux skew products} have all the same analytical properties of $\k$-rational ones, but do not coherently carry geometric data from a skew product on a surface over $\k$. Similar to \autoref{prop:galois:rational} it is not hard to see that the $\k$-algebra endomorphism described below is in fact \adilating{} skew endomorphism of $\K(y)$ or $\hk(y)$.

\begin{defn}
 Let $\phi^* : \K(y) \to \K(y)$ be a $\k$-algebra endomorphism. Suppose that 
\begin{enumerate}
 \item $\phi^*(x) = \lambda x^n + \bO(x^{n+1}) \in \k[[x]]$ for some $n \in \N$, $\lambda \in \k^\times$; and
 \item $\phi^*(y) \in \K(y)$.
\end{enumerate}
Then we say that $\phi^* : \hk(y) \to \hk(y)$ is a \emph{Puiseux skew endomorphism} and that the induced $\phi_* : \P^1_\an(\hk) \to \P^1_\an(\hk)$ as above is a \emph{Puiseux skew product}. We define $\phi_1(x) = \phi^*(x)$ and $\phi_2(y) = \phi^*(y)$.
\end{defn}

Recall that a non-Archimedean skew product is (almost always) neither a homomorphism nor a Berkovich rational map, but a composition of the two. Worse, the translation of a skew product on a ruled surface to its non-Archimedean counterpart fails to be functorial. However, this issue with functoriality is controlled by the Galois group as laid out in \autoref{thm:nearlyfunctorial}.

\begin{rmk}[Warning]
 If $\phi_*$ is rational, then its inverse acts through $(\phi_1^*)^{-1}$ on $\P^1(\K)$, but beware that if $\phi_1^{-1}(x)$ is well defined over $\k((x))$, then this need not be the same as $(\phi_1^{-1})^*$. In the complex case $\k = \C$ they will agree if $\phi_1'(0) \in \R_{>0}$ and $\phi^*$ is chosen as in \autoref{sec:puiseux}. See the example below and also \autoref{rmk:complexnearlyfunctorial}.
\end{rmk}

\begin{ex}
 Let $g(x) = h(x) = -x$. Then $g^* \circ h^* = \prim^*\circ (h \circ g)^* = \prim^* \ne \id = (h \circ g)^*$ by \autoref{prop:nearlyfunctorial}. Therefore $h_* \circ g_* = \prim_* \ne \id = (h \circ g)_*$. In particular if $\zeta = x^{\frac 1m} \in \P^1(\hk)$ then $\prim_*(\zeta) = (\prim^*)^{-1}(x^{\frac 1m}) = e^{-\frac{2\pi i}{m}} x^{\frac 1m} \ne x^{\frac 1m}$.
\end{ex}

\begin{prop}\label{prop:commwithG}
 Let $\omega^* \in G$ and $\phi_*$ be a $\k$-rational skew product with scale factor $1/n$. Then $\omega^* \circ \phi^* = \phi^* \circ \omega^{*n}$.
\end{prop}

\begin{proof}
 $\phi^* = \phi^*_2 \circ \phi_1^*$ and so we only need to prove that $\omega^*$ and $\phi_2^*$ commute since $\omega^* \circ \phi_1^* = \phi_1^* \circ \omega^{*n}$ by \autoref{cor:galois:fieldisometry}. Because $\phi_2 \in \k((x))(y)$, we can write $\phi_2^*(\omega^*(y)) = \phi_2^*(y) = \phi_2(x, y)$ and $\omega^*(\phi_2^*(y)) = \omega^*(\phi_2(x, y)) = \phi_2(\omega^*(x), y) = \phi_2(x, y)$. On the other hand $\phi_2^*|_{\hk} = \id$ and $\omega^*$ is an isomorphism of $\hk$, so they commute. We have shown commutativity on both $\hk$ and the generator $y$, hence also for $\hk(y)$.
\end{proof}

\begin{cor}\label{cor:commwithG}
  Let $\omega^* \in G$ and $\phi_*$ be a $\k$-rational skew product with scale factor $1/n$. Then $\omega_*^n \circ \phi_* = \phi_* \circ \omega_*$
\end{cor}

\begin{thm}\label{thm:galois:extensionnearlyunique}
 Let $\phi^*, \psi^*$ be \dilating{} skew endomorphisms over $\hk$ which are defined and equal over $\k((x))$. Then $\psi^* = \phi^* \circ \omega^*$ and $\psi_* = \omega_* \circ \phi_*$ for some $\omega^* \in G$.
\end{thm}

\begin{proof}
 We do this separately over $\hk$ and then with $y$. In the latter case, $\psi^*(y) = \phi^*(y) \in \k((x))(y)$, therefore $\psi^*(y) = (\phi^* \circ \omega^*)(y)$ for any $\omega^* \in G$. For the former case, by \autoref{prop:galois:extensionnearlyunique} it suffices to show that $\psi^* = \phi^*$ over $\k((x))$, but this is by hypothesis. So we can find an $\omega^* \in G$ such that $\psi_* = \omega_* \circ \phi_*$ on $\hk$. \refproplowerstarfunctorial{}
 finishes the proof.
\end{proof}

\begin{thm}\label{thm:nearlyfunctorial}
 Let $\phi, \psi$ be $\k$-rational skew products. Then \[\psi_* \circ \phi_* = \omega_* \circ (\psi \circ \phi)_*\]
 for some $\omega^* \in G$.
\end{thm}

\begin{proof}
By \autoref{thm:galois:extensionnearlyunique} it suffices to show that $\phi^* \circ \psi^* =  (\psi \circ \phi)^*$ over $\k((x))(y)$, and since these are $\k$-algebra homomorphisms over $\k[[x]]$, we only need to check $x$ and $y$. Indeed
 \[\phi^* \circ \psi^*(x) = \phi^* (\psi_1(x)) = \psi_1(\phi_1(x)) = (\psi \circ \phi)^*(x).\]
 In the case of $y$, we have on one hand that $(\psi \circ \phi)^*(y)$ is the second component of
 \[\psi \circ \phi(x, y) = \psi(\phi_1(x), \phi_2(x, y)) = (\psi_1(\phi_1(x)), \psi_2(\phi_1(x), \phi_2(x, y)))\]
 and on the other
  \[\phi^* \circ \psi^*(y) = \phi^*(\psi_2(x, y)) = \psi_2(\phi^*(x), \phi^*(y)) = \psi_2(\phi_1(x), \phi_2(x, y)).\]
\end{proof}

\begin{rmk}\label{rmk:complexnearlyfunctorial}
 One can be more specific for $\C$-rational skew products using \autoref{prop:nearlyfunctorial}.
\end{rmk}

\begin{prop}
 Let $\phi = (\phi_1, \phi_2)$ be a Puiseux skew product. Then $\phi_{1*} : \P^1_\an \to \P^1_\an$ fixes all directions at $\zeta(0, 1)$. Whence, if $\phi_*(\zeta(0, 1)) = \zeta(0, 1)$, then for every direction $\bvec v$ at $\zeta(0, 1)$, $\phi_\#(\bvec v) = (\phi_2)_\#(\bvec v)$.
\end{prop}

\begin{proof}
 $(\phi_1)_*(c) = c \ \forall c \in \P^1(k)$ and thus $\phi_{1*}(\vec v(c)) = \phi_{1*}(D_\an(c, 1)) = D_\an(c, 1) = \vec v(c)$.
\end{proof}

Hence we also find that $\phi_1$ does not interfere with reduction of $\phi = (\phi_1, \phi_2)$. In fact, if we are careful to write $\phi_2(x, y) = f(y)/g(y) \in \K(y)$ as a fraction of $f, g \in \bO_\K[y]$, two polynomials with respect to $y$ with no common factors, then the reduction of $\phi_*$ to $\overline \phi = \phi_2(0, y)$ is obtained by setting $x = 0$. \[\overline{\phi}(y) = \overline{\phi(x, y)} = \overline{\phi(\phi_1^{-1}(x), y)} = \phi_2(0, y)\]

\begin{cor}[of \autoref{thm:skew:reduction}]\label{cor:skew:reduction}
 Let $\phi = (\phi_1, \phi_2)$ be a $\k$-Puiseux skew product. Then the reduction $\overline{\phi}(y)$ is the same as the rational mapping $y \mapsto \lim_{x\to 0}\phi_2(x, y)$, or simply $\phi_2(0, y)$. In particular, the degree in a fibre is the degree of the reduction.%
 \[\deg(\phi_2(0, y)) = \deg(\overline{\phi})\] 
\end{cor}

\begin{defn}
 Let $\phi = (\phi_1, \phi_2)$ be a rational skew product such that $\phi_1$ fixes $x=0$ (or `over $0$'). Then \[\phi_* : \V \longrightarrow \V\] acts by \[\phi_*(v)(\phi) = v(\phi^*(\phi)).\]
\end{defn}

$\V$ has the positive quality that the functor $\phi \mapsto \phi_*$ taking a skew product on a ruled surface (in the neighbourhood of a fibre) to a $\k$-rational skew product on $\V$ is faithful. We shall usually extend this to $\P^1_\an(\hk)$ and this is unique up to Galois actions. From now on we consider $\P^1_\an(\K)=\V$ and $\P^1_\an(\hk)/G$ to be the same space, so for example we understand that $[\zeta]_G \in \V$ and $\mathfrak p = \pi$.

\begin{prop}~
\begin{enumerate}
 \item Let $\phi_*$ be a $\k$-rational skew product. Then the following diagram commutes
   \[
\begin{tikzcd}
 \P^1_\an(\hk) \arrow[two heads, swap]{d}{\mathfrak p} \arrow{r}{\phi_*}& \P^1_\an(\hk)\arrow[two heads]{d}{\mathfrak p} \\
 \V \arrow[swap]{r}{\phi_*}& \V
\end{tikzcd}
\]
 meaning that $\phi_*([\zeta]_G) = [\phi_*(\zeta)]_G$, and this representation is well defined.
\item Let $\phi, \psi$ be rational skew products over $0$. Then as maps on $\V$, $(\phi \circ \psi)_* = \phi_* \circ \psi_*$.
\end{enumerate}
 \end{prop}
 
\begin{proof}
 Apply the definition of $\phi_*$ with \autoref{prop:commwithG} and \autoref{thm:nearlyfunctorial}.
\end{proof}

On $\V$ we can put the same weak topology as $\P^1_\an$; $v_j \to v$ iff $v_j(f) \to v(f)\ \forall f$. In fact by defining the open sets to be images of open sets $U$ under $\mathfrak p$, we give it the weakest topology making $\mathfrak p$ continuous. Furthermore, we define an (open/closed) affinoid to be the image of an (open/closed) affinoid in $\P^1_\an$. Hence for any (topologically) connected affinoid $U \subseteq \V$, we have that $\mathfrak p^{-1}(U)$ is the union of a Galois orbit of some connected affinoid $W \subseteq \P^1_\an$.

\subsection{Mappings and multiplicity}\label{sec:galois:mappingsandmult}

We've talked about multiplicity, its subtrees, and mappings by skew products, but now it's time to put the three together. The following proposition is useful in applications, but we intend to prove a generalisation of it. The proof is an exercise using \autoref{cor:commwithG}.

\begin{thm}\label{thm:multiplicitydivides}
 Let $\phi_*$ be a simple rational skew product, and $\zeta \in \P^1_\an$. Then $\mm(\phi_*(\zeta)) \mid \mm(\zeta)$. Hence $\phi_*(\T_m) \subseteq \T_m$ for any $m \in \N_+$.
\end{thm}

 If we improve \autoref{cor:commwithG} then we can also drastically improve this result and expand it to non-rational simple skew products. First we make some definitions which will be useful throughout this document, then we will rewrite \ref*{cor:commwithG}.%

In the previous sections we avoided analysing how $G$ acts on $\hk(y)$, but there is only a little more to it. A Galois automorphism $\omega^* \in G$ acts trivially on $y$, so we can easily compute multiplicity by studying coefficients. %
We go further by copying most of \autoref{prop:multiplicitydefn}.

\begin{defn}
 Let $\phi \in \hk(y)$ then we define the \emph{multiplicity of $\phi$}, $\mm(\phi) = |\Orb_\prim(\phi)|$. Suppose that $\phi_* = \phi_{1*} \circ \phi_{2*}$ is a skew product over $\hk$. %
 We define the \emph{multiplicity of $\phi_*$}, $\mm(\phi_*)$ or less formally $\mm(\phi)$, to be $\mm(\phi_2)$ where $\phi^*(y) = \phi_2 \in \hk(y)$.
\end{defn}

\begin{prop}\label{prop:multiplicitywithy}
 Let $\phi \in \hk(y)$ and $m = \mm(\phi)$.
 \begin{enumerate}[label=(\roman*)]
  \item Write
 \[\phi(y) = \left(\sum_{k = 1}^M a_k y^k\right) / \left(\sum_{k = 1}^N b_k y^k\right).\] Then 
 $m = \LCM(\mm(a_k)_{k=1}^M \cup \mm(b_k)_{k=1}^N)$.
 \item $m$ is finite if and only if $\phi \in \K(y)$.
 \item $m$ is the smallest integer such that $\phi \in \k((x^{\frac{1}{m}}))(y) < K(y)$, or $\infty$ otherwise.
 \item $m = |\Orb_G(\phi)|$ and $\Orb_G(\phi) = \Orb_\prim(\phi)$.
\end{enumerate}
In particular, a skew product $\phi_*$ over $\hk$ is Puiseux iff it has finite multiplicity, $\mm(\phi) < \infty$.
\end{prop}

\begin{prop}
 Let $\phi = (\phi_1, \phi_2)$ be a skew product of multiplicity $\mm(\phi) < \infty$ and scale factor $1/n$. If $\omega^* \in G$, then $\phi^* \circ \omega^{n*} = \omega^* \circ \phi^*$\ if and only if\ $\omega_{\mm(\phi)} = 1$. In particular, $\prim^{*M} \circ \phi^* = \phi^* \circ \prim^{*nM}$\ if and only if\ $\mm(\phi) \mid M$.
\end{prop}

\begin{proof}
  Take a decomposition $\phi^* = \phi^*_2 \circ \phi_1^*$ so that $m = \mm(\phi) = \mm(\phi_2)$ and $\phi_1$ has scale factor $1/n$. Then $\omega^* \circ \phi_1^* = \phi_1^* \circ \omega^{n*}$ by \autoref{cor:galois:fieldisometry} and $\phi_2^*(\omega^*(y)) = \phi_2^*(y) = \phi_2 \in \k((x^\frac 1m))(y)$ by \autoref{prop:multiplicitywithy}. First we show that $\omega^*$ and $\phi_2^*$ commute by looking at the action on $\hk$ and then on $y$, hence determining $\hk(y)$ altogether. This is trivial on $\hk$ since $\phi_2^*|_{\hk} = \id$. Next, consider that $\omega^*(x^\frac km) = \omega_m^k x^\frac km = x^\frac km$, therefore $\omega^*$ restricted to $\C((x^\frac 1m))$ is the identity; whence $\omega^*(\phi_2) = \phi_2$. Now\[\omega^*(\phi_2^*(y)) = \omega^*(\phi_2) = \phi_2 = \phi_2^*(y) = \phi_2^*(\omega^*(y))\] shows the commutativity on $y$. Now, \[\omega^* \circ \phi^* = \omega^* \circ \phi^*_2 \circ \phi_1^* = \phi^*_2 \circ \omega^* \circ \phi_1^* = \phi^*_2 \circ \phi_1^* \circ \omega^{n*} = \phi^* \circ \omega^{n*}.\] 
  
 Conversely, consider what happens if $\omega_m \ne 1$. By \autoref{prop:multiplicitywithy}, $\omega^*(\phi_2) = \prim^{*k}(\phi_2)$ for some $0 \le k < m$, and so by minimality of $m$, we must have $\omega^*(\phi_2) \ne \phi_2$. Therefore
 \[\omega^*(\phi_2^*(y)) = \omega^*(\phi_2) \ne \phi_2 = \phi_2^*(y) = \phi_2^*(\omega^*(y)),\] and so $\phi_2^*$ and $\omega^*$ do not commute as witnessed by $y$. Because $\phi_1^*$ is an isomorphism, this implies that $\phi^* \circ \omega^{n*} \ne \omega^* \circ \phi^*$ by the following chain of relations.
 \[\omega^* \circ \phi^*= \omega^* \circ \phi^*_2 \circ \phi_1^* \ne \phi^*_2 \circ \omega^* \circ \phi_1^* = \phi^*_2 \circ \phi_1^* \circ \omega^{n*} = \phi^* \circ \omega^{n*}\]
  
  Finally, $(\prim^k)_m = e^{k\frac{2\pi i}{m}}$ so $\prim^{*k}$ commutes with $\phi^*$ if and only if $m \mid k$.
\end{proof}

\begin{cor}\label{cor:puiseuxcommwithG}
 Let $\phi_*$ be a Puiseux skew product of multiplicity $m = \mm(\phi)$ and scale factor $1/n$. If $\omega^* \in G$, then $\phi_* \circ \omega_* = \omega_*^n \circ \phi_*$\ if and only if\ $\omega_m = 1$. In particular, $\phi_* \circ \prim_*^k = \prim_*^{nk} \circ \phi_*$\ if and only if\ $m \mid k$.
\end{cor}

\begin{thm}\label{thm:multiplicitydividesadv}
 Let $\phi_*$ be a Puiseux skew product with scale factor $1/n$, and multiplicity $m = \mm(\phi)$. If $\zeta \in \P^1_\an$, then $\mm(\phi_*(\zeta))$ divides $n \cdot \LCM(\mm(\zeta), \mm(\phi))$. Hence $\phi_*(\T_N) \subseteq \T_{nM}$ where $M = \LCM(\mm(\phi), N)$.
\end{thm}

\begin{proof}
 Let $\zeta \in \P^1_\an$ and $M = \LCM(\mm(\zeta), \mm(\phi))$; hence $\prim_*^M(\zeta) = \zeta$. 
 One has $\prim_*^{nM} \circ \phi_* = \phi_* \circ \prim_*^M$
from \autoref{cor:puiseuxcommwithG} and therefore 
\[\prim_*^{nM}(\phi_*(\zeta)) = \phi_*(\prim_*^M(\zeta)) = \phi_*(\zeta).\]
 This shows that $\mm(\phi_*(\zeta)) \mid M$. The statement about subtrees follows.
\end{proof}

\begin{cor}\label{cor:genmultiplicitydivides}
 Let $\phi_*$ be a simple Puiseux skew product, and $\zeta \in \P^1_\an$. Then \[\mg(\phi_*(\zeta)) \mid \LCM(\mg(\zeta), \mm(\phi)).\]
\end{cor}

\begin{prop}\label{prop:mappingmult}
Let $\phi_*$ be a simple Puiseux skew product.
\begin{enumerate}[label=(\alph*)]
 \item $\mm(\phi_* \circ \psi_*) \mid \LCM(\mm(\phi_*), \mm(\psi_*))$ where $\mm(\phi_*), \mm(\psi_*) \ne \infty$.
 \item Any Type II point $\zeta$ can be moved to the Gauss point $\zeta(0, 1)$ with a change of coordinates $\eta \in \PGL(2, K)$, where $\mm(\eta) < \infty$. In particular, for a Puiseux $\phi_*$, the lift by $\eta_*$ to $\tilde \phi_* = \eta_*\circ \phi_* \circ \eta_*^{-1}$ is also Puiseux: $\mm(\tilde f) \mid \LCM(\mm(\phi), \mm(\eta))$. Moreover, $\mm(\eta) = \mg(\zeta)$ is optimal and we can pick $\eta$ to be an affine transformation of $\C\left(x^\frac 1{\mg(\zeta)}\right)$. %
\end{enumerate}
\end{prop}

\begin{proof}
 For (a), let $M = \LCM(\mm(\phi_*), \mm(\psi_*))$, then by \autoref{cor:puiseuxcommwithG} it is clear that $\mm(\phi_* \circ \psi_*) \mid M$ since 
 \[\phi_* \circ \psi_* \circ \prim_*^M = \phi_*\circ \prim_*^M \circ \psi_* =\prim_*^M \circ \phi_* \circ \psi_*.\]
 Consider $\zeta = \zeta(a, |x|^q)$ as in (b), with $a \in \C(x^\frac 1m)$ chosen as in \autoref{prop:multiplicity}. Then define
 \[\eta(x, y) = \left(x, \frac{y - a}{x^q}\right)\]
 Clearly $\mm(\eta) = \LCM(\mm(a), q) = \mg(\zeta)$, per \autoref{cor:multsubtreevalency}, moreover $\eta$ to be an affine transformation of $\C\left(x^\frac 1m\right)$. %
 Finally we check that $\eta(\zeta) = \zeta(0, 1)$; indeed $\eta$ has no finite poles, so $D = D(a, |x|^q)$ is mapped to $D(\eta(a), |x|^q \cdot |x|^{-q}) = D(0, 1)$ by the usual power series argument.
\end{proof}

\begin{prop}\label{prop:fatoujuliaginvt}
Let $\phi_*$ be a simple rational skew product. Then $\F_{f, \an}$ and $\J_{f, \an}$ are invariant under the action of the Galois group $G$.
\end{prop}

\begin{proof}
 Suppose $\zeta \in \F_{f, \an}$. Then there is an open set $U \ni \zeta$ such that $\bigcup_{n=0}^\infty f_*^n(U)$ omits infinitely many points.
 Let $\omega^* \in G$. Then by \autoref{cor:commwithG}
 \[\bigcup_{n=0}^\infty f_*^n(\omega_*(U)) = \bigcup_{n=0}^\infty \omega_*(f_*^n(U)) = \omega_*\left(\bigcup_{n=0}^\infty f_*^n(U)\right)\] omits infinitely many points, and so we have shown that $\omega_*(\zeta)$ has a dynamically stable neighbourhood $\omega_*(U)$. Therefore $\zeta \in \F_{f, \an} \implies \omega_*(\zeta) \in \F_{f, \an}$. The rest of the proposition is now obvious by \reffatoujuliabasics.
\end{proof}

This result may be extended to simple Puiseux skew products if we restrict to certain Galois actions. It can also be extended to non-simple skew products given dynamical input.

This result shows that the concepts of Fatou and Julia points descend perfectly well to the valuative tree $\V$ when we have a simple rational skew product $\phi_*$.


\section{Models and Dual Graphs}\label{sec:corresp}

In the next two subsections we explain how to transfer the geometric information from the geometric skew product to a skew product on the Berkovich projective line. It will turn out that Type II points naturally correspond to divisors on our ruled surface, and their dynamics corresponds too. The reader focused on a practical understanding may skip this first subsection.

\subsection{Models}\label{sec:models}

We wish to compare different birational models of some original ruled surface, perhaps allowing only modifications to a fixed fibre. The presentation below is influenced by those of Bosch \& L\"uktebohmert \cite{BL85, BL93}, of Baker, Payne \& Rabinoff \cite{BPR}, of Boucksom, Favre \& Jonsson \cite[\S 1]{BFJ}, and of DeMarco \& Faber \cite[\S 4.1]{DeF2}. See also \cite[\S8, \S10]{Liu}.

Let $h : X \to B$ be a birationally ruled surface. Let $b \in B$ and consider its local ring $\bO_{B, b}$ on $B$, this is a discrete valuation ring, and let $m_b = (x)$ be its maximal ideal. The fraction field is $\ell = \Frac(\bO_{B, b}) = \k(B)$ and the residue field is $\k = \bO_{B, b}/m_b = \k(b)$ (in the complex case, $\k = \C$). The order of vanishing norm $\abs\cdot$ with respect to $x$ equivalently measures the order of vanishing of functions at $b$; this field $(\ell, \abs\cdot)$ is non-Archimedean with ring of integers $\ell^\circ = \bO_{B, b}$. For brevity, let us simply write $\bO$ for $\bO_{B, b}$ where $b$ is understood from context.

\begin{defn}
 Let $\ell$ and $\bO = \ell^\circ$ be as above. A \emph{model of $\P^1_\ell$ over $\bO$} is a pair $(Y_\bO, \iota)$ where $Y_\bO$ is a normal (possibly not smooth) flat projective $\bO$-scheme, whose generic fibre $Y_\eta = Y_\bO \times_{\Spec(\bO)} \Spec(\ell)$ is isomorphic to $\P^1_\ell$ via $\iota : Y_\eta \to \P^1_\ell$. A birational map of models $\rho : (Y_\bO, \iota) \dashto (Y'_\bO, \iota')$ is a birational map $\rho : Y \dashto Y'$ over $\Spec(\bO)$ such that $\iota' \circ \rho_\ell \circ \iota^{-1}$ is the identity on $\P^1_\ell$. If $\rho$ is a birational morphism $\rho : (Y_\bO, \iota) \to (Y'_\bO, \iota')$ then we say $(Y_\bO, \iota)$ \emph{dominates} $(Y'_\bO, \iota')$.
\end{defn}

A model $(Y_\bO, \iota)$ of $\P^1_\ell$ over $\bO$ has a natural pair of coordinates variables $(x, y)$ determined by a uniformiser $x \in \bO$ and $y = \iota^*(y')$ where $\Spec \ell[y'] = \A^1_\ell$ is our favourite affine line within $\P^1_\ell$. 

The most natural way to generate models is as localisations of (or infinitesimal neighbourhoods of) fibres of a birationally ruled surface $h : X \to B$. Let $b \in B$ and $X_b$ denote the fibre $h^{-1}(b)$. We will be interested in performing blowups (and perhaps blowdowns) over this fibre alone. If we do so, the geometry of $X$ away from $X_b$ will remain the same. This means we may focus on the localisation $X_\bO = X \times_B \Spec(\bO)$, which is a scheme over $\Spec(\bO)$. The $\bO$-scheme $X_\bO$ has two fibres over the points of $|\Spec(\bO)| = \set{m_b, \eta}$ where $\eta$ is the generic point $(0) \triangleleft \bO$. The \emph{closed} or \emph{special fibre} above $\Spec(\k(b)) = \Spec(\k) \hookrightarrow \Spec(\bO)$ is $X_\k$ is $X_\bO \times_{\Spec(\bO)} \Spec(\k(b)) \cong X_b$. The generic fibre $X_\eta$ is naturally isomorphic to $\P^1_\ell$. To see this in detail, pick a closed point $b' \in U \subset B$ such that $U \times_B X \cong U \times_\k \P^1_\k$ since the general closed fibre of $X$ is $\P^1_\k$. Then the generic fibre of $X$ is the fibre above \[\Spec(\ell) \cong \Spec(\k(\bO_{B, b'})) = \Spec(\k(B)) \hookrightarrow \Spec(\bO) \hookrightarrow U \hookrightarrow B\] independent of $b'$. We can use associativity of fibre products to see that \[\Spec(\ell) \times_{\Spec(\bO)} X_\bO \cong \Spec(\ell) \times_{\Spec(\bO)} \Spec(\bO) \times_B X = \Spec(\ell) \times_B X\]
\[=\Spec(\ell) \times_U U \times_B X \cong \Spec(\ell) \times_U U \times_\k \P^1_\k \cong \Spec(\ell) \times_\k \P^1_\k = \P^1_\ell.\] 

\begin{defn}
 Let $h : X \to B$ be a birationally ruled surface, $b \in B$.%
  A \emph{global model of $X$ over $b \in B$} is a birationally ruled normal (but possibly singular) surface $g : Y \to B$ for which there is a $\iota : Y \dashto X$ birational such that the $B$-morphism $\iota : Y \sm Y_b \to X \sm X_b$ is an isomorphism. 
 A birational map $\rho : Y \dashto Y'$ of models over $b$ is one such that $\iota' \circ \rho_\ell \circ \iota^{-1}$ restricts to the identity on $X\sm X_b$. 
 Further, given finitely many closed points $b_1, \dots, b_n \in B$, we make a similar definition for a \emph{global model of $X$ over $(b_j)_{j=1}^n$} where the map $\iota$ is an isomorphism away from each $X_{b_j}$. Define $\Div_b(X)$ the subgroup of divisors supported on $X_b$.
\end{defn}

 By the discussion above, a global model of $X$ over $b$ naturally gives rise to a model of $\P^1_\ell$ over $\bO_{B, b}$. %
For brevity we may simply refer to $X$ as a model, with the local over $\bO_{B, b}$, $b$, or perhaps several fibres as understood from context.
We do not assert that the closed fibre $Y_\k$ of $Y_\bO$ is reduced.

A global model $Y$ (over $b$) is always dominated by a global model $Y'$ which is both \emph{smooth} (along $Y'_b$) and where the closed fibre $Y'_b$ has \emph{simple normal crossing support}; this is by resolving singularities of the surface and the fibre through blowups. A curve has simple normal crossings if it is reduced, and the intersections are ordinary double points (i.e.\ locally like $\Spec[z, w]/(zw)$). This implies that $Z_1 \cdot Z_2 = 1$ for any intersecting irreducible components $Z_1 \ne Z_2$. A curve has simple normal crossing support if its induced reduced subscheme is a simple normal crossing curve.

\begin{defn}
 Let $Y$ be a (global) model of $\P^1_\ell$ (or $X$) over $\bO$ (over $b$) with reduced closed fibre $Y_\k$ ($=Y_b$). Then we will say $Y$ is \emph{SC} iff
\begin{itemize}
 \item each closed point of $Y_\k$ is contained in at most two irreducible components, and
 \item for each pair $Z_1, Z_2$ of irreducible components in $Y_\k$, the intersection $Z_1 \cap Z_2$ is at most one point. %
\end{itemize}
Further, if each crossing of components $Z_1, Z_2$ above is normal we say $Y$ is \emph{SNC} (has simple normal crossings).
\end{defn}

\begin{rmk}
 In various other papers on analytic curves \cite{BL85, BPR} (see also \cite{DeF2}), SNC is equivalent to saying $Y$ is a \emph{semistable} model of $\P^1_\ell$. Additionally we would say that $Y$ is \emph{strongly semistable} if the the irreducible components of $Y_\k$ are smooth. The author has chosen to use the SNC terminology of \cite{BFJ} primarily to avoid confusion with other notions of `stability' we are interested in. Smoothness of components is natural for any smooth model. The second condition will manifest when we recognise the dual graph of an SNC model is a tree.
\end{rmk}

It is easy to have an SNC model which is not smooth -- a quotient singularity results from the blowdown of a $\P^1$ divisor with self-intersection $-2$. 
Every smooth model is SNC, but an SNC model may not be smooth.

To connect these models with the Berkovich projective line over a complete non-Archimedean field, we will want to work on the completions of the $\bO$-curves. Regardless of the genus of $B$, the completion of $(\bO_{B, b}, \abs\cdot)$ is isomorphic to $(\k[[x]],\abs\cdot)$ by the Cohen Structure Theorem \cite[p33]{Hart}, and this is the same as the $m_b$-adic completion $\hat\bO = \varprojlim \bO_{B, b}/m_b^n$. Hence also the completion of $\ell$ is isomorphic to $\k((x)) = \Frac \k[[x]]$ and we write $\K$ for its algebraic closure, the field of Puiseux series, and $\hk$ for the subsequent completion, the Levi-Civita series. Hence, the formal completion of $\Spec(\bO)$ at $m_b$ and equivalently the formal completion of $B$ at $b$ is the affine formal scheme $\FSpec(\hat \bO) = (\fml B, \hat\bO)$ whose special/closed fibre and underlying topological space is a point.%

 Let $h : Y \to \bO$ be a model of $\P^1_\ell$. The accompanying \emph{formal model} is $(\fml X, \bO_{\fml X})$ the admissible formal scheme obtained from completing $Y$ along its closed fibre $Y_\k$. The structure sheaf of $\fml X$ is $\bO_{\fml X} = \varprojlim \bO_Y/h^*\bO^n \cong \bO_Y \otimes_{\bO} h^*\hat\bO$. We denote by $\fml X_0$ the closed fibre $\fml X_b = (\fml X, \bO_{\fml X}/\hat m_b) \cong (Y_b, \bO_{Y_b})$. 
 Note that $\fml X$ is constructed by taking the affine subschemes $\Spec A_j$ of $Y$ and gluing together the affine formal scheme $\FSpec(\hat A_j)$ where $\hat A_j = \varprojlim A_j/I_j$, and $Y\cap \Spec(A_j) = \Spec(A_j/I_j)$.
 It is common to not take the special fibre to be a formal scheme but use the \emph{Raynaud generic fiber functor} $\fml X \mapsto \fml X_\an$. This involves gluing together the Berkovich spectra $\mathcal M \left(\hat A_j \otimes_{\k[[x]]} \k((x))\right)$. If we glued instead the algebraic spectra $\Spec\left(\hat A_j \otimes_{\k[[x]]} \k((x))\right)$ we would obtain the fibre product \[Y \times_{\Spec(\bO)} \Spec(\hat\bO) \times_{\Spec(\hat\bO)} \Spec(\k((x))) = Y \times_{\Spec(\bO)} \Spec(\k((x))) \cong Y_\eta \times_{\Spec(\ell)} \Spec(\k((x)))\] which is isomorphic to $\P^1_{\k((x))}$ via the lift of $\iota : Y_\eta \to \P^1_\ell$. Therefore, there is an induced isomorphism between $\fml X_\an$ and $\P^1_{\k((x)), \an}$, which we previously denoted by $\P^1_\an(\k((x)))$. Using similar constructions we can take base changes to either $\K$ or $\hk$, and the closed fibre $\fml X_0$ will remain isomorphic to $Y_\k$ as $\k$-schemes. Herein we will write $\P^1_\an$ in place of $\fml X_\an$ and return to the notation $\P^1_\an(\k((x)))$ to remind the reader of the base field\footnote{Commonly, $\P^1_\an(K)$ would be interpreted as the $K$-points of the analytic space $\P^1_\an$.}.
 
 For each model $Y$ of $\P^1_\ell$, let \[\redct_Y : \P^1_\an(\k((x))) \to Y_\k\] be the `reduction map' described by Berkovich in \cite[\S 2.4]{Berk}. Note that in his book, the spectrum is (an affinoid piece of) $\P^1_\an(\k((x)))$ and the codomain is a particular affine piece $\Spec(\mathcal A) \subset \tilde Y$ of the reduction of the variety $Y$ which is the same as $Y_b = Y_\k$. %
In the context of models, it is easier to characterise the reduction map on an affine piece $\Spec(A_j/I_j) \subset Y_\k$. In the ambient model $Y$, a generic point of $Y_\k$ is represented by some dimension $1$ prime ideal of $A_j$. Assuming $\redct_Y(\zeta) \in \Spec(A_j/I_j)$, the reduction map is given by 
\[\redct_Y : \zeta \longmapsto \set[a + I_j\in A_j/I_j]{\norm[\zeta]{a} < 1}.\] 

We can define the map globally and more precisely as follows. Let $\zeta \in \P^1_\an$ and first define the \emph{home} point $p = \ker(\norm[\zeta]\cdot) \in \P^1_\ell \cong Y_\eta \hookrightarrow Y$; the point we want is either $p$ or a specialisation of it. Now let $(\ell(\zeta), \norm[\zeta]\cdot)$ be the residue field $\bO_{Y, p}/m_p$ equipped with $\zeta$ and let $R_\zeta = \ell(\zeta)^\circ$ be its valuation ring. The former will be isomorphic to either $\ell(Y)$ or $\ell$. Therefore there is an inclusion $\ell \hookrightarrow \ell(\zeta)$ inducing another map $\bO \hookrightarrow R_\zeta$ since $\zeta$ extends the norm on $\ell$; this gives a map $\Spec(R_\zeta) \to \Spec(\bO)$.
\[
 \begin{tikzcd}
 \Spec(\ell(\zeta)) \arrow{r} \arrow{d} & Y \arrow{d}\\
\Spec(R_\zeta) \arrow[swap]{r} \arrow[dashed]{ru} & \Spec(\bO)
\end{tikzcd}
\]
By the valuative criterion for properness, there is a unique lift in the diagram above which maps the generic point $\eta_\zeta$ to $p$ and the closed point $m_\zeta = \set[f \in R_\zeta]{\norm[\zeta]{f} < 1} \in \Spec(R_\zeta)$ to the \emph{centre} $\redct_Y(\zeta)$. Since the image of $m_\zeta$ in $\Spec(\bO)$ is the closed point, we see that $\redct_Y(\zeta)$ lies in $Y_\k$. We remark that this construction did not depend so much on the base field of $\P^1_\an$, and so provides a map $\redct'_Y : \P^1_\an(\K)$ too, however the next result uses the fact that the completion of $\ell$ is $\k((x))$.

Reduction gives a semi conjugacy from a $\k$-rational skew product $\phi_* : \P^1_\an \to \P^1_\an$ on the Berkovich projective line to the restriction of $\phi$ to fibres $Y_b$ for any global model $Y$ of $X$ over $b \in B$ with codomain over $\phi_1(b) \in B$. Denote by $\redct'_X : \P^1_\an(\hk)=\P^1_\an(\K) \to Y_b$ the extension of $\redct_Y$ through base change. We have that $\redct'_Y = \mathfrak p \circ \redct_Y$.

 \[
\begin{tikzcd}[column sep = 2em, row sep = 3em]
 \P^1_\an(\K) \arrow[two heads, swap, bend right=75]{dd}{\redct'_Y} \arrow[two heads]{d}{\mathfrak p} \arrow{r}{\phi_*}& \P^1_\an(\K) \arrow[two heads, bend left=75]{dd}{\redct'_Y} \arrow[swap, two heads]{d}{\mathfrak p} \\
  \P^1_\an(\k((x))) \arrow[two heads]{d}{\redct_Y} \arrow{r}{\phi_*}& \P^1_\an(\k((x))) \arrow[two heads, swap]{d}{\redct_Y} \\
 Y_b \arrow[swap]{r}{\phi}& Y_{\phi_1(b)}
\end{tikzcd}
\]

\begin{prop}[Berkovich {\cite[2.4.2]{Berk}}]
 The reduction maps $\redct_X : \V \to X_b$ and $\redct'_X : \P^1_\an \to X_b$ are anti-continuous with respect to the weak and Zariski topologies. That is, for any closed set $Y \subset X_b$, $(\redct'_X)^{-1}(Y)$ and $\redct_X^{-1}(Y)$ are open.
 \end{prop}

Below we recall another important result due to Berkovich, translated and specialised to our context. It says that $\Gamma = \redct^{-1}(Y_{\text{gen}})$ is a finite set where $Y_{\text{gen}}$ are the generic points of $Y_b$. This motivates the study of vertex sets in the next subsection.
 
\begin{prop}[Berkovich {\cite[2.4.4]{Berk}}]\label{prop:galois:berkred}
 Let $Y$ be a model of $\P^1_\ell$ and $Y_\k$ be its closed fibre. Then
\begin{enumerate}
 \item The reduction map $\redct_Y : \P^1_\an(\k((x))) \to Y_\k$ is surjective.
 \item For any generic point $\eta$ in $Y_\k$, there exists a unique point $\zeta \in  \P^1_\an(\k((x)))$ such that $\redct_Y(\zeta) = \eta$.
\end{enumerate}
\end{prop}

\subsection{Vertex Sets, and \texorpdfstring{$\Gamma$-Domains}{Gamma-Domains}}

Now we study the relationship between models and sets of Type II points. DeMarco and Faber \cite{DeF1} provide the next definitions.

\begin{defn}
 Let $\Gamma \subset \P^1_\an$ be a finite set of Type II points -- a \emph{vertex set}. Then $\P^1_\an\sm \Gamma$ is the union of a collection $\mathcal{S}(\Gamma)$ of disjoint open connected affinoids, each of which we call a $\Gamma$\emph{-domain}. If a $\Gamma$-domain has one boundary point, we call it a $\Gamma$\emph{-disk}, and if it has two we call it a $\Gamma$\emph{-annulus}. 
 If $\Gamma \subset \P^1_\an(\K)$ is $G$-invariant, then projecting to $\P^1_\an(\k((x)))$ we have a vertex set $\Gamma_G$ and $\Gamma_G$-domains $\mathcal{S}_G(\Gamma)$. 
 We let $\mathcal{S}^+(\Gamma) = \mathcal{S}(\Gamma) \cup \Gamma$ be the set of $\Gamma$-domains and the points in $\Gamma$ itself. %
 Given a model $Y$ of $\P^1_\ell$ over $\bO$, we define $\Gamma(Y)$ to be the vertex set $\redct^{-1}(Y_{\text{gen}})$.
\end{defn}

\begin{cor}\label{cor:galois:pointcorrespondence}
 The reduction map $\redct_X$ induces a bijection $\redct^+ : \mathcal{S}^+(\Gamma(X)) \to X_b$. The image of a $\Gamma$-domain $U$ is a closed point and the image of $v \in \Gamma$ is the generic point $\eta_j$ of a curve $E_j \subset X_b$.
\end{cor}

There is a correspondence between global models of $X$ over $b$ and vertex sets, by the second part of the previous proposition. This relationship is quite well understood; the author recommends reading both the accounts by Baker, Payne \& Rabinoff \cite{BPR} and by Boucksom, Favre \& Jonsson \cite{BFJ}. This includes the following fundamental result.

\begin{prop}[{\cite[Theorem 4.11]{BPR}}, {\cite[Proposition 3.6]{BFJ}}]\label{prop:galois:vertexextension}\label{thm:corresp:modelexistence}
 The association $Y \mapsto \Gamma(Y)$ is a bijection between the collection of models of $\P^1_\ell$ over $\bO$ (up to isomorphism) and vertex sets in $\P^1_\an(\k((x)))$. Furthermore, $Y'$ dominates $Y$ if and only if $\Gamma(Y') \supset \Gamma(Y)$.%
\end{prop}

There is a relationship between SC models $Y$ and those vertex sets which include all the (topological) vertices for the finite (sub)graph $\Hull(\Gamma(Y)) \subset \P^1_\an(\k((x)))$. By the latter kind of vertex, we mean points of $\Hull(\Gamma(Y))$ which have valency at least $3$ i.e. where the set is not locally an interval. In the special case where the closed fibre $Y_b$ is already reduced, Bosch and L\"utkebohmert \cite[2.2, 2.3]{BL85} showed that a smooth closed point $p \in Y_b$ corresponds to an open Berkovich disk $\redct^{-1}(p)$ and an ordinary double point correspond to an open Berkovich annulus.
Therefore, any such model $Y$ has a vertex set $\Gamma = \Gamma(Y) \subset \P^1_\an$ with the property that every $\Gamma$-domain is either a disk or an annulus. Conversely, if the closed point $p = \redct_Y(U)$ associated to a $\Gamma$-domain $U$ which is an annulus, then $p$ is the intersection of exactly two prime divisors in $X_0$, which are the images under $\redct_Y$ of $\partial U$; if $U$ is a disk then $p$ must be a non-singular point on a unique irreducible $E \subset X_0$.

\begin{prop}\label{prop:galois:reductionstructure}
 Let $Y$ be an SC (resp.\ SNC) model of $X$ over $\bO$, $\Gamma(Y) \subset \P^1_\an$ be the associated vertex set, and let $p \in Y_b$.
\begin{enumerate}
\item $p$ is the generic point of a curve if and only if $\redct^{-1}(p)$ is a single Type II point $\zeta \in \Gamma(Y)$.
 \item $p$ is a non-singular closed point if and only if $\redct^{-1}(p)$ is a $\Gamma(Y)$-disk.
\item $p$ is a singular closed point (resp.\ ordinary double point) if and only if $\redct^{-1}(p)$ is a $\Gamma(Y)$-annulus.
\end{enumerate}
\end{prop}

\begin{defn}
 Let $X$ be a model and $p \in X_0$. We say that a closed point $q$ or curve $E$ is \emph{infinitely near} $p$ if there is a birational morphism $\pi : Y \to X$ such that $q$ or $E$ is in $\pi^{-1}(p)$.
\end{defn}

Suppose that $\zeta \in \P^1_\an$ reduces to a curve $\redct_Y(\zeta) = E \subset Y$ which is infinitely near $p \in X$. It follows that $\redct_X(\zeta) = p$ because $\redct_X = \pi \circ \redct_Y$. The following is an immediate as a corollary of \autoref{prop:galois:vertexextension}.

\begin{cor}
 Let $X$ be a model and $p \in X_0$. The collection of irreducible curves infinitely near $p$ corresponds to the Type II points in the $\Gamma(X)$-domain $\redct_X^{-1}(p)$.
\end{cor}

\subsection{Intersections and Minimal Models}

Now we develop and recall some standard facts about intersections of divisors in this context of a ruled surface.
Any non-zero fibral divisor (sum of fibril components) has negative self intersection unless it is a multiple of a fibre. Moreover, the intersection form is negative semi-definite on fibral divisors. This is another way to see that any proper subset of components of $X_0$ can be blown down to a (probably singular) surface; see \cite{Art}, \cite{Mum}. The next proposition can be found e.g.\ in \cite[Proposition III.8.2, IV.7.3]{ATAEC}.

\begin{prop}\label{lem:farey:intersections}\label{prop:farey:negativity}
 Let $X$ be a global model over $0$ with $X_0 = \sum_{j=1}^n \mb_j E_j$. Let $F = \sum_{i=1}^n m_i$ be a fibral divisor. Then $F^2 \le 0$ and the following are equivalent:
\begin{enumerate}
 \item $F^2 = 0$
 \item $F \cdot F' = 0$ for any other fibral divisor $F'$
 \item $F = aX_0 = \sum_i a\mb_i E_i$ for some integer $a$.
\end{enumerate}
\end{prop}

\begin{lem}[{\cite[Proposition IV.4.3]{ATAEC}}]\label{lem:jhssec}
 Let $\pi : X \to \Spec(R)$ be a regular arithmetic surface over a Dedekind domain $R$, and let $\mathfrak p \in \Spec(R)$.
\begin{enumerate}
 \item Let $P_0 \in X_{\mathfrak p} \subset X$ be a closed point on the fiber of $X$ over $\mathfrak p$. Then
 \[X_{\mathfrak p}\ \text{is non-singular at\ } x \iff \pi^*(\mathfrak p) \nsubset \mathcal M^2_{X, x}.\]
  Here $\pi^*$ is the natural map $\pi^* : R \to \bO_{X, x}$ induced by $\pi$, and $\mathcal M_{X, x}$ is the maximal ideal of the local ring $\bO_{X, x}$ of $X$ at $x$.
 \item Let $P \in X(R)$. Then $X_{\mathfrak p}$ is non-singular at $P(\mathfrak p)$.
\end{enumerate}
\end{lem}

In other words, if $P : B \to X$ is a section (defined with local uniformizer $x$), then its intersection point $P_0 = P(\mathfrak p)$ is a non-singular point of $X_0$, and thus also the special fibre is locally integral at $P_0$, meaning $P_0$ is contained in a single, reduced component of $X_0$. 
Our initial setup began with a minimal model whose vertex set is $\set{\zeta(0, 1)}$. In fact, the minimal smooth birational models are exactly those represented by a singleton integral Type II point.

\begin{prop}\label{prop:corresp:selfint}
 Let $X$ be a model over $\bO$. Suppose that $X$ is smooth over its special fibre $X_0 = \sum_{i=1}^n \mb_i E_j$ where each $E_i$ is integral. Then at least one of the $\mb_i$ is equal to $1$ and either 
\begin{enumerate}
 \item $n = 1$, $E_1 \cong \P^1$, and $E_1^2 = p_a(E_1) = 0$, or
 \item $n > 1$, $E_i^2 < 0$ for every $i$, and there exists an $i$ such that $E_i^2 = -1$ and $p_a(E_i) = 0$.
\end{enumerate}
In particular, any smooth minimal model of $X$ has an integral rational special fibre.
\end{prop}

\begin{proof}
First, \autoref{prop:farey:negativity} says have that $E_i^2 \le 0$ with equality if and only if $n = 1$ and hence $X_0 = E_1 \cong \P^1$ is the whole fibre. Otherwise we may assume $E_i^2 < 0$ for every $i$.

 Let $K$ be a canonical divisor of $X$, then $K \cdot X_0 = -2$. Now 
 \[-2 = K \cdot X_0 = K \cdot \sum_{i=1}^n \mb_i E_i = \sum_{i=1}^n \mb_i K \cdot E_i\]
 Clearly, one of these terms must be negative, so fix an $i$ with $K \cdot E_i < 0$.
 By the adjunction formula, \autoref{thm:adjunction}, we get $K \cdot E_i + E_i^2 = 2 p_a(E_i) -2$; therefore, 
 \[-E_i^2 + 2 p_a(E_i) < 2.\]
 On the other hand, $ p_a(E_i) \in \N$ and $-E_i^2$ is a positive integer, so the only solution is $p_a(E_i) = 0$ and $E_i^2 = -1$.
\end{proof}

\subsection{Dual Graphs}

\begin{defn}
Let $\Gamma \subset \P^1_\an$ be a vertex set. If $\Gamma$ has the property that every $\Gamma$-domain is a disk or an annulus, we will say $\Gamma$ is \emph{SC}. Then we construct a finite graph $\Delta(\Gamma)$ called the \emph{dual graph} of $\Gamma$ as follows. Let the vertices of $\Delta(\Gamma)$ be the points of $\Gamma$. For every $\Gamma$-annulus $U$ with boundary points $\zeta, \xi \in \Gamma$, we let $\zeta\xi$ be an edge. 
\end{defn}

If $Y$ is an SC model then $\Gamma = \Gamma(Y)$ is SC and we may write $\Delta(\Gamma) = \Delta(Y)$. 
This dual graph embeds in $\P^1_\an$ naturally since $\Gamma \subset \P^1_\an$ and an edge is included as the interval $(\zeta, \xi)$. Hence  $\Delta(\Gamma)$ is always a tree because it is the subgraph of a tree.

By \autoref{prop:galois:berkred} each vertex in $\Gamma(Y)$ represents a unique irreducible component $E$ of $Y_\k$, and each edge $\zeta\xi$ with $\redct_Y(\zeta) = E, \redct_Y(\xi) = F$ represents the (transverse if SNC) intersection of $E$ with $F$. Hence $\Delta(Y)$ is the \emph{dual graph} of the irreducible components in the closed fibre $Y_\k$. We may instead consider the vertices of $\Delta(Y)$ to be irreducible components in $Y_0$ and think of edges corresponding to their intersections; we still require an SC model so that no three components meet at a single closed point. This was explained in the case of SNC models by Boucksom, Favre, and Jonsson \cite{BFJ}, expanding on the work of several others. They further show how the dual graph $\Delta(Y)$ of an SNC model $Y$ can be canonically embedded in $\P^1_\an$ with an isometry including the metric structures of the two. In \cite[\S 3]{BFJ}, this map is denoted by $\emb_Y : \Delta(Y) \to \P^1_\an$. In \cite{FJ04} it is shown that the injective limit of (images of) all such embeddings over all models $Y$ is $\P^1_\an$ itself. This essentially says that the collection of dual graphs sweep out the whole Berkovich line (or at least Type II/III points).

\begin{defn}
 Let $E$ be a vertex in $\Delta(\Gamma)$. Then we can define a \emph{direction} $\bvec u$ at $E$ to be a connected component of the subgraph $\Delta(\Gamma) - E$, equivalently it corresponds to the connected component of $\overline{X_0 \sm E}$. If $F \in \Gamma \sm \set E$ we let $\vec v_E(F)$ denote the unique direction at $E$ containing $F$. We call the set of directions at $E$ for all possible models containing $E$, the \emph{tangent space} at $E$ in $\P^1_\an$ (viewing each dual graph $\Delta(\Gamma)$ in $\P^1_\an$). 
\end{defn}

There is a natural correspondence between the tangent space $\Dir{\zeta}$ of directions at a Type II point $\zeta$ and the (potential) directions at $E = \redct(\zeta)$ within various dual graphs which contain $E$. More importantly, each direction at $E$ is associated to a closed point on $E$.

\begin{prop}[{\cite[Proposition 6.3]{FJ04}}]\label{prop:galois:directions}
 Let $X$ be a local model of $\P^1_\ell$ over $\bO$. Let $E \in \Gamma_G(X)$ and pick $p \in E$. Let $Y$ be the model obtained by blowing up $p$ and let $E_p \in \Gamma_G(Y)$ be the exceptional component. Denote by $\vec v_E(p) = \vec v_E(E_p)$ the tangent vector represented by $E_p$ at $E$.
Then the map $p \mapsto \vec v_E(p)$ is a bijection between the set of closed points in $E$ and the tangent space at $E$ in $\V$.%
\end{prop}

\begin{cor}\label{cor:corresp:domains}
 Let $X$ be a local model of $\P^1_\ell$ over $\bO$. Let $p \in X_0$ and $U$ be a $\Gamma(X)$-domain associated to $p$ ($\redct^+(U) = p$). Then the vertices $\zeta_1, \dots, \zeta_n$ bounding $U$ correspond to the divisors $E_1 = \redct^+(\zeta_1), \dots, E_n = \redct^+(\zeta_n) \in \Div_0(X)$ which contain $p$. In particular, if $p$ is contained in only one component $E$ of $X_0$, then $(\redct'_X)^{-1}(p)$ is a disjoint union of disks, each of which is a direction at a Type II point $\zeta \in (\redct'_X)^{-1}(E)$; further, these disks are Galois conjugates and lifts of $\vec v_E(p)$. Similarly, if $p = E_1 \cap E_2$ is satellite, then each connected component of $(\redct'_X)^{-1}(p)$ is an annulus.
\end{cor}

In the remainder of this article, we may abuse notation effectively identifying a valuation/norm $\zeta_E \in \V$ with its fibral divisor $E = \redct_X(v_E)$. We can also identify the directions at $\zeta_E$ (which contain other vertices of $\Gamma(X)$) with the directions at $E \in \Delta(X)$. Hence, we might refer to $\vec v_E(p)$ as a direction at a Type II valuation $\zeta$.

\begin{defn}\label{defn:corresp:freesat}
 Let $X$ be an model over $\bO$ with $X_0 = \sum_{j=1}^n \mb_j E_j$. We say a closed point $p \in X_0$ is \emph{free} iff $p \in E_i$ for some $i$ but $p \nin E_j\ \forall j \ne i$. We say $p$ is \emph{satellite} iff $p = E_i\cap E_j$ (but $p \nin E_m\ \forall m \ne i, j$). 
  If a curve $E$ is obtained (only) by blowing up a satellite point $p$ then say $E$ is \emph{satellite}, otherwise call $E$ \emph{free}.
\end{defn}

\begin{prop}\label{prop:corresp:smoothsnc}
 Let $X$ be a model over $\bO$. If $X$ is smooth then it is SNC and its dual graph is a tree. Furthermore, let $p \in X_0$ be a smooth closed point and $\pi : \hat X \to X$ be the point blowup at $p$, with exceptional curve $E$.
\begin{itemize}
 \item If $p$ is a free point on component $F$, then the dual graph $\Delta(\hat X)$ is obtained from $\Delta(X)$ by adding a leaf $E$ to $F \in \Delta(X)$.
  \item If $p = F_1 \cap F_2$ is a satellite point, then the dual graph $\Delta(\hat X)$ is $\Delta(X)$ with the edge $F_1F_2$ subdivided into $\hat F_1E$ and $E \hat F_2$ ($\hat F_1$ and $\hat F_2$ no longer intersect).
\end{itemize}
\end{prop}

\begin{proof}[Proof sketch]
Let $X$ be a smooth model. By proving the dual graph $\Delta(X)$ is a tree (without loops) it follows that $X$ has simple crossings.
 A smooth global model $X$ dominates a $\P^1$ bundle $S$ over the base curve (e.g.\ a Hirzebruch surface when the base is rational) by a sequence of point blowups--this is a result of the minimal model program for surfaces. We continue by induction on the number of blowups in this sequence i.e.\ the size of dual graph. Firstly, if $X=S$ then the special fibre is a single smooth irreducible $\P^1$; vacuously $\Gamma(S)$ is SNC. Now suppose that $\pi : X \to Y$ is a point blowup at $p \in Y$ of a smooth model $Y$, with exceptional curve $E \subset X$. Assume by induction that $Y$ is smooth and $\Gamma(Y)$ is SNC; in particular $p$ is a smooth point of $Y$. If $p$ is free, residing on the irreducible component $F \subset Y$, then $E$ is a smooth curve intersecting the proper transform, $\hat F$, of $F$ and has no other intersections with $X_0$. Hence, the dual graph $\Delta(X)$ is $\Delta(Y)$ with the addition of a leaf $E$ attached by an edge to $\hat F$. If $p$ is satellite between $F_1, F_2 \subset Y$, then $F_1$ and $F_2$ intersect normally at $p$ and nowhere else. In the blowup, the proper transforms of these $\hat F_1$ and $\hat F_2$ are separated (disjoint in $Y$). So similarly, $\Delta(X)$ is obtained from $\Delta(Y)$ by deleting the edge $F_1F_2$, inserting two edges $\hat F_1E$ and $E \hat F_2$. One can see that in either case $\Delta(X)$ is also a tree and that new intersections are normal by nature of smooth point blowups.
\end{proof}

Any SC model is dominated by a smooth (hence SNC) model; so in the reverse direction, by contracting irreducible components in an SNC model which only intersect two others we can reach all SC models.

Relaxing the SC condition is reasonable; this allows for arbitrary vertex sets related to arbitrary models. The dual graph of intersections will not be a tree anymore, and cannot be injectively mapped into $\P^1_\an(\k((x)))$. However, this is only because each $\Gamma(X)$-domain with $n > 1$ boundary points--corresponding to a closed point which is the intersection of $n$ fibral components--induces a complete subgraph of order $n$ in $\Delta(X)$; indeed, contracting all these subgraphs would result in a tree.

In the most general circumstances it is better to rely on the perspective of vertex sets and their complimentary domains, as in \autoref{cor:galois:pointcorrespondence}. One should not try to find a correspondence edges and intersections, but focus on the correspondence between $\Gamma$-domains with $n$ boundary points $\zeta_1, \dots, \zeta_n$ and closed points $p \in Y_\k$ which are the intersection of $n$ irreducible components $E_1, \dots, E_n$, with $\redct_Y(\zeta_j) = E_j$.

\begin{lem}\label{lem:corresp:simplesection}
 Let $X$ be a local model and $S \subset X$ a section defined by $s(x) = (x, \gamma(x))$, where $\gamma \in \k((x))$. The intersection with the special fibre $X_0 \cap S$ is a single point $p$. If $p$ is a smooth point 
 then $p$ lies in a single component $E \subset X_0$ and ${\redct'_X}^{-1}(p) = \vec v_E(p) = \vec v(\gamma)$ is the $\Gamma(X)$-domain representing $p$. 
 Furthermore, if $\pi : \hat X \to X$ is a blowup of this $p$, then its exceptional curve $F$ is the component intersecting (the proper transform of) $S$ on $\hat X$ at a free point.
\end{lem}

\begin{proof}
 It is clear that the intersection is a single point $p$. Suppose $p$ is smooth; since the problem is local, we may desingularise outside of an open neighbourhood of $p$ to ensure $X$ is smooth for the following argument. By \autoref{lem:jhssec}, $S$ is a well defined section on a smooth model, so it must intersect the special fibre (transversely) at a single reduced component of $X_0$, say $E$. Whence $p$ is a free point. Now blowup $p$ to obtain $\pi : \hat X \to X$ with exceptional curve $F$; let $\hat E$ be the proper transform of $E$ and $\hat p$ be the intersection of $S$ with $X_0$.  As stated in \autoref{prop:galois:reductionstructure}, the direction $\redct_{X}^{-1}(p)$ corresponding to $D = \vec v(p) = \vec v(\gamma)$ is a $\Gamma(X)$-disk bounded by $E$. 
 By definition, $\redct_{\hat X}(F)$ lies in the corresponding direction $\vec v(\gamma)$ at $\redct_{\hat X}(E)$. Now, the disk $D$ splits into an annulus $\redct^{-1}(\hat E \cap F)$ (for the new satellite point) and many disks reducing to the free points on $F$. By the above applied to $\hat X$, the point $\hat p$ is free, so $\vec v(\hat p)$ lies in a $\Gamma(\hat X)$-disk, not an annulus. Therefore $\redct(\gamma)$ is a free point of $F$, not $\hat E \cap F$.
\end{proof}

\begin{defn}
 Given $X$ a local model of $\P^1_\ell$ over $\bO$, let $S_\infty$ denote the section defined by $y = \infty$ and $p_{X, \infty} = S_\infty \cap X_0$ denote its unique point of intersection with $X_0$.
\end{defn}

Note that $p_{X, \infty} = \redct_X(\infty)$ is a closed point regardless of the model we pick by \autoref{cor:galois:pointcorrespondence}, because $\infty$ is Type I. 
The embedding $\emb_X : \Delta(X) \to \P^1_\an(\k((x)))$ induces a partial order.
If $p_{X, \infty}$ is a smooth point on $Y$, then it lies on a single component $E_\infty$, and this is the maximum in $\Delta(Y)$, just as $\infty$ is the maximum in $\P^1_\an$. It follows that $F_1 \prec F_2$ if and only if $F_2 \in (F_1, E_\infty]$ in $\Delta(Y)$ as a simplicial tree. Another way to define this ordering is $F_1 \prec F_2$ iff $\vec v(E_\infty) \ne \vec v(F_1)$. In the alternative dual graph $\Delta^+(X)$ in the proof above, $S_\infty$ is the maximum, a leaf attached to $E_\infty$.

\begin{defn}
 Let $X$ be a local model of $\P^1_\ell$ over $\bO$. Define a partial order on its dual graph $\Delta(X)$ as the smallest such that for every $\zeta, \xi \in \P^1_\an$, $\zeta \preceq \xi \implies \redct'_X(\zeta) \preceq \redct'_X(\xi)$. 
\end{defn}

\begin{prop}\label{prop:corresp:ordering}
 Let $X$ be a local model of $\P^1_\ell$ over $\bO$. If $p_{X, \infty} = \redct'_X(\infty)$ is smooth then the unique component $E_{X, \infty}$ containing $p_{X, \infty}$ is the maximum point of $\Delta(X)$. If $E \cap F$ is a smooth point then either $E \prec F$ or $E \succ F$.
 Let $p \in X_0$ be a non-singular closed point of $X$ and blowup $p$ to obtain the model $Y$ with exceptional curve $E$. Let $\zeta = \redct_Y^{-1}(E)$.
 \begin{enumerate}
 \item If $p = p_\infty \in E_1$ is free and $\xi = \redct_Y^{-1}(E_1)$, then $E \succ E_1$, i.e.\ $\zeta \succ \xi$
 \item If $p \ne p_\infty$ is free in component $E_1$ and $\xi = \redct_Y^{-1}(E_1)$, then $E \prec E_1$, i.e.\ $\zeta \prec \xi$.
 \item Otherwise, suppose $p = E_1 \cap E_2$ is a satellite point, $\xi_1 = \redct_Y^{-1}(E_1)$, and $\xi_2 = \redct_Y^{-1}(E_2)$, and $\xi_1 \prec \xi_2$. Then $\zeta \in [\xi_1, \xi_2]$, and $\xi_1 \prec \zeta \prec \xi_2$.
\end{enumerate}
\end{prop}

\begin{proof}
 We follow the induction on $\Delta(X)$ in the above proofs.
 The base case is vacuous. If we perform a free blowup $p \ne p_{X, \infty}$ on $F$, then the new exceptional curve lies in a direction $\vec v_F(p) \ne \vec v_F(E_\infty)$; it follows immediately that $E \prec F$ because $\zeta \prec \xi$ for any $\zeta \in \bvec v$ at $\xi$ so long as $\bvec v \ne \vec v(\infty)$. In the case of a satellite blowup at $p = F_1 \cap F_2$ then $p \ne p_{X, \infty}$ by \autoref{lem:corresp:simplesection} and there is an interval $[F_1, F_2] \subset \Delta(X)$ with say $F_1 \prec F_2$ by inductive hypothesis. Subdividing this with $E$ yields $F_1 \prec E \prec F_2$ because $[\xi_1, \xi_2] \ni \zeta$ is a totally ordered subinterval of $\P^1_\an$. Finally, suppose that $\pi : \hat X \to X$ is the point blowup of $p_{X, \infty}$, which must be a free point on the component $E_{X, \infty} = \redct_X(\zeta)$ by \autoref{lem:corresp:simplesection}; this yields exceptional curve $E_{\hat X, \infty} = \redct_{\hat X}(\xi)$. Abusing notation, write $E_{X, \infty} = \redct_{\hat X}(\zeta)$. Now the new point at infinity $p_{\hat X, \infty}$ is free on $E_{\hat X, \infty}$ and distinct from the point $E_{X, \infty} \cap E_{\hat X, \infty}$, thus $\vec v_{E_{\hat X, \infty}}(E_{X, \infty}) \ne \vec v_{E_{\hat X, \infty}}(p_{\hat X, \infty})$. Therefore $\xi \in (\zeta, \infty)$, i.e.\ $\zeta \prec \xi \prec \infty$, which implies moreover that $E_{X, \infty} \prec E_{\hat X, \infty}$ as required.
\end{proof}

\begin{rmk}
  There is a correspondence between vertex sets in $\P^1_\an(\k((x))$ and Galois invariant vertex sets in $\P^1_\an(\K)$, because $\redct'_Y = \mathfrak p \circ \redct_Y$. One can easily check with this proposition that a model $Y$ is SC if and only if the corresponding vertex set $\Gamma(Y)$ has only disks and annuli in $\mathcal S(\Gamma(Y))$. Whilst remembering the importance of Galois invariance, we will continue without constantly specifying the base field of $\P^1_\an$.
\end{rmk}

\subsection{Correspondence of Metrics and Invariants}\label{sec:metrics}

Now we explore correspondence between geometric information of an irreducible curve $E \subseteq X_0$ to the multiplicity and parametrisation of the corresponding point $\redct^{-1}(E) = \zeta \in \P^1_\an$. 
For the rest of the chapter, we will typically forget the type of model or identity of the point $b$ on our base curve $B$ and refer only to $x=0$, which refers to the maximal ideal in our completed local ring $\k[[x]]$. 

A birationally ruled surface $X$ is birationally equivalent to $B \times \P^1$, so the closed fibre $X_0$ contains finitely many rational curves $E_1, \dots, E_n$ and they are represented by the generic points $\eta_1, \dots, \eta_n \in X_0$. Moreover the fibre is possibly not reduced meaning the principal divisor $(x)$ can be written as $\sum_{j=1}^n \mb_j E_j$ where $\mb_j = \ord_{E_j}(x)$ are positive integers, and $\ord_E$ is the order of vanishing valuation at $E$. This valuation is one on the DVR $\bO_{X, E}$ such that in local coordinates $z, w$ we have $E = (z)$,\ $\ord_E(z) = 1$, and $\ord_E(w) = 0$. In general we can define $\ord_E(x) = \mb(E)$. Then $\mb(E_j) = \mb_j$ and
\begin{equation}\label{eq:normval}
\frac{\ord_E(\cdot)}{\mb(E)} = \log_{\abs x}(\norm[E]\cdot)
\end{equation}
defines a unique Type II point in $\P^1_\an(\k((x))$. The integer $\mb(E)$ forms one half of the \emph{Farey parametrisation} for $E \in \Gamma$. Following \cite[\S 6.3]{FJ04}, we define another, $\ma(E)$. This requires an initial choice of minimal smooth model $h : \smodel \to B$ with $\Gamma(\smodel) = \set{\zeta(0, 1)}$. We say that $X$ is \emph{minimal} if it is smooth and $n=1$ so that $\Gamma(X) = \set \zeta$. One can show  equivalent to $X$ being locally trivial near $0$, i.e. $X_\bO \cong \Spec(\bO) \times_\k \P^1_\k$. Let $U \subset B$ be an affine neighbourhood of $0 \in B$, and assume that the variable $x$ uniformises both the completed and original local ring at $0$. Using Tsen's theorem or otherwise, we can trivialise $h^{-1}(U) \cong U \times \P^1$ with local coordinates $(x, y)$ on $U \times \A^1$. There is a section $y=\infty$ which extends to a divisor $S_\infty \subset \smodel$. These coordinates determine by extension the reduction map for any other model via birational transformation, for instance the Gauss point is chosen by $\redct'_\smodel(\zeta(0, 1)) = E_0$ and this constitutes the whole special fibre $\smodel_0$ on $\smodel$. Then on any model $X$ there exists a rational $2$-form $\frac{\dd x \wedge \dd y}{x}$, referring to the coordinates on $\smodel$.
\begin{defn}
Define $\mb(E) = \ord_E(x)$ and
 \[\ma(E) = \ord_E\left(\frac{\dd x \wedge \dd y}{x}\right) + 1 = \ord_E\left(\dd x \wedge \dd y\right) - \mb(E) + 1.\]
 We will call the pair $(\ma(E), \mb(E))$, the \emph{Farey parameter} of $E$. We will see soon that $\ma(E)$ is an integer, and henceforth we might alternatively refer to the \emph{Farey fraction} $\ma(E)/\mb(E)$.
\end{defn}

Let $X$ be any global model of $\smodel$ with irreducible components $E_0, \dots, E_r$ in the special fibre, and $S_\infty$ be the proper transform of the one defined above on $\smodel$. Then
\[\dv(\dd x \wedge \dd y) = F' - 2S_\infty + \sum_E (\ma(E) +\mb(E) - 1) E,\] where $F'$ is a fibral divisor with $0$ support in $X_0$.

\begin{rmk}
 In \cite{FJ04} $\ma$ is denoted $a$ and $\frac\ma\mb$ is denoted $A$ (at least in the relative valuative tree), however in \cite{berkandapps} $A_E$ is used for our $\ma(E)$ and $A_\pi(E)$ for $\frac{\ma(E)}{\mb(E)}$. They reserve fraktur font for ideals. The parameter $\ma$ is often referred to as the \emph{log discrepancy} and sometimes as \emph{thinness}.

\end{rmk}

Any irreducible curve $E \subseteq X_0$ is transversely swept out by curves in $X$ -- these are called curvettes in \cite{FJ04}. In particular, for any point $P \in E$, we can find a section $C \subset X$ such that $C \cap X_0 = P$ and $C \cdot E = 1$. In the completion, $\hat C$ is given locally by some rational function $P \in \k((x))(y)$. In fact, given some local coordinates $z, w \in \k((x))(y)$ at $E$ such that $E = \set{z = 0}$, we have that a family of transverse curves is defined by $\set[\dv(w = c)]{c \in \k}$. 

Supposing $g$ vanished on $E$ to order $n$, we could write $g(z, w) = z^n u(z, w)$ in local coordinates, where $z \nmid u$. For each $c$ where $u(0, c) \ne 0$, the intersection multiplicity of $\dv(w - c)$ and $\dv(g)$ on $E$ is $n$. Therefore this intersection multiplicity with $\set{w = c}$ defines a valuation that is non-trivial since $E \subset \dv(x)$. From the geometric and algebraic structure of $\P^1_\an$ at a Type II point $\zeta$ which reduces to $E$, we know that these curvettes described by $w=c$ correspond or factor into germs of the form $y = \gamma(x)$ where $\gamma \prec \zeta$ in $\P^1_\an$.

\begin{lem}\label{lem:galois:coordinates}
 Let $X$ be a model and suppose that $E \subset X_0$ is a smooth prime divisor on $X$. Then there are coordinates $(z, w)$ containing a neighbourhood of $E \sm \set P \cong \A^1(\k)$, such that
 $z, w \in \k(x, y)$ and $x, y \in \k(z, w)$ and $z$ is the uniformiser of the discretely valued field $(\k(z, w), \ord_E)$, with $\ord_E(z) = 1$.
\end{lem}

We will call the pair $(z, w)$ \emph{local (rational) coordinates} for $E$. 
We shall refer to $w$ as the \emph{abscissa} of $E$ (the coordinate transverse to the line), and to $z$ as the \emph{ordinate} (a measure of parallel distance). %

\begin{prop}[{\cite[Theorem 7.2]{spiv}}, extended]\label{prop:galois:intersection}%
 Let $E \subseteq X_b$ be a curve, and $z, w \in \k(x, y)$ be local rational coordinates on $E$. Suppose $f \in \k((x))(y)$ can be written as \[f(z, w) = a_n(w)z^n + a_{n+1}(w)z^{n+1} + \cdots,\] with $a_n(w) \ne 0$, then $\ord_E(f) = n =\ord_z(f(z, c)) = n$ for all but finitely many $c \in \k$.
Furthermore \[\ord_E(f) = \min_{f(z) \in \k[z]} \ord_{z}(f(z, f(z)) = \min_{\substack{C\cdot E \ge 1\\P \in C \cap E}} i(C, V(f); P) = \min_{C\cdot E = 1} C \cdot V(f).\]
\end{prop}

Translating \autoref{prop:galois:intersection} from valuative language into a statement about seminorms we obtain the following corollary. Note that every $y = \gamma \in \K$ solving $w(x, y) = c$ will intersect the lift of $E$, after a base change. It turns out that for all but one value of $c \in \k$ we have that $w(x, \gamma(x)) = c \implies \gamma \in \CD(a, r)$ where $\zeta_E = \zeta(a, r)$. So naturally, $\gamma \prec \zeta_E$ and $\norm[\zeta_E]f \ge \norm[\gamma]f$. If $\gamma$ does not lie in the direction of a root $y= \gamma'(x)$ of $f(x, y)$ then equality holds, although this detail is easier to see after the lemma (\autoref{lem:galois:coordinates}).%

\begin{cor}
 Let $E \subseteq X_0$ be a curve, let $\zeta_E \in (\redct'_X)^{-1}(E)$ and let $z, w \in \k(x, y)$ be local rational coordinates such that $\ord_E(z) = 1$. Then $\norm[\zeta_E]{z} = \abs x^{1/\mb(E)}$, and
 \[\norm[\zeta_E]f = \max_{w(x, \gamma(x)) = c} \norm[\gamma]f.\] %
 For all but finitely $c \in \k$, $\norm[\zeta_E]f = \norm[\gamma]f$.
\end{cor}

Following \cite[\S 6]{FJ04} we show that $\mb(E)$ is in fact the generic multiplicity of the associated Type II point $\zeta$. Hence we may write $\mg(\zeta) = \mg(E) = \mb(E)$.

\begin{thm}\label{thm:galois:parametercorrespondence}\label{thm:galois:diameter}%
 Let $X$ be a model, suppose that $E \subset X_0$ is an irreducible curve and $\redct(\zeta) = E$. %
Then $\mb(E) = \mg(\zeta)$ and $\displaystyle\frac{\ma(E)}{\mb(E)} = \log_{\abs x}(\diam(\zeta))$; hence $\displaystyle \ma(E) = \max_{\gamma \in K}\ord_E(y-\gamma)$ is an integer.
 Furthermore, one can write $\zeta = \zeta\shiftbracket{2pt}{14pt}{\gamma_0, {\abs x}^{\tfrac{\ma(E)}{\mb(E)}}}$ where $\mm(\gamma_0) = \mm(\zeta)\mid\mb(E)$.
\end{thm}

\begin{rmk}
 The above maximum is attained for any Puiseux series $\gamma$ in the disk $D = \CD(\gamma_0, {\abs x}^{\frac{\ma(E)}{\mb(E)}})$ and $\gamma_0$ can be chosen to be the unique order $\frac{\ma(E)}{\mb(E)}$ truncation of any Puiseux series in $D$.
\end{rmk}

\begin{proof}
 Firstly, by \autoref{lem:galois:coordinates}, we have (rational) coordinates $z, w$ at $E$ and rational functions $P, Q \in \k(z, w)$ such that $x = P(z, w),\ y = Q(z, w)$, and $\ord_E(z) = 1$. By definition, $\mb(E) = b$ is the integer such that $P(z, w) = z^b \cdot u(z, w)$ where $\norm[\zeta_E]{u(z, w)} = 1$, i.e. $u$ is a unit with respect to $\zeta_E$ and $z$. Whence $\norm[\zeta_E]{x} = \norm[\zeta_E]{z^b}$, so $\norm[\zeta_E]{z} = \abs x^{1/b}$.
Let $g = \mg(E)$; there exists $\gamma \in \k(x^{1/g})$ and $a \in \Z$ such that $\zeta_E$ is the Galois equivalence class of $\zeta\left(\gamma, \abs x^{\frac a g}\right)$, which is a norm on $\k(x, y)$ taking values in $\set[\abs x^{\frac {mg + na}g}]{m, n \in \Z}$. Therefore $\mg(E) = N\mb(E)$ for some positive integer $N$.

 For clarity write $w_b = w^{1/b}$ and $x_b = x^{1/b}$. By Puiseux's Theorem (see \autoref{sec:twovarpuiseux}), there exists a solution $z = \alpha(w_b, x_b) \in \k((w_b))((x_b))$ to the equation $x = z^b \cdot u(z, w)$. So $\mm(\alpha) \le b$, and since $\abs{\alpha} = \abs x^{1/b}$ it necessarily has multiplicity $\mm(\alpha) = b$. Whence we obtain the corresponding solution $y = \beta(w_b, x_b) = Q(\alpha(w_b, x_b), w_b^b)$, which has multiplicity $\mm(\beta) \le \mm(\alpha) = b$. More simply, this can be carried out for fixed general value of $w=c$ with root $y = \beta_c(x_b)$. Note that the direction $\vec v(\beta_c)$ is a direction at $\zeta_E$ which is associated to the point with coordinates $(z, w-c)$ (see \autoref{prop:galois:directions}), and we may assume it is \emph{generic} in the sense of \autoref{defn:specgendirn}. Now \autoref{prop:multsubtreevalency} says that for any generic direction $\bvec v \in \Dir{\zeta_E}$ we have $\mg(\zeta) = \mm(\bvec v) \le \mm(\gamma)\ \forall \gamma \in \bvec v$; hence $\mg(\zeta) \le \mm(\beta_c) \le b$. Therefore $\mg(\zeta) = \mb(E) = \mb(\beta)$.

Now we consider $\ma(E)$, first by rewriting $\dd x \wedge \dd y$ with respect to local coordinates $(z, w)$. By definition of $\ma$, we necessarily have
\[\dd x \wedge \dd y = z^{\ma(E) + \mb(E) -1} \tilde u(z, w) \dd z \wedge \dd w,\] where $\ord_E(\tilde u(z, w)) = 1$. Consider $y = \beta(w_b, x_b) \in \k((w_b))((x_b))$ as above. %
 Now we can write $\zeta = \zeta\left(\gamma_0, \abs x^{\frac ab}\right)$ and 
 \[y = \beta(w_b, x_b) = \gamma_0(x_b) + \delta(w_b) x_b^a + \eps(w_b, x_b)x_b^{a+1}\]
  where $\delta(w_b)$ depends only on our choice of root for $w = c$, and $\abs\eps \le 1$.
  Whence $\sup_{\gamma \in K}\ord_E(y-\gamma) = a$ and this maximum is attained at $\gamma = \gamma_0$. Note for the final result that $a = \mb(E)\log_{\abs x}(\diam(\zeta))$ by \autoref{eq:normval}.  We aim to show that $a = \ma(E)$. 
  
 Observe that $\dd x = \dd{(x_b^b)} = bx_b^{b-1}\dd x_b$ so $\dd x \wedge \dd y = bx_b^{b-1}\dd x_b \wedge \dd y$. Using $y = \beta(w_b, x_b)$, we have 
 \[\dd y = \pd \beta {x_b} \dd x_b + \pd \beta {w_b} \dd w_b = \pd \beta {x_b} \dd x_b + \pd \beta {w_b} \dd w_b = \pd \beta {x_b} \dd x_b + \left(\td \delta {w_b} + x_b \pd\eps {w_b}\right)x_b^a\dd w_b,\]
 where $\abs{\pd \eps {w_b}} \le 1$. Since $\delta$ does depend on our choice of $w$, and hence $w_b$, it follows that $\td \delta {w_b} + x_b \pd\eps {w_b} = u'(w_b, x_b)$ is a unit with respect to $x_b$.
 Thus $\dd x_b \wedge \dd y = u'(w_b, x_b)x_b^a\dd x_b \wedge \dd w_b$ and so
 $\dd x \wedge \dd y = bu'(w_b, x_b) x_b^{a + b-1}\dd x_b \wedge \dd w_b$. 
 Since $x = z^b u(z, w)$, $x_b = z u(z, w_b^b)^{\frac 1b}$, and we have
  \[\dd{x_b} = \left(u(z, w)^{\frac 1b} + \frac zb u^{\frac {1-b}b} \pd uz \right)\dd z + \frac zb u^{\frac {1-b}b}\pd u{w_b}\dd w_b.\] 
  It follows that
 \[\dd x_b \wedge \dd w_b = u''(z, w_b)\dd z \wedge \dd w_b,\]
 where $u''(z, w_b) = u(z, w_b^b)^{\frac 1b} + \frac zb u^{\frac {1-b}b} \pd uz$ is a unit with respect to $z$, and so
  \[\dd x \wedge \dd y = bu'(w_b, x_b) x_b^{a + b-1} \dd x_b \wedge \dd w_b = bu'(w_b, x_b)u''(z, w_b) x_b^{a + b-1}\dd z \wedge \dd w_b.\] %
  Finally, write $\dd w = b w_b^{b-1}\dd {w_b}$ and let $\tilde u(z, w_b) = w_b^{1-b}u(z, w_b^b)^{\frac{a+b-1}b}u'\left(w_b, z u(z, w_b^b)^{\frac 1b}\right)u''(z, w_b)$, which is a unit with respect to $z$; now one can check that $\dd x \wedge \dd y = \tilde u z^{a+b-1}\dd z \wedge \dd w$.
  Therefore $a = \ma(E)$ because $\ma(E)+\mb(E)-1 = \ord_E(\dd x \wedge \dd y) = a+\mb(E)-1$.
\end{proof}

 Define the metric on $\Delta(\Gamma)$ to be the one induced by the hyperbolic metric $d_\bH$ through embedding. By \autoref{thm:galois:diameter}, this is the same metric induced by declaring that whenever $E \succeq F$
 \[d(E, F) = \frac{\ma(E)}{\mb(E)} - \frac{\ma(F)}{\mb(F)}.\]

\begin{prop}
 Let $X$ be a local model. If $E, F \in \Delta(X)$ and $E \prec F$, then $\frac{\ma(E)}{\mb(E)} > \frac{\ma(F)}{\mb(F)}$.
\end{prop}

\begin{proof}
  If $\zeta \prec \xi$ are Type II points then we can write $\zeta = \zeta(\gamma, r)$ and $\xi = \zeta(\gamma, R)$, where $r < R$. Recall that $\frac{\ma(\zeta)}{\mb(\zeta)} = \log_{\abs x}(\diam(\zeta))$. Therefore $\frac{\ma(E)}{\mb(E)} = \log_{\abs x}(r)$ and $\frac{\ma(F)}{\mb(F)} = \log_{\abs x}(R)$. The inequality follows because $\abs x <1$ and so $\log_{\abs x}$ is decreasing.
\end{proof}

We consider $\P^1_\an$ as the \emph{universal dual graph}, meaning that it contains the dual graph for all models, and we shall move toward only doing calculations in $\P^1_\an$, rather than always concerning ourselves with choice of model. In this larger setting, one ought to think of Type I points as \emph{sections} of the surface defined by Puiseux series. Recall that \autoref{prop:galois:vertexextension} (Baker-Payne-Rabinoff, Favre-Jonsson) says that any finite collection $\Gamma$ of Type II points yields a model $X$ with $\Gamma(X) = \Gamma$. Next, \autoref{lem:corresp:simplesection} will generalise to any multisection below in \autoref{thm:farey:puiseuxsection}. 
This theorem lays the foundation in our study of smoothness.%

Let $h : X \to \Spec \bO$ be a local model with special fibre $X_0$ over the maximal ideal $\mathcal M \triangleleft \bO$, and $p \in X_0$ a closed point with maximal ideal $\mathcal M_{X, p}$. Define the \emph{multiplicity} of $p$, $\mm(p)$ as the greatest power of $\mathcal M_{X_0, p}$, containing $h^*(\mathcal M)$. More simply put on a model over $\k[[x]]$, $\mm(p)$ is the multiplicity of the function $x$ on the surface at $p$.

\begin{thm}\label{thm:farey:puiseuxsection}
 Let $X$ be a local model and $S \subset X$ a multi-section parametrised by $s(t) = (t^m, Q(t))$, where $Q \in \k((t))$, or equivalently $y = \gamma(x) = Q(x^{1/m})$. The intersection $X_0 \cap S$ with the special fibre is a closed point $p$ and $\gamma$ lies in a $\Gamma(X)$-domain $U = {\redct_X}^{-1}(p)$ associated to $p$.
 If $p$ is a smooth point then one of the following holds.
\begin{enumerate}
 \item $p$ lies in a one component $E \subset X_0$, $U = \vec v_E(p) = \vec v(\gamma)$ is a $\Gamma(X)$-disk, and $\mm(p) = \mb(E) = \mm(\vec v_E(p)) = \mg(\vec v_E(p)) \le \mm(\gamma)$. Hence, $\vec v_E(p)$ is a generic direction.
 \item $p$ is the intersection of two components $E, F \subset X_0$, $U = \vec v_E(p) \cap \vec v_F(p)$ is a $\Gamma(X)$-annulus, and $\mm(p) = \mb(E) + \mb(F) = \mg(U) \le \mm(\gamma)$. 
\end{enumerate}
 Furthermore, if $\pi : \hat X \to X$ is a blowup of this $p$ and $\mm(\gamma) = \mm(p)$, then its exceptional curve $F$ is the component intersecting (the proper transform of) $S$ on $\hat X$ at a free point.
\end{thm}

In fact, in \autoref{cor:smooth:multdesc} we see that $\mm(\gamma) \in \N^+\langle\mm(p)\rangle$ for a smooth free $p$ and $\mm(\gamma) \in \N^+\langle\mb(E) + \mb(F)\rangle$ for smooth satellite $p$. 

\begin{proof}
 Let $h : X \to \Spec \bO$ be a local model, with $(x)$ the maximal ideal for $\bO$ as usual, and $p \in X_0$ a closed point with maximal ideal $\mathcal M_{X, p}$. Take a base extension $\iota : \bO \hookrightarrow R = \bO[t]/(t^m - x)$. Then $R$ is also a local ring with maximal ideal $(t)$ and $S$ is the map $\Spec R \to X$ given by $x = t^m, y = Q(t)$. Note that a parametrised multisection of the form $s(t) = (t^m, Q(t))$ splits into $m$ sections of the form $(t, Q(\rho t))$ i.e.\ $y = \gamma_i(x) = Q(\rho^i x^{1/m})$ for $i = 0, \dots, m-1$, where $\rho$ is any $m$-th root of unity.
 
 The first part of the result is restating results of the previous section such as reduction; see for instance \autoref{cor:galois:pointcorrespondence}. Now assume $p$ is a smooth closed point, and since the problem is local, one may assume $X$ is smooth. By \autoref{prop:corresp:smoothsnc} $p$ lies on either one or two components.
 
If $p$ is smooth free point on the special fibre of $X$ lying (uniquely) on $E \subset X_0$, then $U$ is a $\Gamma(X)$ disk bounded by $\alpha = \emb(E)$. Otherwise, $p$ is a smooth satellite point at the intersection of components $E = \redct(\alpha), F = \redct(\beta) \subset X_0$; then $U$ is an annulus bounded by $\alpha$ and $\beta$, or equivalently, $U = \vec v_E(p) \cap \vec v_F(p)$. See \autoref{prop:galois:reductionstructure}, \autoref{cor:corresp:domains}. By \autoref{lem:galois:coordinates}, we may assume $p$ has local coordinates $(z, w)$ such that $\mathcal M_{X, p} = \langle z, w\rangle$, $\ord_E(z) = 1$ and $x = x(z, w) = z^{\mb(E)}u(z, w)$, where $\ord_E(u) = 0$. In the satellite picture, we can also arrange for $\ord_F(w) = 1$, $\ord_E(w) = \ord_F(z) = 0$ and $x = x(z, w) = z^{\mb(E)}w^{\mb(F)}u(z, w)$ with $u$ a unit. In the free case, $x \in \mathcal M_{X, p}^{\mb(E)} \sm \mathcal M_{X, p}^{\mb(E)+1}$; whence $\mm(p) = \mb(E)$. In the satellite case, $x \in \mathcal M_{X, p}^{\mb(E) + \mb(F)} \sm \mathcal M_{X, p}^{\mb(E) + \mb(F)+1}$; whence $\mm(p) = \mb(E) + \mb(F)$.
 
 Consider the multisection $s(t)$ as above, one has the commuting diagram.
 \[
\begin{tikzcd}
 & X \arrow{d}{h} \\
 \Spec R \arrow[swap]{r}{\iota} \arrow{ru}{s} & \Spec \bO
\end{tikzcd}
\]
Note that $\iota^*(x) = (t^m)$ and $h^*(x) \subseteq \mathcal M_{X, p}^{\mm(p)}$, and observe the chain of ideals
\[(t)^m = \iota^*(x) = (h\circ s)^*(x) = s^*h^*(x) \subseteq s^*(\mathcal M_{X, p}^{\mm(p)}) = (s^*\mathcal M_{X, p})^{\mm(p)} = (t)^{\mm(p)}.\]
Therefore $m \ge \mm(p)$. Moreover, since this holds for any such $\gamma \in U$, we have \[\mg(U) = \min_{\gamma \in U} \mg(\gamma) \ge \mm(p).\] 
Here, we also use that the minimum is attained by Type I points; see \autoref{prop:galois:multiplicitysubset}. Recall from \autoref{cor:multsubtreevalency} that $\mm(\vec v(p)) \in \{1, \mm(\alpha), \mg(\alpha)\}$ where $\mm(\alpha) \mid \mg(\alpha) = \mb(E)$. Hence, when $p$ is free, \[\mm(p) = \mb(E) \ge \mm(\vec v(p)) =\mg(\vec v(p)) \ge \mm(p)\] so equality holds throughout. 
When $p$ is satellite we have \[\mg(U) \ge \mm(p) = \mb(E) + \mb(F).\]
To obtain equality, we can exhibit a $\zeta \in U$ with $\mb(\zeta) = \mb(E) + \mb(F)$ where the minimum $\mm(U) = \min_{\zeta \in U} \mb(\zeta)$ is attained. For this, we can take the divisorial valuation $\nu$ for which $\nu(z)=\nu(w) = 1$, so that $\nu(x) = \nu(z^{\mb(E)}w^{\mb(F)}u(z, w)) = \mb(E) + \mb(F) + 0$. Some details of this are deferred to the next theorem because $\nu$ is precisely the order of vanishing along the exceptional divisor obtained by blowing up at $p$. 
\end{proof}

\newpage


\section{Farey Parameters and Smoothness}\label{sec:smooth}

So far we have setup the Berkovich projective line over the Puiseux series as a universal dual graph and given it a parametrisation with multiplicity $\mm$, generic multiplicity $\mg=\mb$ and the log discrepancy $\ma$. Now we look at how these parameters behave under smooth point blowup and study the patterns of Farey parameters that can appear with a smooth surface. More specifically, the primary goal of this section is to show that a vertex set $\Gamma$ represents the divisors of a smooth model $X$ if and only if it has a generalised Farey sequence property. This is \autoref{thm:smooth}. 
As stated in the introduction, much of this section is inspired by and continues a discussion from \cite[\S 6]{FJ04}.

\subsection{Parameters and Blowups}\label{sec:paramblowup}
The next few results are a reframing some of Favre and Jonsson such as \cite[Proposition 6.36, Proposition 6.37]{FJ04}. Their approach differs from ours in that they take these properties somewhat axiomatically on the dual graph and afterwards prove an isomorphism with the valuative tree.

\begin{thm}\label{thm:farey:addition}
 Let $X$ be a global model over $0$ and let 
  $S_\infty$ be the section at infinity as defined previously and $p_\infty$ be the unique point $S_\infty \cap X_0$. Let $p \in X_0$ be a non-singular closed point of $X$ and blowup $p$ to obtain the model $Y$ with exceptional curve $E$.%
\begin{enumerate}
 \item If $p = p_\infty \in E_1$ is free, then $\ma(E) = \ma(E_1) - 1$ and $\mb(E) = \mb(E_1) = 1$.%
 \item If $p \ne p_\infty$ is free in component $E_1$, then $E$ has Farey parameter \[(\ma(E), \mb(E)) = (\ma(E_1) + 1, \mb(E_1)).\]%
 \item If $p = E_1 \cap E_2$ is a satellite point, then $E$ has Farey parameter \[(\ma(E), \mb(E)) = (\ma(E_1) + \ma(E_2), \mb(E_1) + \mb(E_2)).\]%
\end{enumerate}
\end{thm}

\begin{proof}
Recall from \autoref{sec:metrics} that we fix a smooth minimal model $\smodel$ with local coordinates $(x, y)$ near the special fibre, further there is a meromorphic $2$-form $\alpha = \frac 1x \dd x \wedge \dd y$. This $\alpha$ yields a choice of canonical divisor $K_X$ in every other birational model $X$. As discussed above,  $K_X = -2S_\infty + \sum_E (\ma(E) -1) E$ modulo the fibral divisors disjoint from $X_0$, where the sum ranges over irreducible components of $X_0$; so $\ma(E) = \ord_E(\alpha) + 1$.
 Let $(z_1, z_2)$ be local coordinates at a smooth point $p \in X$ on the surface corresponding to $z_1 = z_2 = 0$. Further let $E_1 = \{z_2 = 0\}$ and $E_2= \{z_1 = 0\}$ be prime divisors--these have abscissae $z_1$ and $z_2$ respectively. Suppose that $\ord_{E_1}(x) = b_1$, $\ord_{E_2}(x) = b_2$, with $\ord_{E_1}(\alpha) = n_1$ and $\ord_{E_2}(\alpha) = n_2$. One can write $x = z_1^{b_2}z_2^{b_1}u(z_1, z_2)$ and $\alpha = z_1^{n_1} z_2^{n_2} \tilde u(z_1, z_2) \dd z_1 \wedge \dd z_2$ where $u, \tilde u$ are units with respect to both $z_1$ and $z_2$, assuming $E_1 \cup E_2$ are the only components of $K_X$ near $p$. 
 We consider the point blowup $\pi : Y \to X$ of $p$. 
 This blowup $\pi$ has an exceptional curve, say $E$, which is covered by two charts on $Y$; one such chart has coordinates $(z_1, w)$ where $z_1w = z_2$. Here, $w$ is the abscissa for $E = \{z_1 = 0\}$, and $E_1 = \{w = 0\}$. 
 In these new coordinates, \[x = z_1^{b_1+b_2}w^{b_1}u(z_1, z_1w)\text{, and}\]
 \[\alpha = z_1^{n_1+n_2} w^{n_2} \tilde u(z_1, z_1w) \dd z_1 \wedge (w\dd z_1 + z_1\dd w) = z_1^{n_1+n_2+1} w^{n_2} \tilde u(z_1, z_1w) \dd z_1 \wedge \dd w,\]
  so $\ord_{E}(x) = b_1 + b_2$ and $\ord_E(\alpha) = n_1 + n_2 + 1$.

 Now we take cases. Note that, by \autoref{lem:corresp:simplesection}, $p_\infty$ (if smooth on $X$) is necessarily a free point.
 
 1) Suppose that $p$ is free on the component $E_1$. Choose local coordinates $(z_1, z_2)$ at $p$ such that (locally) $E_1$ is the vanishing locus of $z_2$ (ordinate), and $z_1$ is any transverse abscissa; let $E_2$ be the prime divisor defined by $\{z_1 = 0\}$. Then in the above setup, $b_1 = \mb(E_1)$ and $b_2 = 0$, so $\mb(E) = \mb(E_1) + 0 = \mb(E_1)$.\\
 i) If $p \ne p_\infty$, then $E_1$ is an isolated zero of $\alpha$ on $X$, so $n_1 = \ma(E_1) -1$ and $n_2 = 0$. Therefore $\ord_E(\alpha) = \ma(E_1) -1 + 0 +1$, and so $\ma(E) = \ma(E_1) + 1$.\\
 ii) Otherwise, if $p = p_\infty$, then we ought to choose $z_1 = 1/y$ as the ordinate of $S_\infty$ so that $\tilde u$ is a unit; now $S_\infty$ takes the r\^ole of $E_2$. As explained above, $\ord_{S_\infty}(\alpha) = n_2= -2$ and again $n_1 = \ma(E_1) -1$. Thus $\ord_E(\alpha) = \ma(E_1) - 1 - 2 + 1 \implies \ma(E) = \ma(E_1) -1$.\\
 2) Suppose that $p$ is satellite on components $E_1$ and $E_2$. The choices of coordinates, $z_1$ and $z_2$, are naturally given as the ordinates of $E_2$ and $E_1$, respectively. Now in the above setup, $b_1 = \mb(E_1)$, $b_2 = \mb(E_2)$, $n_1 = \ma(E_1) -1$, and $n_2 = \ma(E_2) -1$. Therefore so $\mb(E) = \mb(E_1) + \mb(E_2)$ and $\ord_E(\alpha) = \ma(E_1) -1 + \ma(E_2) -1 + 1$, so $\ma(E) = \ma(E_1) + \ma(E_2)$.
\end{proof}

The arithmetical rules outlined above motivate us to call $\ma$ and $\mb$ \emph{Farey parameters}, meaning they obey the rules of Farey addition in the context of blowups. A goal of this section is to understand that Type II points form a generalised version of a `Farey sequence' on the tree precisely when they come from a smooth model. Let us briefly recall from \autoref{sec:fareyarithmetic} some of these properties of Farey sequences in $\Q$. Two adjacent Farey fractions $\frac {a_1}{b_1}$ and $\frac {a_2}{b_2}$ are \emph{adjacent} if their `determinant', $a_1b_2-a_2b_1 = \pm 1$, which also implies $\frac {a_1}{b_1} > \frac {a_2}{b_2}$ in the positive case. We will say that a sequence of rational numbers 
\[\frac {a_1}{b_1} < \frac {a_2}{b_2} < \cdots < \frac {a_n}{b_n}\]
is a \emph{Farey sequence} if each pair of successive fractions are Farey adjacent. %
\emph{Farey addition} is defined as follows and always gives the simplest fraction between two adjacent rational numbers.
\[\frac {a_1}{b_1} \oplus \frac {a_2}{b_2} = \frac{a_1+a_2}{b_1+b_2}\]

%
%
%
%
%
%
%
%
%
%
%
%
%
%
%
%
%
%
%
%
%
%

In general, we can view smooth models as a series of blowups from a minimal model, with various free and satellite blowups. The Farey addition above helps us understand and visualise any pattern of satellite blowups between two components. 

Keep in mind the dictionary below, true for components $E \sim \zeta \in \T_1$ with multiplicity $\mm = 1$, see \autoref{defn:corresp:freesat} and later also \autoref{lem:farey:freesat}.
\[
\begin{tikzcd}
 \text{Free component } E \arrow[r, leftrightsquigarrow]{} & \frac{\ma(E)}{\mb(E)} \in \Z\\
  \text{Satellite component } E \arrow[r, leftrightsquigarrow]{} & \frac{\ma(E)}{\mb(E)} \nin \Z
\end{tikzcd}
\]
\begin{figure}[ht]
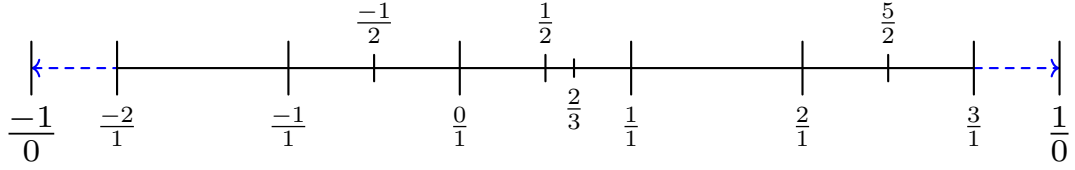

\def\ExtraFracList{1/2, -1/2, 2/3, 5/2}
\fareyline*{0.9\linewidth}{-2}{3}{1}
\caption{Blowups by Farey addition, with parameters at infinity for sections.}
\end{figure}

Furthermore, Farey addition can generate integer points, whence they simulate \emph{free} blowups. Suppose we introduce two other Farey fractions at infinity on the number line, $\frac{-1}0$ and $\frac 10$. If the lowest finite fraction is $\frac ab$ then it is adjacent to $\frac {-1}0$, so $b = 1$ and adding yields $\frac {-1}0 \oplus \frac a1 = \frac {a-1}1$; similarly, if $\frac a1$ the highest finite fraction then $\frac a1 \oplus \frac 10 = \frac {a+1}1$. %
In the theorem above we see $\ma \leadsto \ma \pm1$ (fixing $\mb$) after a free blowup, where $\frac 10$ parametrises a Type I point in an extended dual graph.

Of course, our universal dual graph is more than a real/rational line of points, it is an infinitely branched tree. At every vertex, there are not just two directions, but $\P^1(k)$-many directions, which generically give us choices of another integer point down the line. Most of the tree has multiplicity $\mm > 1$, but fortunately, the logic above extends if we use $\mb(E)/\mm(E)$ as the denominator of the Farey parameter with an important caveat: after a free blowup $F$ on a (e.g.\ satellite) component $E$, all further satellite blowups $E_i$ between $E$ and $F$ have multiplicity $m = \mm(E_i) = \mb(F) = \mm(\vec v_E(F))$. So in this new interval, we must pretend $E$ is an `integer point' despite $\mm(E) < m$. The next theorem clarifies the adjacency behaviour of  $\ma, \mb, \mm$ at smoothly intersecting divisors. Compare with \cite[Lemma 6.16]{FJ04}.

\begin{thm}\label{thm:farey:coprime}
 Let $X$ be a local model. Suppose that $E$ and $F$ are two components in $X_0$ which intersect at a smooth point $p = E \cap F$ of $X$; WLOG assume $E \prec F$. The following hold.
\begin{enumerate}[label=(\alph*)]
 \item $\HCF(\ma(E), \ma(F)) = 1$
 \item $\HCF(\mb(E), \mb(F)) = \mm(E)$ \label{item:farey:coprime:gen}
 \item \label{item:farey:coprime:matrix}Moreover, \[\det
\begin{pmatrix}
\ma(E) & \ma(F)\\
\mb(E) & \mb(F)
\end{pmatrix}
= \ma(E)\mb(F) - \ma(F)\mb(E) = \mm(E)\]
\item Hence, %
\[ d(E, F) = \frac{\mm(E)}{\mb(E) \mb(F)}.\]%
\item The multiplicity $\mm$ is constant on the interval $[E, F)$. \label{item:farey:coprime:constmult}
\end{enumerate}
\end{thm}

In \autoref{sec:galois:multiplicity} we defined the multiplicity $\mm(\bvec v)$ of a direction $\bvec v$ as the minimum multiplicity that occurs in direction $\bvec v$. Also recall \autoref{cor:multsubtreevalency} and \autoref{defn:specgendirn} that a Type II point $\zeta$ has up to two \emph{special} directions with multiplicity $\mm(\bvec v)$ less than the generic multiplicity $\mg(\zeta)$; all other directions are called \emph{generic} having multiplicity equal to $\mg(\zeta)$.

In \autoref{defn:specgendirn} we defined $\zeta$ to be free iff $\mm(\zeta) = \mg(\zeta)$, otherwise satellite. However, in \autoref{defn:corresp:freesat} we defined a divisor $E$ to be free or satellite depending on its blowup behaviour. Now we reconcile the two definitions and allay any concerns that the former was ill-defined. The following important lemma is interesting in its own right.

\begin{lem}\label{lem:farey:freesat}
 Let $X$ be a local model and let $E$ be a component of the special fibre. The following hold.
\begin{itemize}
   \item If $E$ is free then $\mm(E) = \mb(E)$.
   \item If $E$ is satellite, then $\mm(E) < \mb(E)$.
   \item Suppose that $E$ is obtained by blowing up a model $Y$ at the smooth intersection of components $F_1 \prec F_2$, then $\mm(E) = \mm(F_1)$. Furthermore, $\vec v(F_1)$ and $\vec v(F_2) = \vec v(\infty)$ are the two special directions at $E$.
   \item Suppose that $E$ is obtained by blowing up a model $Y$ at a smooth free point on the component $F$; then $\mm(E) = \mb(F)$ and $\vec v(E)$ is a generic direction at $F$.
\end{itemize}
\end{lem}

\begin{proof}
 We may assume WLOG that $X$ is smooth, since modifications away from $p$ do not affect the hypotheses or any numerical invariants. We induct on the components of the special fibre, starting with some smooth minimal model (any will do). We shall also prove \autoref{thm:farey:coprime} \ref{item:farey:coprime:constmult} here, as it is most appropriate.
 If $Y$ is a smooth model with special fibre $F$ then $\mb(F) = 1$ necessarily, for example by Tsen and \autoref{lem:corresp:simplesection}. Thus $\mm(F) = 1$ because $\mm(\zeta) \mid \mg(\zeta)$. So $F$ is free and $\mm(F) = \mb(F)$. Now assume that $Y$ is any smooth model and the results hold. Suppose that $p \in Y_0$ is a point and $\pi : \hat Y \to Y$ is a point blowup producing an exceptional curve $E \subset \hat Y$; we take cases.
 
 Suppose that $p = p_{Y, \infty}$ lies on $F$, then $p$ is free and $F \prec E$ by \autoref{prop:corresp:ordering}, and both $E$ and $F$ must be integral \autoref{thm:farey:addition}. It follows that $E$ is free, $\mm(E) = \mb(E) = 1$, and trivially $\vec v(E) = \vec v(\infty)$ is a generic direction of multiplicity $1$. Using \autoref{prop:multiplicityorder}, we see that $\mm \cong 1$ on $[F, E]$.
 
 Suppose that $p \ne p_{Y, \infty}$ is a free point lying on the component $F$, then $E \prec F$ by \autoref{prop:corresp:ordering}, and $\mb(E) = \mb(F)$ by \autoref{thm:farey:addition}. 
  If $F$ is free, then every direction other than $\vec v(\infty)$ is generic, including $\vec v(E)$ because $E \prec F$. 
 If $F$ is satellite, then by induction it is intersected by two other components $F_1, F_2$ of $Y$ lying in its two special directions $\vec v(F_1), \vec v(F_2)$. Clearly $F \cap F_1$ and $F \cap F_2$ are satellite points in $Y$ and $p$ is free so cannot be one of them. Hence, $\vec v_F(p) = \vec v(E)$ is distinct from $\vec v(F_1)$ and $\vec v(F_2)$, meaning $\vec v(E)$ is a generic direction. 
 Provided that $\vec v(E)$ is generic, \[\mb(E) = \mb(F) = \mm(\vec v(E)) \le \mm(E) \le \mb(E),\] using \autoref{cor:multsubtreevalency}; therefore $\mm(E) = \mb(E) = \mm(\vec v(E))$. 
 If $D \in [E, F)$ then $E \prec D$ so $\mm(D) \le \mm(E)$ by \autoref{prop:multiplicityorder}; on the other hand, $D \in \vec v(E)$ at $F$, so $\mm(D) \ge \mb(E)$; thus $\mm(D) = \mm(E)$, as required.
 
 Suppose that $p = F_1 \cap F_2$ is satellite; WLOG $F_1 \cap F_2$ by \autoref{prop:corresp:ordering}. Then by induction, $\mm(E) = \mm(F_1)$ because $E \in [F_1, F_2)$. By \autoref{thm:farey:addition}, \[\mb(E) = \mb(F_1) + \mb(F_2) \ge \mm(F_1) + \mm(F_2) > \max (\mm(F_1), \mm(F_2)).\]
 Therefore $E$ has generic multiplicity greater than that of either direction $\vec v(F_1)$ or $\vec v(F_2) = \vec v(\infty)$; indeed,  $\mb(E) > \mm(F_1) \ge \mm(\vec v(F_1))$ and $\mb(E) > \mm(F_2) \ge \mm(\vec v(F_2))$. Therefore $E$ has two special directions, $\vec v(F_1)$ and $\vec v(F_2)$, and so $\mm(E) < \mb(E)$ by \autoref{cor:multsubtreevalency}.
\end{proof}

\begin{proof}[Proof of \autoref{thm:farey:coprime}]
 Part \ref{item:farey:coprime:constmult} was done in the proof of the lemma. Again, we may desingularise $X$, or assume it is smooth. We will induct on the components of the special fibre, beginning with any minimal smooth model dominated by $X$. There is nothing to prove in the base case, where there is a single component of generic multiplicity $1$. Now assume that $Y$ is any smooth model and the results hold. Suppose that $p \in Y_0$ is a point and $\pi : \hat Y \to Y$ is a point blowup producing an exceptional curve $E \subset \hat Y$; we take cases.
  
 Suppose that $p = p_{Y, \infty}$ lies on $F$, then $p$ is free and $F \prec E$ (note the contrary notational ordering in this case) by \autoref{prop:corresp:ordering}. By \autoref{thm:farey:addition}, $\mb(E) = \mb(F) = 1$ and $\ma(E) = \ma(F) - 1$. Clearly, $\HCF(\ma(E), \ma(F)) = 1$, $\HCF(\mb(E), \mb(F)) = 1$, and \[\det
\begin{pmatrix}
\ma(F) & \ma(E)\\
\mb(F) & \mb(E)
\end{pmatrix}
  = \det
  \begin{pmatrix}
\ma(F) & \ma(F) -1\\
1 & 1
\end{pmatrix}
  =1\]
  Since $\mm(F) = 1$, this case is done.

Suppose that $p \ne p_{Y, \infty}$ is a free point lying on the component $F$, so $E \prec F$. Then $\ma(E) = \ma(E) + 1$ and $\mb(E) = \mb(F)$ by \autoref{thm:farey:addition}. Also, \autoref{lem:farey:freesat} states that $\mm(E) = \mb(E) = \mb(F)$ since $E$ is free. Similarly, $\HCF(\ma(E), \ma(F)) = \HCF(\ma(E), \ma(E) + 1) = 1$, $\HCF(\mb(E), \mb(F)) = \mb(F) = \mm(E)$, and \[\det
\begin{pmatrix}
\ma(E) & \ma(F)\\
\mb(E) & \mb(F)
\end{pmatrix}
  = \det
  \begin{pmatrix}
\ma(F) + 1 & \ma(F)\\
\mb(F) & \mb(F)
\end{pmatrix}
  = (\ma(F) + 1)\mb(F) - \ma(F)\mb(F) = \mb(F) = \mm(E).\]

Finally, suppose that $p = F_1 \cap F_2$ is satellite with $F_1 \prec F_2$. By \autoref{thm:farey:addition}, \[\ma(E) = \ma(F_1) + \ma(F_2)\quad \and \quad \mb(E) = \mb(F_1) + \mb(F_2).\] It remains to verify the theorem for both new intersecting pairs of divisors $F_1 \prec E$ and $E \prec F_2$. We prove the former and leave the latter as a partial exercise. The first items fall to elementary number theory, using the inductive hypothesis. \[\HCF(\ma(F_1), \ma(E)) = \HCF(\ma(F_1), \ma(F_1) + \ma(F_2)) = \HCF(\ma(F_1), \ma(F_2)) = 1\] \[\HCF(\mb(F_1), \mb(E)) = \HCF(\mb(F_1), \mb(F_1) + \mb(F_2)) = \HCF(\mb(F_1), \mb(F_2)) = \mm(F_1)\] 
Next, one can use column operations (or otherwise) to show that \[\det
\begin{pmatrix}
\ma(F_1) & \ma(E)\\
\mb(F_1) & \mb(E)
\end{pmatrix}
  = \det
  \begin{pmatrix}
\ma(F_1) & \ma(F_1) + \ma(F_2)\\
\mb(F_1) & \mb(F_1) + \mb(F_2)
\end{pmatrix}
= \det
  \begin{pmatrix}
\ma(F_1) & \ma(F_2)\\
\mb(F_1) & \mb(F_2)
\end{pmatrix}
  = \mm(F_1).\]
  For the pair $E$ and $F_2$, we similarly obtain $\HCF(\ma(E), \ma(F_2)) = 1$, $\HCF(\mb(E), \mb(F_2))) = \mm(F_1)$, and the determinant calculation results in $\mm(F_1)$ by expanding and simplifying the left column instead. In this instance, note that $\mm(F_1) = \mm(E)$ by part \ref{item:farey:coprime:constmult}, as proved in the lemma. This concludes the induction. For the last part, we simply apply our new formula.
   \[d(E, F) = \frac{\ma(E)}{\mb(E)} - \frac{\ma(F)}{\mb(F)} = \frac{\ma(E)\mb(F) - \ma(F)\mb(E)}{\mb(E)\mb(F)} = \frac{\mm(E)}{\mb(E) \mb(F)}\]
\end{proof}

\begin{cor}
 Let $X$ be a local model and $E$ a component of the special fibre. Then $\mm(E) \mid \mb(E)$ and \[\HCF\left(\ma(E), \frac{\mb(E)}{\mm(E)}\right) = 1.\] In particular, if $\frac {\mm(E)\ma(E)}{\mb(E)}$ is written as $\frac {a'}c$ in lowest terms, then $a' = \ma(E)$ and $\mb(E) = c \cdot \mm(E)$. 
\end{cor}

This means one can determine $\ma$ and $\mb$, only knowing $\diam(\emb_X(E)) = \abs x ^{\ma(E)/\mb(E)}$ and $\mm(E)$. We leave it as an exercise to give an alternative proof more directly from Puiseux series representations of Type II points.

\begin{proof}
 Let $\hat X$ be a smooth model dominating $X$. First assume $E$ is not maximal so there exists a curve $F \succ E$ intersecting $E$ in $\hat X$. Now, by the theorem part (\ref{item:farey:coprime:matrix}), 
 \[\det
\begin{pmatrix}
\ma(F) & \ma(E)\\
\mb(F) & \mb(E)
\end{pmatrix}
  = \mm(E)\]
 and by part (\ref{item:farey:coprime:gen}), the bottom row of the matrix is divisible by $\mm(E)$. Thus
 \[\det
\begin{pmatrix}
\ma(F) & \ma(E)\\
\frac{\mb(F)}{\mm(E)} & \frac{\mb(E)}{\mm(E)}
\end{pmatrix}\]
has determinant $1$, which by B\'ezout's Theorem, is proof that every row and column pair are coprime. In particular, $\HCF\left(\ma(E), \frac{\mb(E)}{\mm(E)}\right) = 1$.

Otherwise, if $E$ is maximal, then $E = E_{\hat X, \infty}$, and so $\mm(E) = \mm(B) = 1$ by \autoref{thm:farey:addition}; thus the result is trivial.

For the final piece, note that if $\ma(E)$ and $\frac{\mb(E)}{\mm(E)}$ are coprime then $\frac{\ma(E)}{{\mb(E)}/{\mm(E)}}$ is indeed the representation of $\frac{\ma(E)\mm(E)}{\mb(E)}$ in lowest terms.
\end{proof}

\begin{ex}
 Suppose that $E$ is such a divisor with $\mm(E) = 6$, $\frac{\ma(E)}{\mb(E)} = \frac 59$. For instance, perhaps $E = \redct_X(\zeta(x^{1/3} - 7x^{1/2}, \abs x ^{5/9}))$. Then in fact $\ma(E) = 10$ and $\mb(E) = 18$ because \[\frac {\mm(E)\ma(E)}{\mb(E)} = \frac {30}{9} = \frac {10}{3}\] and $3 \cdot \mm(E) = 18$. Observe that $\HCF(10, 3) = 1$.
\end{ex}

\begin{ex}\label{ex:farey:multtwo}
 Consider the surface $\P^1 \times \P^1$, with the central fibre $E_0 = \set{x=0} \cong \P^1$ modelled by $\zeta(0, 1)$. After blowing up $(0, 0)$ to obtain $E_1$, and then the intersection $E_0 \cap E_1$ to obtain $E_{\frac 12} = \redct(\zeta(0, \abs x^{\frac 12}))$, we can blow up some free point of $E_{\frac 12}$ to get the exceptional curve $D_1$, associated with $\zeta(cx^{\frac 12}, \abs x)$ for some non-zero $c \in \C$. Finally, blowing up $E_{\frac 12} \cap D_1$ yields a curve $D_{\frac 34} = \redct(\zeta(cx^{\frac 12}, \abs x^{\frac 34}))$. All this can be verified using \autoref{thm:farey:addition}.
 
 Now, the intersecting divisors $E_{\frac 12}$ and $D_{\frac 34}$ are each smooth on the smooth surface. Their Farey parameters $\frac 12, \frac 34$ do not obey a standard adjacency law for fractions, since $3 \cdot 2 - 1 \cdot 4 = 2$. However, $2 = \mm\left(D_{\frac 34}\right) = \HCF\left(\mb\left(E_{\frac 12}\right), \mb\left(D_{\frac 34}\right)\right)$. Indeed, dividing the denominators through by this $2$, yields adjacent Farey fractions.
\end{ex}

The last part of \autoref{thm:farey:coprime} gives us that $d(E, F) \le \frac{1}{\mb(F)}$ with equality if and only if $E$ is free. However, we can find a tighter upper bound on distance which depends on $E$.

\begin{prop}
 Let $X$ be a local model. Suppose that $E$ and $F$ are two components in $X_0$ which intersect at a smooth point $p = E \cap F$ of $X$; WLOG assume $E \prec F$. 
 \begin{enumerate}[label=(\roman*)]
 \item If $F$ is satellite but $\vec v(E)$ is a generic direction at $F$, then $\mm(E) = \mb(F) > \mm(F)$.
 \item Otherwise, $\mm(E) = \mm(F)$.
\end{enumerate}
\end{prop}

\begin{proof}
 This is an addendum to the induction completed for \autoref{lem:farey:freesat}. 
 Suppose $E$ is a free blowup from $F$. If $F$ is also free then we have $\mm(E) = \mb(E) = \mb(F) = \mm(F)$. If $F$ is satellite then $\vec v(E)$ must be a generic direction by the lemma, so $\mm(E) = \mm(\vec v(E)) = \mb(F) > \mm(F)$. 
 Suppose $E$ is produced by satellite blowup between $F_1 \prec F_2$. Here, the proposition holds for the pair $F_1 \prec E$ since $\mm(E) = \mm(F_1)$, $E$ is satellite and $\vec v(F_1)$ is special. Considering the pair $E \prec F_2$, we find that because $\mm(E) = \mm(F_1)$ and $\vec v(E) = \vec v(F_1)$ at $F_2$, the various (in)equalities hold if and only if they held for $F_1$ and $F_2$.
\end{proof}

\begin{cor}\label{cor:farey:dist}
 Let $X$ be a model, and suppose $E \prec F$ intersect at a smooth point. Then one has \[d(E, F) \le \frac{1}{\mb(E)},\]
 which is strict if and only if $F$ is satellite and $E$ lies in a special direction at $F$.
\end{cor}

\autoref{thm:farey:puiseuxsection} showed that for a smooth free point, $p$ on $E$ the direction $\vec v(p)$ is generic ($\mm(p) = \mg(\vec v(p)) = \mb(E)$) and so for every resulting $F$ infinitely near $p$, the (generic) multiplicities $\mb(F)$ are \emph{multiples} of $\mb(E)$. The theorem also stated that at a smooth satellite point $p = E_1 \cap E_2$, for any $F$ infinitely near $p$, have $\mb(F) \ge \mb(E_1) + \mb(E_2)$. The next theorem will reprove the former fact with new perspective, and upgrades the latter.

Recall the multiplicity $\mm$ of points on the Berkovich projective line over the Puiseux series. Just as $\mg(\zeta)$ is equivalent to $\mb(\redct(\zeta))$, we will \emph{define} $\mm(E) = \mm(\zeta)$ where $\redct(\zeta) = E$; this is well-defined because $\mm$ is the size of Galois orbit of $\zeta$ in $\P^1_\an(\K)$. %

\subsection{Criteria for Smoothness}\label{sec:fareyalt}

Here we compare smoothness of the model surface to a generalised notion of Farey sequence on $\P^1_\an(\K)$. We have already seen that blowups correspond to Farey addition of the Farey parameter; it will also be enlightening to discuss adjacency conditions.

\autoref{thm:farey:puiseuxsection} says that a smooth point $p \in X$ has a $\Gamma(X)$-domain whose constituent vertices have nice multiplicities, described by the local geometry of $X$ at $p$. The Farey parameters of the dual graphs of a smooth model $X$ clearly exhibit the three equivalent conditions which (general) Farey sequences of rational numbers do. The last two options require us to include points at infinity within an extended dual graph (i.e.\ Puiseux sections) whose Farey parameters are $\pm \frac 10$.

\begin{itemize}
 \item In the $\Gamma$-domain $U$ between every two points $\alpha$ and $\beta$ with Farey parameters $\frac{a_1}{b_1} < \frac{a_2}{b_2}$, we only find points with Farey fractions $\frac ab$ such that $b > \max(b_1, b_2)$. Endpoints of the set are `integral' i.e.\ free.
 \item The sequence can be generated purely by Farey addition on the starting sequence $\frac{-1}{0}, \frac a1, \frac 10$, i.e.\ generated by blowups from a smooth minimal model with a single component $E = \emb(\zeta(\gamma, \abs x ^a))$ where $\gamma \in \k((x))$.
 \item Each neighbouring pair of points $\alpha \succ \beta$ with parameters $\frac{a_1}{b_1/m} < \frac{a_2}{b_2/m}$ have `coprime' parameters: $a_2(b_1/m) - a_1(b_2/m) = 1$, where $m = \mm(\alpha)$ is determined only by $\beta$.
\end{itemize}
We encapsulate the first condition on $\P^1_\an$ in a definition, and then restate this implied corollary of \autoref{cor:smooth:multdesc}.

\begin{defn}\label{defn:smooth:smooth}
 Let $U \subset \P^1_\an$ be an open set bounded by Type II points. We will say that $U$ is \emph{\dsmooth} if and only if either
\begin{enumerate}
 \item $U$ is a disk with boundary point $\zeta$ and $\mm(U) = \mg(\partial U)$, i.e.\ $U$ is a generic direction at $\zeta$; or
 \item $U$ is an annulus with boundary points $\alpha, \beta$ and $\mg(\zeta) > \max(\mg(\alpha), \mg(\beta))$ for all $\zeta \in U$. %
\end{enumerate}

 Let $\Gamma \subset \P^1_\an(\k((x)))$ be a vertex set (a finite set of Type II points).
 We say that $\Gamma$ is \emph{\vsmooth} if and only if each of its $\Gamma$-domains are \dsmooth{}.
 If instead we consider $\Gamma \in \P^1_\an(\K)$, then we also demand that $\Gamma$ is Galois-invariant.
\end{defn}

From the axioms for a \vsmooth{} vertex set $\Gamma$, various observations about vertex sets of smooth models also follow for $\Gamma$. We start with those that align the weak formal definition to the motivational one. For instance \autoref{thm:smooth:multdesc} will give an even stronger result than $\mg(\zeta) \ge \mg(\alpha) + \mg(\beta)$ in a \dsmooth{} annulus.

\begin{prop}
 Suppose that $\Gamma$ is a \vsmooth{} vertex set and $\zeta$ is an endpoint of the induced graph $\Delta(\Gamma)$. Then $\zeta$ is free.
\end{prop}

\begin{proof}
 meaning $\Gamma \sm \set \zeta$ is contained in a single direction $\bvec v$ at $\zeta$. This means any other direction $\bvec u \ne \bvec v$ at $\zeta$ is a $\Gamma$-disk, which must be generic if $\Gamma$ is \vsmooth{}. Any Type II point with at most one special direction must be free.
\end{proof}

\begin{prop}\label{prop:smooth:ordered}
 If $U$ is a \dsmooth{} $\Gamma$-annulus with boundary points $\alpha, \beta$, then $\alpha \prec \beta$ or vice versa. Moreover, $\frac{\ma(\beta)}{\mg(\beta)/m} < \frac{\ma(\alpha)}{\mg(\alpha)/m}$ are neighbouring Farey fractions, where $m = \mm(\gamma) = \mm(\alpha) \mid \mg(\beta)$.
 \end{prop}

\begin{proof}
 If not, then have a join $\zeta$ such that $\zeta \succneq \alpha$ and $\zeta \succneq \beta$. But then clearly $\vec v_\zeta(\infty) \subset U$, so $\mg(U) = 1$ which contradicts $\mg(U) > \max(\mg(\alpha), \mg(\beta))$. Now let $\alpha = \zeta(\gamma, \abs x^{\frac{\ma(\alpha)}{\mg(\alpha)/m}})$ and $\beta = \zeta(\gamma, \abs x^{\frac{\ma(\beta)}{\mg(\beta)/m}})$ where $m = \mm(\gamma) = \mm(\alpha) \mid \mg(\beta)$. Hence $\frac{\ma(\beta)}{\mg(\beta)/m} < \frac{\ma(\alpha)}{\mg(\alpha)/m}$ must be neighbouring Farey fractions by virtue of the definition. 
\end{proof}

Consider \autoref{lem:corresp:simplesection} and \autoref{prop:corresp:ordering}. Suppose for instance that $\xi$ is a point (of any type) of generic multiplicity $m$. If $\xi$ lies in a \dsmooth{} $\Gamma$-disk $U$ bounded by $\zeta \in \Gamma$, then $\mg(\zeta) = \mm(U) \mid m$. If $\xi$ lies in a $\Gamma$-annulus $U$ bounded by $\zeta_1, \zeta_2 \in \Gamma$, then \[\max(\mg(\zeta_1), \mg(\zeta_2)) < \mg(\zeta_1) + \mg(\zeta_2)  \le m.\] 
In particular, if $m = 1$ and $\xi$ lies in a \dsmooth{} $\Gamma$-domain $U$, then $U$ is a $\Gamma$-disk whose boundary point is integral. Choosing the Type I point $\xi = \infty$ shows that a \vsmooth{} vertex set $\Gamma$ must contain an integral point $\zeta_\infty$ whose direction $\vec v(\infty)$ is a $\Gamma$-disk. %
Moreover, one can see that the point $\zeta_\infty$ outlined above is the maximum amongst $\Gamma$.

\begin{prop}\label{prop:smooth:basics}
 Let $\Gamma \subset \P^1_\an$ be a vertex set. 
\begin{enumerate}
\item If $\xi \in \P^1_\an \sm \Gamma$ has generic multiplicity $1$ and is contained in a \dsmooth{} $\Gamma$-domain $U$, then $U$ is a $\Gamma$-disk bounded by an integral point.
\item In particular, if $\Gamma$ is \vsmooth{}, then it contains an integral point $\zeta_\infty$, which is a maximum point in $\Gamma$ with respect to $\prec$.
\end{enumerate}
\end{prop}

\begin{rmk}\label{rmk:smooth:semigroup}
 Given $a, b \in \N^+$, let $\N^+\langle a, b\rangle$ denote the commutative semigroup $\set[xa + yb]{x, y \in \N^+}$. Here, $\N^+ = \N \sm \set 0$. Caution: neither $a$ nor $b$ are elements of this set; in fact $\N^+\langle a, b\rangle = \N \langle a, b\rangle \sm (\N \langle a \rangle \cup \N \langle b \rangle)$. Note that $\N^+\langle a, b\rangle$ is closed under addition of its elements and furthermore by adding simply $a$ or $b$.
\end{rmk}

\begin{thm}\label{thm:smooth:multdesc}
\begin{enumerate}[label=(\alph*)]
 \item Suppose $U$ is a \dsmooth{} disk with boundary point $\xi$, then $\mg(\xi) = \mm(U) \mid \mm(\zeta) \mid \mg(\zeta)$ for every $\zeta \in U$. In effect, for every $\zeta \in U$, $\mg(\zeta) \in \N^+\langle \mb(E) \rangle$.
 \item Suppose that $U$ is the \dsmooth{} annulus bounded by $\alpha \prec \beta$, then $\mg(\zeta) \in \N^+\langle\mg(\alpha), \mg(\beta)\rangle$ for every $\zeta \in U$.
\end{enumerate}
\end{thm}

\begin{proof}
(a)\: This is \autoref{prop:galois:multindirn}.

(b)\: First, let $\zeta \in (\alpha, \beta)$ be a Type II point and $\alpha = \zeta(\gamma, \abs x^{\frac{\ma(\alpha)}{\mg(\alpha)}})$; we may set $m = \mm(\gamma) = \mm(\alpha)$. Then $\zeta = \zeta(\gamma, \abs x^\frac ab)$, where $m \mid b = \mg(\zeta)$. So, $\frac{\ma(\beta)}{\mg(\beta)/m} < \frac a{b/m} < \frac{\ma(\alpha)}{\mg(\alpha)/m}$. Therefore, some sequence of Farey additions between $\frac{\ma(\alpha)}{\mg(\alpha)/m}$ and $\frac{\ma(\beta)}{\mg(\beta)/m}$ yield $\frac a{b/m}$. Therefore, $b \in \N^+\langle \mg(\alpha), \mg(\beta) \rangle$. 

Second, suppose $\zeta \in U \sm (\alpha, \beta)$. Then there is a $\xi \in (\alpha, \beta)$ and $\zeta$ lies in a generic direction $\bvec v$ at $\xi$. By the above we have $\mg(\xi) \in \N^+\langle \mg(\alpha), \mg(\beta) \rangle$ and by part (a), we have $\mg(\xi) = \mm(U) \mid \mm(\zeta) \mid \mg(\zeta)$, so certainly $\mg(\zeta) \in \N^+\langle \mg(\alpha), \mg(\beta) \rangle$.
\end{proof}

\begin{rmk}
 There is an alternate geometric proof of \autoref{cor:smooth:multdesc} inducting over smooth models, using \autoref{thm:farey:addition} directly, with \autoref{rmk:smooth:semigroup}. For instance, the Farey addition utilised in part (b) corresponds to a sequence of satellite blowups of smooth points between $E_1, E_2$. Moreover, one can see it implies the theorem.
\end{rmk}

 The next theorem is a restatement of \autoref{thm:farey:puiseuxsection} which shows that smooth points have \dsmooth{} $\Gamma$-domains, along with a series of corollaries.

\begin{thm}\label{thm:smooth:smimpvsm}%
 Let $X$ be a model and let $\Gamma = \Gamma(X) \subset \P^1_\an$ be the corresponding vertex set. If $U$ is a $\Gamma$-domain and $\redct_X(U) = p$ is a smooth point, then $U$ is \dsmooth{}. Hence, if $X$ is smooth at each closed point in $X_0$, then $\Gamma$ is \vsmooth{}. 
\end{thm}

\begin{rmk}
 \autoref{lem:farey:freesat} provides an inductive alternative proof to \autoref{thm:farey:puiseuxsection}.
\end{rmk}

\begin{cor}\label{cor:smooth:multdesc}
 Let $X$ be a local model and $p \in X_0$ be a smooth point and $U$ its associated $\Gamma(X)$-domain.
\begin{enumerate}[label=(\alph*)]
 \item If $p$ is free on $E$ and $U$ is a $\Gamma(X)$-disk bounded by $E$, then $\mb(E) = \mm(U) \mid \mm(F) \mid \mb(F)$ for every $F \in U$. In other words, for every $F$ infinitely near $p$, $\mb(F) \in \N^+\langle \mb(E) \rangle$.
 \item If $p = E_1 \cap E_2$ is satellite, and $U$ is the $\Gamma(X)$-annulus bounded by $E_1, E_2$, then $\mb(F) \in \N^+\langle\mb(E_1), \mb(E_2)\rangle$ for every $F \in U$ (or in other words, for every $F$ infinitely near $p$).
\end{enumerate}
\end{cor}

\begin{proof}
Combine \autoref{thm:smooth:smimpvsm} with \autoref{thm:smooth:multdesc}.
\end{proof}

Now we study the converse direction to \autoref{thm:smooth:smimpvsm}, meaning that every \vsmooth{} vertex set comes from a smooth model. The approach we take is to justify that any desingularisation of a point is unnecessary: meaning the exceptional curves resulting from blowing up the supposed singularity can be blown down again. This, \autoref{lem:smooth}  is also an interesting generalisation to \autoref{prop:corresp:selfint}.

\begin{lem}\label{lem:smooth}
 Let $X$ be a model of $\check X$ over $\bO$ and let $\Gamma = \Gamma(X) \subset \P^1_\an$. Suppose that  $U$ is a \dsmooth{} disk or annulus bounded by points of $\Gamma$ and that $X$ is smooth over the locus of points $\redct_X(U)$. If $U \cap \Gamma \ne \emp$, then there is a vertex $\zeta \in \Gamma \cap U$ which reduces to a smooth rational curve $E = \redct_X(\zeta)$ of self-intersection $E^2 = -1$, i.e.\ a contractible one.
\end{lem}

\begin{thm}\label{thm:smooth}
Let $X$ be a model of $\check X$ over $\bO$ and let $\Gamma = \Gamma(X) \subset \P^1_\an$ be its vertex set (a finite set of Type II points), and $U$ a $\Gamma$-domain. Then $U$ is \dsmooth{} if and only if $\redct_X(U) = p$ is a smooth point of $X$.

In particular, any given vertex set (a finite set of Type II points) $\Gamma \subset \P^1_\an$ is \vsmooth{} if and only if there exists a model $X$ of $\check X$ over $\bO$ such that $\Gamma(X) = \Gamma$ and $X$ is smooth at every $p \in X_0$.
\end{thm}

\begin{proof}%
 Let $X, \Gamma$ be as described with $U$ a $\Gamma$ domain and $\redct_X(U) = p$. If $p$ is smooth in $X$, then $U$ is \dsmooth{} by \autoref{thm:smooth:smimpvsm}. Further, if $X$ is a model over $\bO$ which is smooth at every point in the special fibre, then every $\Gamma(X)$-domain is \dsmooth{}, meaning $\Gamma(X)$ overall is \vsmooth{}. The remaining task is to prove that if $U$ is \dsmooth{}, then $p$ is a smooth point of $X$. For contradiction, suppose that $U$ is \dsmooth{} but $p$ is singular, requiring a non-trivial smooth desingularisation $\pi : \hat X \to X$ over $p$; assume $\pi$ is minimal with respect to the number of components in $\pi^{-1}(p)$. Therefore, we may assume that the locus $L = \redct_{\hat X}(U) \subset \hat X$ is smooth. It follows by applying \autoref{lem:smooth} to $\hat X$ and $U$ that there is a (smooth rational) curve $E$ within $\pi^{-1}(p)$ with self intersection $E^2 = -1$. By Castelnuovo's criterion, there is a smooth $X'$ and blowdown $\rho : \hat X \to \hat X'$ which contracts $E$ to a point; see \autoref{thm:castelnuovo}. By the universal property of point blowups, $\pi$ factors as $\pi = \pi' \circ \rho$ where $\pi' : \hat X' \to X$. This contradicts the minimality of $\pi$. Therefore $\pi = \id$, meaning $X$ is smooth at $p$.

 Finally, suppose $\Gamma$ is \vsmooth{}. Using \autoref{thm:corresp:modelexistence}, there exists a model $X$ of $\check X$ over $\bO$ which has vertex set $\Gamma(X) = \Gamma$. 
Now, for any closed point $p$ in $X_0$, $\redct_X^{-1}(p)$ is a \dsmooth{} $\Gamma$-domain, whence $p$ is smooth by the above. Hence $X$ is smooth over every point of its special fibre.
\end{proof}
 
 \begin{proof}[Proof of \autoref{lem:smooth}]
 First note that since every point of $\redct_X(U)$ is smooth, by \autoref{thm:smooth:smimpvsm}, every $\Gamma(X)$-domain contained in $U$ is \dsmooth{}. That $E$ will be smooth and rational is immediate from the construction and hypotheses.

Let $\zeta$ be a $\mg$-maximal and otherwise $\prec$-minimal element of $\Gamma(X) \cap U$. However, if $\infty \in U$, instead require that $\zeta$ is $\prec$-maximal. In either case, if $U$ is a disk then we can also assume that $d_\bH(\zeta, \partial U)$ is maximised. Define $E = \redct_{X}(\zeta)$.

The rest is split into cases, but first we make a few observations.\\
\textbf{Observation I} For any other $\xi \in \Gamma(X)$, either $\vec v(\xi)$ is special or intersects $\partial U$. Indeed, if $\bvec v \ne \vec v(\partial U)$ is generic and $\xi \in \bvec v \cap \Gamma(X)$, then $\mg(\xi) \ge \mm(\bvec v) = \mg(\zeta)$ and $\xi$ would have a greater distance from $\partial U$ than $\zeta$ does. This is contrary to the choice of $\zeta$.\\
\textbf{Observation II} If $E$ is not integral then $\partial U$ and is contained in the special directions at $\zeta$, necessarily intersecting $\vec v(\infty)$. 
Indeed, if $E$ is not integral, then $\vec v(\infty)$ is guaranteed to be a special direction and $\infty \nin U$ by \autoref{prop:smooth:basics}. Therefore, $\vec v(\infty)$ intersects $\partial U$. In particular, if $U$ is a disk then $\partial U \in \vec v(\infty)$. If $U$ is an annulus then $\mg(\zeta) > \max(\mg(\alpha), \mg(\beta))$, so both $\alpha$ and $\beta$ lie in (possibly the same) special directions at $\zeta$.

\textbf{Observation III} In any situation, $\mg(\xi) \le \mg(\zeta)$ for any $\xi \in \overline U$, using either our choice of $\zeta$ with maximal generic multiplicity, or the definition of \dsmooth{}ness of $U$. In particular, if $E$ and another component $F$ intersect, then $\mb(F) \le \mb(E)$.

for any $\xi \in \overline U$, using either our choice of $\zeta$ with maximal generic multiplicity, or the definition of \dsmooth{}ness of $U$. In particular, if $E$ and another component $F$ intersect, then $\mb(F) \le \mb(E)$. 

\textbf{Case: $E$ is integral.}
First note that because $U$ must be a disk bounded by an integral point, as explained by \autoref{prop:smooth:basics}. Since every direction is generic, observation I says that the rest of $\Gamma(X)$ is contained in $\vec v(\partial U)$. Hence, $E$ intersects only one other component of the special fibre, $F$ say, the reduction of $\zeta' \in \Gamma(X) \cap \overline U \cap \vec v(\partial U)$. Now, from \autoref{lem:farey:intersections}, $\mb(E)E^2 + \mb(F) = 0$. Since $\mb(E) = \mb(F) = 1$ by assumption, we have $E^2 = -1$.

\textbf{Case: $E$ is free with $\mb(E) > 1$.} 
In this case, $\zeta$ has a unique special direction $\vec v(\infty)$. By observation II, $\partial U \subset \vec v(\infty)$. 
We find that $E$ intersects only one other component of the special fibre, $F$ say, the reduction of $\zeta' \in \Gamma(X) \cap \overline U \cap \vec v(\partial U)$. So $\mb(E)E^2 + \mb(F) = 0$ and thus  \[0 < -E^2 \le \frac{\mb(F)}{\mb(E)} \le 1,\] where the last inequality comes from observation III.
 One assumes $E^2$ is an integer because $X$ is smooth over $E$; therefore $E^2 = -1$.

\textbf{Case: $E$ is satellite.} Here, $\zeta$ has exactly two special directions and recall that $\vec v(\infty)$ contains at least one boundary point of $U$. If $U$ is an annulus bounded by $\alpha \prec \beta$, then $\alpha$ lies in the lower special direction. 
Again, by our first and second observations, every generic direction is contained in $U$ and these directions contain no points of $\Gamma(X)$. In other words, $\Gamma(X)$ is contained in the two special directions at $\zeta$.
Thus, $E$ may intersect two other components $E_1, E_2 \in \redct_{X}(\overline U)$ of the special fibre $X_0$. We may assume that $E_1 \prec E \prec E_2$ by \autoref{prop:smooth:ordered}. Using observation III, we know $\mb(E_2) \le \mb(E)$ and furthermore that $\mb(E_1) < \mb(E)$, by either our $\prec$-minimal choice of $\zeta \in \Gamma(X) \cap U$ or because $\mg(\alpha) < \mg(\zeta)$ in the situation that $U$ is an annulus and $E_1 = \redct_{X}(\alpha)$. By \autoref{lem:farey:intersections}
\[\mb(E)E^2 + \mb(E_1) + \mb(E_2) = 0\]
whence 
\[0 < -E^2 = \frac{\mb(E_1) + \mb(E_2)}{\mb(E)} < 2,\]
and therefore $E^2 = -1$.

This amounts to the contradiction, proving $p$ is in fact a smooth point of $X$.
\end{proof}

\begin{rmk}
 The theorem admits an inductive proof that avoids any application of desingularisation of surfaces, assuming simply the existence of any smooth birational model (such as $C \times \P^1$). In fact, it is not even required to know a priori that $\Gamma = \Gamma(X)$ for some $X$. The sketch is as follows. First, any smooth model has a component with self intersection $-1$ (\autoref{prop:corresp:selfint}). Blowing down successively reaches a model $Y_{(\zeta)}$ with just one component $E_0 = \redct(\zeta)$ where $\zeta$ is integral. Second, by induction, one can show that for any integral Type II point $\xi$ a surface $Y_{(\xi)}$ can be obtained from $Y_{(\zeta)}$ through a repeated sequence of blowup and blowdown pairs (successively picking integral points along $[\zeta, \xi]$). Next, any \vsmooth{} vertex set $\Gamma$ contains an integral point by \autoref{prop:smooth:basics}.
 
 Henceforth, induct by size and containment of vertex sets which contain a fixed integral point $\xi$. Certainly the base case $Y_{(\xi)}$ is smooth as above. Now, if $\Gamma$ is \vsmooth{}, then the proof above for the lemma defines a candidate $E = \redct(\zeta)$ that `ought to be blown down'. Instead, one can check that $\Gamma' = \Gamma \sm \set \zeta$ is also \vsmooth{}. Again, there are two main cases to check: where $E$ is free or satellite, requiring slightly different analyses. As in the proof of the lemma, we find that $E$ only intersects one or two components, and its generic multiplicity is at least as high. This is the bulk of the work done. Let $U$ be the $\Gamma'$-domain containing $\zeta$. On one hand, $\zeta$ is the unique vertex that can be added to $\Gamma'$ in $U$ which makes $\Gamma' \cup \set \zeta$ \vsmooth{}, and on the other hand, $E = \redct(\zeta)$ is the exceptional curve obtained by the point blowup at $p = \redct(U)$. All these facts are obtained by studying the generic multiplicity and using Farey addition.
\end{rmk}

In closing this section, we summarise the equivalent conditions for a \dsmooth{} $\Gamma$-annulus, which we have by now discovered. This further expounds upon our generalisation of Farey sequences from the rational number line to $\P^1_\an(\K)$, such as the usual adjacency condition for $\frac ab > \frac cd$, $ad-bc =1$. For $E \sim \zeta(\gamma, \abs x^\frac ab)$ and $F \sim \zeta(\gamma, \abs x^\frac cd)$, with $b = \mb(E),\ d = \mb(F)$, a necessary condition was seen in \autoref{thm:farey:coprime} for smooth points as $ad-bc = \mm(\alpha)$. However, this is not quite sufficient for smoothness as the remark demonstrates.

\begin{thm}\label{thm:farey:annulustfae}
 Let $\Gamma \subset \P^1_\an$ be a finite vertex set of Type II points, and $\Delta$ the induced graph. Let $U$ be a $\Gamma$-annulus bounded by $\alpha$ and $\beta$ TFAE
\begin{enumerate}
 \item $U$ is \dsmooth{}: $\mb(U) > \max(\mb(\alpha), \mb(\beta))$.
 \item $\mb(\zeta) > \max(\mb(\alpha), \mb(\beta))$ for every $\zeta \in (\alpha, \beta)$.
 \item $\mg(\zeta) \in \N^+\langle \mb(\alpha), \mb(\beta)\rangle$ for all $\zeta \in U$.
 \item $\alpha \prec \beta$ and $\ma(\alpha)\mb(\beta)-\mb(\alpha)\ma(\beta) = \mm(\alpha) = \HCF(\mb(\alpha), \mb(\beta))$ or vice versa.
\end{enumerate}
\end{thm}

\begin{rmk}\label{rmk:smooth:adjacency}
 Weakening the condition to $\ma(\alpha)\mb(\beta)-\mb(\alpha)\ma(\beta) = \HCF(b, d)$ is impossible as $\alpha = \zeta(0, \abs x^\frac 13),\ \beta = \zeta(0, \abs x^\frac 23)$ shows. Nor can the condition be presented as $\ma(\alpha)\mb(\beta)-\mb(\alpha)\ma(\beta) = \mm(\alpha)$. This is more subtle than the former, but the pair $\alpha = \zeta(x^\frac 13, \abs x^\frac 12),\ \beta = \zeta(0, 1)$ provides a non-\dsmooth{} counter-example with $\ma(\alpha) = 3,\ \mb(\alpha) = 6,\ \ma(\beta) = 0,\ \mb(\beta) = 1,\ \gamma = x^{\frac 13}$.
\end{rmk}

\clearpage


\appendix
\section{Proofs with Puiseux Series}\label{sec:puiseuxappx}

\begin{proof}[Proof of \autoref{prop:rootsauto}]\label{appx:prop:rootsauto}
 First we should check that $\omega^*$ does not depend on the writing of the exponent of $x$. Indeed, consider $x^{\frac pq} = x^{\frac{pn}{qn}}$ and recall that $\omega_q^p = \omega_{nq}^{np}$ by the definition of a sequence of roots.
 \[\omega^*\left(x^{\frac pq}\right) = \omega_q^p x^{\frac pq} = \omega_{nq}^{np} x^{\frac{np}{nq}} = \omega^*\left(x^{\frac{np}{nq}}\right).\]
 The additive nature of $\omega^*$ is by construction. Since $|\omega_n| = 1$ for every $n$, it also follows that $\omega^*$ preserves norms. This continuity means we may prove the multiplicative property on finite Puiseux series, whence testing two monomials will do. Let $\alpha = sx^{\frac ab}$ and $\beta = tx^{\frac cd}$.
 \[\omega^*(\alpha\beta) = \omega^*\left(st\,x^{\frac{ad+bc}{bd}}\right) = st\, \omega_{bd}^{ad+bc}\, x^{\frac{ad+bc}{bd}}\]
 \[\omega^*(\alpha)\omega^*(\beta) = \left(s\,\omega_b^a\,x^{\frac ab}\right)\left(t\,\omega_d^c\,x^{\frac cd}\right) = st\, \omega_b^a\omega_d^c\, x^{\frac{ad+bc}{bd}}\]
 Now recall that for a sequence of roots $\omega_{mn}^m = \omega_n$ for every $m, n \in \N$. So 
 \[\omega_b^a\omega_d^c = \omega_{bd}^{ad}\omega_{bd}^{bc} = \omega_{bd}^{ad+bc},\] and this completes the claim that $\omega^*(\alpha\beta) = \omega^*(\alpha)\omega^*(\beta)$.
 It is clear that $(\omega^*)^{-1} = (\omega^{-1})^*$, provides an inverse homomorphism, and so $\omega^*$ is an isomorphism.
 
 Considering that $\omega^*(cx^n) = c\omega_1^nx^n$, then $\omega^*: \k((x)) \mapsto \k((x))$ is fully determined, so it is the unique isomorphism which maps $x$ to $\omega_1 x$. If $\omega$ is a sequence of roots of unity, then $\omega^*|_{\k((x))} = \id$.
 
 Now we prove that any element of the Galois group for $\K$ is indeed an action by some sequence of roots of unity. Let $\Psi \in G$. As usual, we can trivially extend the action of $\Psi$ to $\K[y]$. Consider $y^n - x \in \K[y]$. Then $\Psi(y^n - x) = y^n - x$. On the other hand \[y^n - x = \prod_{j=1}^n \left(y - \xi_n^j x^{\frac 1n}\right)\] where $\xi_n$ is some primitive $n$-th root of unity.
 Therefore \[\Psi\left(\prod_{j=1}^n (y - \xi_n^j x^{\frac 1n})\right) = \prod_{j=1}^n \left(y - \xi_n^j \Psi\left(x^{\frac 1n}\right)\right) = \prod_{j=1}^n \left(y - \xi_n^j x^{\frac 1n}\right).\]
 Since $\K[y]$ is a UFD, this shows that $\Psi\left(x^{\frac 1n}\right) = \xi_n^j x^{\frac 1n}$ for \emph{some} $j$.  Define $\omega_n = \xi_n^j$ and take a power of this equation, then we have shown $\Psi\left(x^{\frac mn}\right) = \omega_n^m x^{\frac mn}$. We may do this for every $n$. Let us now check that $(\omega_n)_{n=1}^\infty$ is a sequence of roots of unity. Indeed, $\Psi(x) = x \implies \omega_1 = 1$; $\Psi$ being multiplicative and well defined ensures that \[\omega_{q} x^{\frac{1}{q}} = \Psi\left(x^{\frac{1}{q}}\right) = \Psi\left(x^{\frac{p}{pq}}\right) = \Psi\left(x^{\frac{1}{pq}}\right)^p = \omega_{pq}^{p} x^{\frac{p}{pq}}.\]
 Finally, recall that `denomenators are bounded' for Puiseux series, therefore for any $\gamma \in \K$ there is an $N \in \N$ such that
 \[\gamma(x) = \sum_{j = 0}^{N-1} x^{\frac{j}{N}}\,g_j(x)\] where $g_j \in \k((x))$ for every $j$. Since $\Psi(g_j) = g_j = \omega^*(g_j)$ for every $j$, we have that \[\Psi(\gamma(x)) = \sum_{j = 0}^{N-1} \Psi\left(x^{\frac{j}{N}}\right)\,g_j(x) = \sum_{j = 0}^{N-1} \omega_N^j\,x^{\frac{j}{N}}\,g_j(x) = \sum_{j = 0}^{N-1} \omega^*\left(x^{\frac{j}{N}}\right)\,g_j(x) = \omega^*(\gamma(x)).\]
 Therefore $\omega^* = \Psi$, and we have shown that any Galois action is an action by a sequence of roots of unity.
\end{proof}

\begin{proof}[Proof of \autoref{prop:primghom}]\label{appx:prop:primghom}
 First we prove the case where $\lambda =1$ can be ignored. Write $g(x) = x^{\frac 1q}(1 + h(x))$ with $h \in (x) \triangleleft \k[[x]]$. Then we can write $g^*(x^r) = x^{\frac rq}(1 + h(x))^r$ where we expand the bracket with the binomial theorem, \autoref{lem:binom}. Note that $(1 + h(x))^r \in \k[[x]]$. Applying this to an element of the form \[\gamma(x) = \sum_{j = 1}^\infty a_j x^{r_j},\] we get
 \[g^*(\gamma(x)) = \sum_{j = 1}^\infty a_j x^{\frac {r_j}q}(1+h(x))^{r_j}.\] This is a new convergent series, since $|x^{\frac {r_j}q}(1+h(x))^{r_j}| = |x|^{\frac {r_j}q}$ and there are only finitely many coefficients of $x^p$ in the infinite sum for any given $p \in \Q$.
 
To finish showing that $g^*$ is well defined, we should check that $g^*$ does not depend on how we write the exponent of $x$. Indeed, \autoref{lem:binom} states that $(1+h(x))^{\frac ab} = (1+h(x))^{\frac{ma}{mb}}$ is unambiguous; thus
 \[g^*\left(x^{\frac ab}\right) = x^{\frac 1q\frac {a}b}(1+h(x))^{\frac ab} = x^{\frac 1q\frac {ma}{mb}}(1+h(x))^{\frac {ma}{mb}} = g^*\left(x^{\frac{ma}{mb}}\right).\]

As in the proof of \autoref{prop:rootsauto}, additivity is trivial and multiplicativity comes down to two monomials.
\[g^*(sx^{\frac ab}) g^*(tx^{\frac cd}) = \left(s\, x^{\frac 1q\frac {a}b}(1+h(x))^{\frac ab}\right)\left(t\,x^{\frac 1q\frac {c}d}(1+h(x))^{\frac cd}\right) \]\[= st\, x^{\frac 1q\frac{(ad+bc)}{bd}}(1+h(x))^{\frac{ad+bc}{bd}} = g^*\left(stx^{\frac{ad+bc}{bd}}\right)\] the second equality is again by \autoref{lem:binom}.

Suppose for the moment that $g\in \k((x))$. To show that $g^*$ is a homomorphism extending the one on $\k((x))$, let $m \in \N$ and observe that
\[g^*(x^m) = x^{mn}(1+h(x))^m = (x^n(1+h(x)))^m = g(x)^m.\]
The fact about the $m$th roots of $g(x)$ is clear once we write $g^*(x^{\frac 1m})^m = g^*((x^{\frac 1m})^m) = g^*(x) = g(x)$.
 
 Now assume that $q = 1$, $g(x) = x(1 + h(x))$ has an inverse (for example by the inverse function theorem), also of the form $x(1 + \tilde h(x)) = \tilde g(x)$. Now observe that
 \[\left(\tilde g^*(g^*(x^{\frac 1m}))\right)^m = \tilde g^*(g^*(x)) = \tilde g^*(g(x)) = g(\tilde g(x)) = x.\]
 Therefore $\tilde g^*(g^*(x^{\frac 1m}))$ is some $m$th root of $x$ in $\hk$. On the other hand
 \[\tilde g^*(g^*(x^{\frac 1m})) = \tilde g^*(x^{\frac 1m}(1+h(x))^{\frac 1m}) = x^{\frac 1m}(1+\tilde h(x))^{\frac 1m})(1+h(\tilde g(x)))^{\frac 1m} = x^{\frac 1m} + \lo(x^{\frac 1m})\] which shows that the coefficient of $x^{\frac 1m}$ is $1$ and so $\tilde g^*(g^*(x^{\frac 1m})) = x^{\frac 1m}$ for any $m$. Therefore $\tilde g^*\circ g^* = \id$; similarly $g^*\circ \tilde g^* = \id$, and so we have found an inverse.
 
 To show the isometry, we only need that $\abs{g^*(\gamma(x))} = \abs{\gamma(x)}\ \forall \gamma \in \hk$. Given an increasing sequence of rationals with $r_n \to \infty$, $a_n \in \k$ and $a_1 \ne 0$, define
\[\gamma(x) = \sum_{n = 1}^\infty a_n x^{r_n},\]
\[\abs{\gamma(x)} = \left |a_1 x^{r_1}\right| = |x|^{r_1}.\] Whereas for \[g^*(\gamma(x)) = \sum_{n = 1}^\infty a_n x^{r_n}(1 + h(x))^{r_n} = a_1 x^{r_1} + \cdots,\]
\[|g^*(a)| = \left |a_1 x^{r_1}\right| = |x|^{r_1}.\]

Finally we show that $g^*\omega^{*n} = \omega^*g^*$ for any sequence of roots of unity $\omega$ when $q = n$ is an integer. It is enough to show on the generators $x^{\frac ab}$. In this case $h(x) \in \k[[x]]$ and recall from \autoref{lem:binom} that $(1+h(x))^r$ is also in $\k[[x]]$. So \[\omega^*((1+h(x))^r) = (1 + h(x))^r.\]
\[\omega^*(g^*(x^{\frac ab})) = \omega^*\left(x^{\frac{na}b}(1+h(x))^{\frac ab}\right) = \omega^*\left(x^{\frac{na}b}\right)\,\omega^*\!\left((1+ h(x))^{\frac ab})\right) = \omega_b^{na}x^{\frac{na}b}(1+h(x))^{\frac ab}\]
\[g^*(\omega^{*n}(x^{\frac ab})) = g^*(\omega_b^{na}x^{\frac ab}) = \omega_b^{na} x^{\frac {na}b}(1+h(x))^{\frac ab}\]

To conclude, suppose $\lambda \ne 1$ and set $g^* = \tilde g^* \circ \lambda_\bullet^*$ where $\tilde g^*$ is the simpler version above. By the above, this is a well defined homomorphism. If $q = 1$ then both $\tilde g^*$ and $\lambda_\bullet^*$ are isometric isomorphisms which commute with $\omega^* \in G$.
\end{proof}

\begin{rmk}[Warning]
 $\lambda_\bullet^*$ and $\tilde g^*$ do not commute.
\end{rmk}

\begin{proof}[Proof of \autoref{prop:nearlyfunctorial}]\label{appx:prop:nearlyfunctorial}
 \[g_2^*\circ g_1^* = \tilde g_2^* \circ \mu_\bullet^* \circ \tilde g_1^* \circ \lambda_\bullet^*\]
 \[g_2^*\circ g_1^*(x^{\frac 1N}) = \tilde g_2^* \circ \mu_\bullet^* \left(\lambda_N x^{\frac mN}(1 + h_1(x))^{\frac1N}\right) = \tilde g_2^*\left(\mu_N^m\lambda_N x^{\frac mN}(1 + h_1(\mu x))^{\frac1N}\right)\]
 \[= \mu_N^m\lambda_N x^{\frac {mn}N}(1+ h_2(x))^{\frac mN}(1 + h_1(g(x)))^{\frac1N}\]
 Whereas $g_1 \circ g_2(x) = \lambda \mu^m x^{mn}(1 + h_3(x))$ and so
 \[(g_1 \circ g_2)^*(x^{\frac 1N}) = (\lambda \mu^m)_N x^{\frac {mn}N} (1 + h_3(x))^{\frac 1N}\] 
 \[2k\pi + \arg(\lambda \mu^m) = \arg(\lambda) + m\arg(\mu)\] and so it follows as it did in \autoref{prop:rootopsprim} that \[\mu_N^m\lambda_N = e^{2k\pi i /N}  (\lambda \mu^m)_N.\]
 So \[g_2^*\circ g_1^*(x^{\frac 1N}) = \prim^{k*}(g_1 \circ g_2)^*(x^{\frac 1N})\]
\end{proof}

 \bibliographystyle{alpha}
\bibliography{calcbib}

\end{document}